\documentclass[10pt]{amsart}
\usepackage{amsmath}
\usepackage{enumerate}
\usepackage{mathrsfs}
\usepackage{tikz}
\usepackage{amssymb}
\usepackage{amsfonts}
\usepackage{bm}
\usepackage{marvosym}
\usepackage{caption}
\usepackage{subcaption}
\usepackage{hyperref}
\usepackage[OT2,T1]{fontenc}
\usepackage[
textwidth=14.5cm, 
textheight=21cm,
hmarginratio=1:1,
vmarginratio=1:1]{geometry}
\DeclareSymbolFont{cyrletters}{OT2}{wncyr}{m}{n}
\DeclareMathSymbol{\Lfun}{\beta}{cyrletters}{"62}
\DeclareMathSymbol{\Rone}{\beta}{cyrletters}{"01}
\DeclareMathSymbol{\Rtwo}{\beta}{cyrletters}{"02}
\DeclareMathSymbol{\Rthree}{\beta}{cyrletters}{"03}
\DeclareMathSymbol{\Rfour}{\beta}{cyrletters}{"04}
\DeclareMathSymbol{\Rfive}{\beta}{cyrletters}{"05}
\DeclareMathSymbol{\Rsix}{\beta}{cyrletters}{"06}
\DeclareMathSymbol{\Rseven}{\beta}{cyrletters}{"07}
\DeclareMathSymbol{\Reight}{\beta}{cyrletters}{"08}
\DeclareMathSymbol{\Rnine}{\beta}{cyrletters}{"09}
\DeclareMathSymbol{\Rten}{\beta}{cyrletters}{"10}
\DeclareMathSymbol{\Releven}{\beta}{cyrletters}{"11}
\DeclareMathSymbol{\Rtwelve}{\beta}{cyrletters}{"12}
\DeclareMathSymbol{\Rthirteen}{\beta}{cyrletters}{"13}
\DeclareMathSymbol{\Rfourteen}{\beta}{cyrletters}{"14}
\DeclareMathSymbol{\Rfifteen}{\beta}{cyrletters}{"15}
\DeclareMathSymbol{\Rsixteen}{\beta}{cyrletters}{"16}
\DeclareMathSymbol{\Rseventeen}{\beta}{cyrletters}{"17}
\DeclareMathSymbol{\Reighteen}{\beta}{cyrletters}{"18}
\DeclareMathSymbol{\Rnineteen}{\beta}{cyrletters}{"19}
\DeclareMathSymbol{\Rtwenty}{\beta}{cyrletters}{"20}
\DeclareMathSymbol{\Rtwentyone}{\beta}{cyrletters}{"21}
\DeclareMathSymbol{\Rtwentytwo}{\beta}{cyrletters}{"22}
\DeclareMathSymbol{\Rtwentythree}{\beta}{cyrletters}{"23}
\DeclareMathSymbol{\Rtwentyfour}{\beta}{cyrletters}{"24}
\DeclareMathSymbol{\Rtwentyfive}{\beta}{cyrletters}{"25}
\DeclareMathSymbol{\Rtwentysix}{\beta}{cyrletters}{"26}
\DeclareMathSymbol{\Rtwentyseven}{\beta}{cyrletters}{"27}
\DeclareMathSymbol{\Rtwentyeight}{\beta}{cyrletters}{"28}
\DeclareMathSymbol{\Rtwentynine}{\beta}{cyrletters}{"29}
\DeclareMathSymbol{\Rthirty}{\beta}{cyrletters}{"30}
\DeclareMathSymbol{\Rthirtyone}{\beta}{cyrletters}{"31}
\DeclareMathSymbol{\Rthirtytwo}{\beta}{cyrletters}{"32}
\DeclareMathSymbol{\Rthirtythree}{\beta}{cyrletters}{"33}
\DeclareMathSymbol{\Rthirtyfour}{\beta}{cyrletters}{"34}
\DeclareMathSymbol{\Rthirtyfive}{\beta}{cyrletters}{"35}
\DeclareMathSymbol{\Rthirtysix}{\beta}{cyrletters}{"36}
\DeclareMathSymbol{\Rthirtyseven}{\beta}{cyrletters}{"37}
\DeclareMathSymbol{\Rthirtyeight}{\beta}{cyrletters}{"38}
\DeclareMathSymbol{\Rthirtynine}{\beta}{cyrletters}{"39}
\DeclareMathSymbol{\Rforty}{\beta}{cyrletters}{"40}
\DeclareMathSymbol{\Rfortyone}{\beta}{cyrletters}{"41}
\DeclareMathSymbol{\Rfortytwo}{\beta}{cyrletters}{"42}
\DeclareMathSymbol{\Rfortythree}{\beta}{cyrletters}{"43}
\DeclareMathSymbol{\Rfortyfour}{\beta}{cyrletters}{"44}
\DeclareMathSymbol{\Rfortyfive}{\beta}{cyrletters}{"45}
\DeclareMathSymbol{\Rfortysix}{\beta}{cyrletters}{"46}
\DeclareMathSymbol{\Rfortyseven}{\beta}{cyrletters}{"47}
\DeclareMathSymbol{\Rfortyeight}{\beta}{cyrletters}{"48}
\DeclareMathSymbol{\Rfortynine}{\beta}{cyrletters}{"49}
\DeclareMathSymbol{\Rfifty}{\beta}{cyrletters}{"50}
\DeclareMathSymbol{\Rfiftyone}{\beta}{cyrletters}{"51}
\DeclareMathSymbol{\Rfiftytwo}{\beta}{cyrletters}{"52}
\DeclareMathSymbol{\Rfiftythree}{\beta}{cyrletters}{"53}
\DeclareMathSymbol{\Rfiftyfour}{\beta}{cyrletters}{"54}
\DeclareMathSymbol{\Rfiftyfive}{\beta}{cyrletters}{"55}
\DeclareMathSymbol{\Rfiftysix}{\beta}{cyrletters}{"56}
\DeclareMathSymbol{\Rfiftyseven}{\beta}{cyrletters}{"57}
\DeclareMathSymbol{\Rfiftyeight}{\beta}{cyrletters}{"58}
\DeclareMathSymbol{\Rfiftynine}{\beta}{cyrletters}{"59}
\DeclareMathSymbol{\Rsixty}{\beta}{cyrletters}{"60}
\DeclareMathSymbol{\Rsixtyone}{\beta}{cyrletters}{"61}
\DeclareMathSymbol{\Rsixtytwo}{\beta}{cyrletters}{"62}
\DeclareMathSymbol{\Rsixtythree}{\beta}{cyrletters}{"63}
\DeclareMathSymbol{\Rsixtyfour}{\beta}{cyrletters}{"64}
\DeclareMathSymbol{\Rsixtyfive}{\beta}{cyrletters}{"65}
\DeclareMathSymbol{\Rsixtysix}{\beta}{cyrletters}{"66}
\DeclareMathSymbol{\Rsixtyseven}{\beta}{cyrletters}{"67}
\DeclareMathSymbol{\Rsixtyeight}{\beta}{cyrletters}{"68}
\DeclareMathSymbol{\Rsixtynine}{\beta}{cyrletters}{"69}
\DeclareMathSymbol{\Rseventy}{\beta}{cyrletters}{"70}
\DeclareMathSymbol{\Rseventyone}{\beta}{cyrletters}{"71}
\DeclareMathSymbol{\Rseventytwo}{\beta}{cyrletters}{"72}
\DeclareMathSymbol{\Rseventythree}{\beta}{cyrletters}{"73}
\DeclareMathSymbol{\Rseventyfour}{\beta}{cyrletters}{"74}
\DeclareMathSymbol{\Rseventyfive}{\beta}{cyrletters}{"75}
\DeclareMathSymbol{\Rseventysix}{\beta}{cyrletters}{"76}
\DeclareMathSymbol{\Rseventyseven}{\beta}{cyrletters}{"77}
\DeclareMathSymbol{\Rseventyeight}{\beta}{cyrletters}{"78}
\DeclareMathSymbol{\Rseventynine}{\beta}{cyrletters}{"79}
\DeclareMathSymbol{\Reighty}{\beta}{cyrletters}{"80}
\DeclareMathSymbol{\Reightyone}{\beta}{cyrletters}{"81}
\DeclareMathSymbol{\Reightytwo}{\beta}{cyrletters}{"82}
\DeclareMathSymbol{\Reightythree}{\beta}{cyrletters}{"83}
\DeclareMathSymbol{\Reightyfour}{\beta}{cyrletters}{"84}
\DeclareMathSymbol{\Reightyfive}{\beta}{cyrletters}{"85}
\DeclareMathSymbol{\Reightysix}{\beta}{cyrletters}{"86}
\DeclareMathSymbol{\Reightyseven}{\beta}{cyrletters}{"87}
\DeclareMathSymbol{\Reightyeight}{\beta}{cyrletters}{"88}
\DeclareMathSymbol{\Reightynine}{\beta}{cyrletters}{"89}
\DeclareMathSymbol{\Rninety}{\beta}{cyrletters}{"90}
\DeclareMathSymbol{\Rninetyone}{\beta}{cyrletters}{"91}
\DeclareMathSymbol{\Rninetytwo}{\beta}{cyrletters}{"92}
\DeclareMathSymbol{\Rninetythree}{\beta}{cyrletters}{"93}
\DeclareMathSymbol{\Rninetyfour}{\beta}{cyrletters}{"94}
\DeclareMathSymbol{\Rninetyfive}{\beta}{cyrletters}{"95}
\DeclareMathSymbol{\Rninetysix}{\beta}{cyrletters}{"96}
\DeclareMathSymbol{\Rninetyseven}{\beta}{cyrletters}{"97}
\DeclareMathSymbol{\Rninetyeight}{\beta}{cyrletters}{"98}
\DeclareMathSymbol{\Rninetynine}{\beta}{cyrletters}{"99}
\DeclareMathSymbol{\Rhundred}{\beta}{cyrletters}{"100}
\DeclareMathSymbol{\Rhundredone}{\beta}{cyrletters}{"101}
\DeclareMathSymbol{\Rhundredtwo}{\beta}{cyrletters}{"102}
\DeclareMathSymbol{\Rhundredthree}{\beta}{cyrletters}{"103}
\DeclareMathSymbol{\Rhundredfour}{\beta}{cyrletters}{"104}
\DeclareMathSymbol{\Rhundredfive}{\beta}{cyrletters}{"105}
\DeclareMathSymbol{\Rhundredsix}{\beta}{cyrletters}{"106}
\DeclareMathSymbol{\Rhundredseven}{\beta}{cyrletters}{"107}
\DeclareMathSymbol{\Rhundredeight}{\beta}{cyrletters}{"108}
\DeclareMathSymbol{\Rhundrednine}{\beta}{cyrletters}{"109}
\DeclareMathSymbol{\Rhundredten}{\beta}{cyrletters}{"110}
\DeclareMathSymbol{\Rhundredeleven}{\beta}{cyrletters}{"111}
\DeclareMathSymbol{\Rhundredtwelve}{\beta}{cyrletters}{"112}
\DeclareMathSymbol{\Rhundredthirteen}{\beta}{cyrletters}{"113}
\DeclareMathSymbol{\Rhundredfourteen}{\beta}{cyrletters}{"114}
\DeclareMathSymbol{\Rhundredfifteen}{\beta}{cyrletters}{"115}
\DeclareMathSymbol{\Rhundredsixteen}{\beta}{cyrletters}{"116}
\DeclareMathSymbol{\Rhundredseventeen}{\beta}{cyrletters}{"117}
\DeclareMathSymbol{\Rhundredeighteen}{\beta}{cyrletters}{"118}
\DeclareMathSymbol{\Rhundrednineteen}{\beta}{cyrletters}{"119}
\DeclareMathSymbol{\Rhundredtwenty}{\beta}{cyrletters}{"120}
\DeclareMathSymbol{\Rhundredtwentyone}{\beta}{cyrletters}{"121}
\DeclareMathSymbol{\Rhundredtwentytwo}{\beta}{cyrletters}{"122}
\DeclareMathSymbol{\Rhundredtwentythree}{\beta}{cyrletters}{"123}
\DeclareMathSymbol{\Rhundredtwentyfour}{\beta}{cyrletters}{"124}
\DeclareMathSymbol{\Rhundredtwentyfive}{\beta}{cyrletters}{"125}
\DeclareMathSymbol{\Rhundredtwentysix}{\beta}{cyrletters}{"126}
\DeclareMathSymbol{\Rhundredtwentyseven}{\beta}{cyrletters}{"127}
\DeclareMathSymbol{\Rhundredtwentyeight}{\beta}{cyrletters}{"128}
\DeclareMathSymbol{\Rhundredtwentynine}{\beta}{cyrletters}{"129}
\newcommand{\hrho}{\varrho}

\DeclareMathSymbol{\Rfun}{\beta}{cyrletters}{"17}

\DeclareFontFamily{U}{rcjhbltx}{}
\DeclareFontShape{U}{rcjhbltx}{m}{n}{<->rcjhbltx}{}
\DeclareSymbolFont{hebrewletters}{U}{rcjhbltx}{m}{n}

\let\aleph\relax\let\beth\relax
\let\gimel\relax\let\daleth\relax

\DeclareMathSymbol{\aleph}{\mathord}{hebrewletters}{39}
\DeclareMathSymbol{\beth}{\mathord}{hebrewletters}{98}
\DeclareMathSymbol{\gimel}{\mathord}{hebrewletters}{103}
\DeclareMathSymbol{\daleth}{\mathord}{hebrewletters}{100}

\DeclareMathSymbol{\lamed}{\mathord}{hebrewletters}{108}
\DeclareMathSymbol{\mem}{\mathord}{hebrewletters}{109}
\DeclareMathSymbol{\ayin}{\mathord}{hebrewletters}{96}
\DeclareMathSymbol{\tsadi}{\mathord}{hebrewletters}{118}
\DeclareMathSymbol{\qof}{\mathord}{hebrewletters}{113}
\DeclareMathSymbol{\shin}{\mathord}{hebrewletters}{152}
\DeclareMathSymbol{\memschloss}{\mathord}{hebrewletters}{77}
\DeclareMathSymbol{\nunlange}{\mathord}{hebrewletters}{78}
\DeclareMathSymbol{\vav}{\mathord}{hebrewletters}{79}
\DeclareMathSymbol{\tet}{\mathord}{hebrewletters}{84}
\DeclareMathSymbol{\tsadiklange}{\mathord}{hebrewletters}{90}
\DeclareMathSymbol{\He}{\mathord}{hebrewletters}{104}
\DeclareMathSymbol{\kaf}{\mathord}{hebrewletters}{107}
\DeclareMathSymbol{\nun}{\mathord}{hebrewletters}{110}
\DeclareMathSymbol{\pei}{\mathord}{hebrewletters}{112}
\DeclareMathSymbol{\resh}{\mathord}{hebrewletters}{114}
\DeclareMathSymbol{\samekh}{\mathord}{hebrewletters}{115}
\DeclareMathSymbol{\Het}{\mathord}{hebrewletters}{116}
\DeclareMathSymbol{\vav}{\mathord}{hebrewletters}{119}
\DeclareMathSymbol{\het}{\mathord}{hebrewletters}{120}
\DeclareMathSymbol{\yod}{\mathord}{hebrewletters}{121}
\DeclareMathSymbol{\zayin}{\mathord}{hebrewletters}{122}
\newtheorem{thm}{Theorem}[subsection]
\newtheorem{cor}[thm]{Corollary}
\newtheorem{lem}[thm]{Lemma}
\newtheorem{prop}[thm]{Proposition}

\newtheorem*{thm-expl}{Theorem~\ref{thm:main-coeff}, Explicit form}
\theoremstyle{definition}
\newtheorem{defn}[thm]{Definition}
\newtheorem{conj}[thm]{Conjecture}

\theoremstyle{remark}
\newtheorem{rem}[thm]{Remark}

\numberwithin{equation}{subsection}
\numberwithin{figure}{section}

\hypersetup{colorlinks=true,
linktoc=all,linkcolor=black,
citecolor=black}

\newcommand{\dbar}{\bar\partial}
\newcommand{\e}{\mathrm e}

\newcommand{\C}{{\mathbb C}}
\newcommand{\D}{{\mathbb D}}

\newcommand{\T}{{\mathbb T}}
\newcommand{\R}{{\mathbb R}}

\newcommand{\Z}{{\mathbb Z}}
\newcommand{\N}{{\mathbb N}}

\newcommand{\calS}{\mathcal{S}}

\newcommand{\dA}{\mathrm{dA}}
\newcommand{\checkQ}{\hat{Q}}

\newcommand{\Mop}{\mathbf{M}}
\newcommand{\Vop}{\mathbf{\Lambda}}
\newcommand{\indset}{\Het}
\newcommand{\indsett}{\tsadiklange}
\newcommand{\indsetS}{\samekh}
\newcommand{\indsetT}{\tet}

\newcommand{\calH}{{\mathcal H}}
\newcommand{\calD}{{\mathcal D}}
\newcommand{\calT}{{\mathcal T}}

\newcommand{\calB}{{\mathcal B}}
\newcommand{\calN}{{\mathcal N}}
\newcommand{\calF}{{\mathcal F}}

\newcommand{\calK}{{\mathcal K}}
\newcommand{\calQ}{{\mathcal Q}}

\newcommand{\calG}{{\mathcal G}}

\newcommand{\frakR}{\mathfrak{R}}

\newcommand{\frakr}{\mathfrak{r}}

\newcommand{\frakS}{\mathfrak{S}}
\newcommand{\frakG}{\mathfrak{G}}
\newcommand{\frakg}{\mathfrak{g}}
\newcommand{\frakH}{\mathfrak{H}}
\newcommand{\frakh}{\mathfrak{h}}
\newcommand{\frakJ}{\mathfrak{J}}
\newcommand{\frakT}{\mathfrak{T}}

\newcommand{\Kcorrker}{\mathrm{K}}
\newcommand{\kcorrker}{\mathrm{k}}

\newcommand{\POL}{\mathrm{POL}}

\newcommand{\re}{\operatorname{Re}}

\newcommand{\Lop}{{\mathbf L}}

\newcommand{\Pop}{{\mathbf P}}

\newcommand{\Hop}{{\mathbf H}}

\newcommand{\hDelta}{{\varDelta}}

\newcommand{\diff}{{\mathrm d}}
\newcommand{\diffs}{\mathrm{ds}}
\newcommand{\diffA}{\mathrm{dA}}
\newcommand{\imag}{{\mathrm i}}

\newcommand{\Ordo}{\mathrm{O}}
\newcommand{\ordo}{\mathrm{o}}
\renewcommand{\hm}{\varpi}
\renewcommand{\Re}{\operatorname{Re}}
\renewcommand{\Im}{\operatorname{Im}}
\newcommand{\logdens}{\Pi}

\begin{document}

\title[Asymptotics of planar orthogonal polynomials]
{Planar orthogonal polynomials and 
Boundary Universality in the Random Normal Matrix model}


\author[Hedenmalm]
{Haakan Hedenmalm}

\address{Hedenmalm: Department of Mathematics
\\
The Royal Institute of Technology
\\
S -- 100 44 Stockholm
\\
SWEDEN}

\email{haakanh@math.kth.se}

\author[Wennman]{Aron Wennman}

\address{Wennman: School of mathematical sciences
\\
Tel Aviv University
\\
Tel Aviv 69978
\\
ISRAEL
}

\email{aronwennman@tauex.tau.ac.il}



\date{\today}
\begin{abstract}
We show that the planar normalized orthogonal polynomials 
$P_{m,n}(z)$ of degree $n$ with respect to an 
exponentially varying planar measure $\e^{-2mQ}\diffA$ enjoy 
an asymptotic expansion
\[
P_{m,n}(z)\sim m^{\frac{1}{4}}\sqrt{\phi_\tau'(z)}[\phi_\tau(z)]^n
\e^{m\mathcal{Q}_\tau(z)}\left(\mathcal{B}_{\tau, 0}(z)
+m^{-1}\mathcal{B}_{\tau, 1}(z)+m^{-2}
\mathcal{B}_{\tau,2}(z)+\ldots\right),
\]
as $n,m\to\infty$ while the ratio $\tau=\frac{n}{m}$ is fixed. 
Here $\mathcal{S}_\tau$ denotes the droplet, the boundary of which is 
assumed to be a smooth simple closed curve, and  $\phi_\tau$ is a 
conformal mapping from the complement $\mathcal{S}_\tau^c$ to the exterior disk
$\D_\e$.
The functions $\mathcal{Q}_\tau$ and $\mathcal{B}_{\tau, j}$ are 
bounded holomorphic functions which 
may be expressed in terms of $Q$ and $\mathcal{S}_\tau$. 
We apply these results to obtain boundary universality in the random 
normal matrix model for smooth droplets, i.e., that the 
limiting rescaled process is the random process
with correlation kernel
$$
\kcorrker(\xi,\eta)=
\e^{\xi\bar\eta\,-\frac12(\lvert\xi\rvert^2+\lvert \eta\rvert^2)}
\,\mathrm{erf}\,(\xi+\bar{\eta}).
$$
A key ingredient in the proof of the asymptotic expansion of the orthogonal
polynomials is the construction of an {\em orthogonal foliation} -- a smooth 
flow of closed curves near $\partial\calS_\tau$, on each of 
which $P_{m,n}$ is appropriately orthogonal to lower order polynomials.
To compute the coefficient functions, we develop an algorithm which 
determines the coefficients $\calB_{\tau, j}$ successively  
in terms of inhomogeneous Toeplitz kernel conditions. 
These inhomogeneous Toeplitz kernel conditions may be understood 
in terms of scalar Riemann-Hilbert problems.
\end{abstract}

\maketitle

\addtolength{\textheight}{2.2cm}

\tableofcontents

\section{Introduction}
\subsection{Orthogonal polynomials}
\label{ss:onp-def}
We consider polynomials in one complex variable of the form
\begin{equation}\label{eq:pol-def}
P(z)= c_n z^n + c_{n-1} z^{n-1} + \ldots  + c_0,
\end{equation}
where $c_0, c_1,\ldots, c_n$ are complex numbers. 
If $c_n\ne0$, we say that $P$ has degree $n$, and call
$c_n$ is the {\em leading coefficient}.
We denote the $(n+1)$-dimensional space of all
polynomials of the form \eqref{eq:pol-def} by 
$\operatorname{Pol}_{n+1}$.
Given a positive Borel measure $\mu$ with infinite support 
on the complex plane $\C$, with finite moments
\begin{equation}\label{eq:moments-finite}
\int_\C |z|^{2k}\diff\mu(z)<\infty,\qquad 0\le k\le N,
\end{equation}
for some positive integer $N$, 
we define the {\em system} $\{P_{n}(z)\}_{n=0}^N$ 
{\em of normalized orthogonal polynomials (ONPs)}
with respect to $\mu$ recursively by applying the Gram-Schmidt algorithm 
to the sequence $\{z^n\}_{n=0}^N$ of monomials. 
Equivalently, the orthogonal polynomial $P_n$ is the unique element
in $\operatorname{Pol}_{n+1}$ of unit norm
in $L^2(\C,\mu)$ with positive 
leading coefficient $c_n>0$, such that 
for all lower degree polynomials $q\in\operatorname{Pol}_n$ we have
\[
\int_\C P_n(z)\overline{q(z)}\diff\mu(z)=0.
\]
When the measure $\mu=\mu_m$ 
depends on a parameter $m$, the orthogonal 
polynomials will be denoted by $P_{m,n}$,
where the first index is the parameter 
for the measure, and the second
is the degree of the polynomial.

For additional definitions and notation we refer the reader to
Subsection~\ref{ss:notation}.

\subsection{Carleman-\texorpdfstring{Szeg\H{o}}{Szego} asymptotics}
The 1920s witnessed a rapid development in the understanding of 
orthogonal polynomials and related kernel 
functions. Among the pioneers were Gabor Szeg\H{o}, 
Stefan Bergman and Torsten Carleman. One of the early results is
that of Szeg{\H{o}} \cite{Szeg1} (see also \cite{Szeg-book}), 
who considered the orthogonal
polynomials in $L^2(\Gamma,\diffs)$, where $\Gamma$ a real-analytically 
smooth Jordan curve in the complex plane $\C$ supplied with normalized arc 
length measure $\diffs=(2\pi)^{-1}|\diff z|$.
Let $\C\setminus\Gamma=\Omega\cup\Omega_\e$ be the decomposition
of the complement into disjoint connected components, where $\Omega$ is 
bounded and $\Omega_\e$ is unbounded, and denote by $\phi$ the 
conformal mapping of the exterior domain $\Omega_\e$ onto the exterior disk 
$\D_\e:=\{z\in\C:\,|z|>1\}$, which fixes the point at infinity
with positive derivative.
Szeg\H{o}'s theorem asserts that
\begin{equation}\label{eq:szego}
P_n(z)=\sqrt{\phi'(z)}[\phi(z)]^n
\left(1+\Ordo(\rho^{n})\right),\qquad z\in\Omega_\e, 
\end{equation}
where $\rho$ is some number with $0<\rho<1$.
Due to the real-analytically smooth boundary, 
the conformal mapping $\phi$ extends 
conformally past the boundary $\partial\Omega$. With the extended mapping
still denoted by $\phi$, the asymptotic formula \eqref{eq:szego} 
remains valid in a neighborhood of $\Omega_\e\cup\Gamma$. 

Slightly later, Carleman 
\cite{Carl1, Carl2} -- inspired by the work of Szeg{\H{o}} 
-- considered instead 
the orthogonal polynomials in 
$L^2(\Omega,\diffA)$, where 
$\diffA=(2\pi\imag)^{-1}\diff z\wedge\diff \bar{z}$ denotes the
normalized area 
element and $\Omega$ is a simply connected domain 
with real-analytic boundary curve $\Gamma$. 
He found an analogous asymptotic formula for the planar orthogonal 
polynomials, which holds in a neighborhood $\tilde{\Omega}_\e$ of the
closure of the 
exterior domain $\Omega_\e$ 
and is expressed in terms of the conformal mapping
$\phi$:
\begin{equation}\label{eq:Carleman}
P_n(z)=(n+1)^{\frac12}\,\phi'(z)[\phi(z)]^n\left(1+\Ordo(\rho^n)\right),
\qquad z\in \tilde{\Omega}_\e,
\end{equation}
for some $\rho$ with $0<\rho<1$. 
In the 1960s, Suetin extended Carleman's result
to domains whose boundary has a lower degree of smoothness,
as well as to weighted cases 
(see the monograph \cite{Suet}). 
We should also mention the more recent work of 
Dragnev and Mi\~na-Diaz
(\cite{Dragnev1}, \cite{Dragnev2}, and \cite{M-D1})
which strengthens Carleman's theorem on orthogonal polynomials, 
and gives information on the asymptotic distribution of the zeros.

In the above asymptotic formul\ae{} a Jacobian factor appears, 
it is $(\phi')^{\frac12}$ in the case of Szeg\H{o}'s
theorem and $\phi'$ in Carleman's case.
By inspection, the orthogonal polynomials
are asymptotically push-forwards of the monomials 
under the conformal mapping in the relevant $L^2$-space.

We wish to contrast the above-mentioned results 
with the more classical study of orthogonal polynomials on the real line
$\R$.
Here, the earliest work is associated with Legendre, Jacobi, Chebyshev, 
Hermite, Laguerre, and Gegenbauer, with
further contributions by Markov, Stieltjes, 
Szeg{\H{o}}, Bernstein, and Akhiezer.  
The structure of orthogonal polynomials on the line is rather rigid 
with the appearance of a three-term recursion relation, which comes from 
the fact that multiplication by the independent variable is self-adjoint 
on the weighted $L^2$-space. 
Analogous rigidity applies to the orthogonal polynomials 
on the unit circle $\T$ as well.
These facts are basic in many of the standard approaches 
to the asymptotics of orthogonal polynomials, 
see e.g. \cite{simonbook1, simonbook2}.
Going beyond measures supported on the line or the circle,
the rigidity is lost (except in some special cases, including
arc length measure on ellipses \cite{Duren}).
For planar orthogonal polynomials,
recursion formul\ae{} are rare, even if we allow any
finite number of terms \cite{Putinar}.

\subsection{Exponentially varying weights}
\label{ss:varying-weights}
For a $C^2$-smooth function $Q:\C\to \R\cup\{+\infty\}$ called the {\em potential},
subject to the growth bound
\begin{equation}\label{eq:Q-growth}
\liminf_{z\to\infty}\frac{Q(z)}{\log |z|}>1
\end{equation}
and a real parameter $m>0$, we consider the weighted area measures 
of the form
\begin{equation}\label{eq:exp-weight}
\diff\mu_{2mQ}(z)=\e^{-2mQ(z)}\diffA(z),\qquad z\in\C
\end{equation}
where we recall that $\diffA$ denotes the normalized planar area element. 
The condition \eqref{eq:Q-growth} guarantees that 
the measure $\mu=\mu_{2mQ}$ has finite moments \eqref{eq:moments-finite},
with upper range given by $N=N_m:=\lceil(1+\epsilon_1) m\rceil-2$ 
for some $\epsilon_1>0$. Here, $\lceil \cdot \rceil$ denotes the standard
ceiling function.
This allows us to consider the sequence $\{P_{m,n}\}_{0\le n\le N_m}$
of normalized orthogonal polynomials (ONPs)
with respect to the measure $\diff\mu_{2mQ}$ (cf.\ Subsection~\ref{ss:onp-def}).
Under certain additional assumptions on the regularity of the weight $Q$, 
we will obtain an asymptotic expansion of $P_{m,n}$
valid as $m$ and $n$ tend to infinity with the ratio 
$\tau=\tfrac{n}{m}$ confined to an open interval around $\tau=1$.

The motivation for studying this particular class of orthogonal polynomials 
comes from the theory of Random Normal Matrix (RNM)
ensembles, a particular instance of two-dimensional Coulomb gas.
If $m$ is a positive integer, 
the connection is that the eigenvalue 
process associated to an $m\times m$ matrix
from the RNM-ensemble with potential $Q$ 
is determinantal with correlation 
kernel $\Kcorrker_m$ given by
$$
\Kcorrker_m(z,w)=
K_m(z,w)\,\e^{-m\left(Q(z)+Q(w)\right)}\;\;\text{where}\;\quad
K_m(z,w)=\sum_{j=0}^{m-1}P_{m,j}(z)\overline{P_{m,j}(w)},
$$
see Subsection~\ref{ss:rnm} below for details.
Analogous families of exponentially varying 
weights confined to the real line 
appear in connection with the study of 
random Hermitian matrices.
In the 1980s, successive progress was 
made towards understanding the 
asymptotics of weighted ONPs on
the real line, with important contributions 
by Freud, Nevai, Lubinsky, 
Mhaskar, Saff, and Totik, to mention a 
few (see e.g.\ the monographs
\cite{LubinskySaff}, \cite{StahlTotik}, 
and \cite{Totik}).
A deeper understanding 
came through the efforts of Fokas, Its, Kitaev, and Deift and Zhou,
whose work brought novel methods into play.
Their approach analyzes the ONPs with respect to 
rather general potentials $Q$ on the real line
in terms of solutions to matrix Riemann-Hilbert problems, see, e.g., 
\cite{deiftbook, DeiftZhou, Its1, Its2}.

In the work \cite{Its} of Its and Takhtajan a natural 
{\em soft Riemann-Hilbert problem}, or 
matrix $\bar\partial$-problem, is considered, whose solution would give us
the orthogonal polynomial $P_{m,n}$ for the planar measure $\mu_{2mQ}$.
However, unlike the one-dimensional situation, 
it is not clear how to constructively 
solve these soft Riemann-Hilbert problems. 
The main obstruction appears to be the complex conjugation of the matrix,
which results from the sesquilinearity of the inner product.
While our analysis of the asymptotics of the ONPs is different,
we try to connect with the Its-Takhtajan approach 
later on in Section~\ref{s:RHP}.

\subsection{The boundary universality conjecture}
We return to the study of RNM-ensembles with the associated 
correlation kernel $\Kcorrker_m$.
Macroscopically, the situation is well understood. For instance, 
in the limit as $m\to+\infty$ the eigenvalues condensate to a certain 
compact set $\calS_1$, called the {\em droplet}, or alternatively
\emph{spectral droplet} (see Subsection~\ref{ss:rnm}
below). For simplicity, we assume below that $Q$ is $C^2$-smooth
with positive Laplacian $\hDelta Q>0$ in a neighborhood of $\calS_1$.
An interesting question is how the process behaves at the microscopic level,
which we express in rescaled coordinates as follows.
For a point $z_0\in\C$ with $\hDelta Q(z_0)>0$ and
a direction ${\rm n}\in\T$, we let
\begin{equation}\label{eq:resc1}
z_m(\xi)=z_0+{{\rm n}}\frac{\xi}{\sqrt{2m\hDelta Q(z_0)}}
\end{equation}
where $\hDelta_z=\partial_z\bar\partial_z$ denotes the (quarter) Laplacian,
and consider
\begin{equation}\label{eq:resc2}
\rho_m(\xi)=\frac{1}{2m\hDelta Q(z_0)}\Kcorrker_m(z_m(\xi), z_m(\xi)).
\end{equation}
We introduce the notation $\mathcal{E}^\circ$ for the interior and 
$\overline{\mathcal{E}}$ for the closure
of a subset $\mathcal{E}\subset\C$, 
while $\mathcal{E}^c=\C\setminus \mathcal{E}$ 
denotes the complement.
Near any bulk point $z_0$, i.e., a point in the interior 
$\in\calS_1^\circ$ of the droplet, there exists a full asymptotic 
expansion of the kernel $\Kcorrker_m$, see e.g.\ \cite{ahm2, ahm3}. 
In this case $\lim_m\rho_m(\xi)= 1$, uniformly on compact subsets. 
Away from the droplet, i.e.\ for 
$z_0\in\calS_1^c$ we instead have $\lim_m\rho_m(\xi)=0$.
It remains to analyze the boundary points $z_0\in\partial\calS_1$.
An illustration of this blow-up procedure for a boundary point 
in the context of RNM-ensembles is supplied in Figure~\ref{fig:resc}.

A natural simplifying assumption is that the boundary $\partial\calS_1$ 
is smooth near $z_0$, in which case we let ${\rm n}$ be 
the outer normal to $\calS_1$ at $z_0$. 
It is not known what is the limit of the density $\rho_m$,
but the following universal behavior is expected.

\begin{figure}[t!]
\begin{subfigure}[t]{.4\textwidth}
  \centering
  \includegraphics[width=1\linewidth]{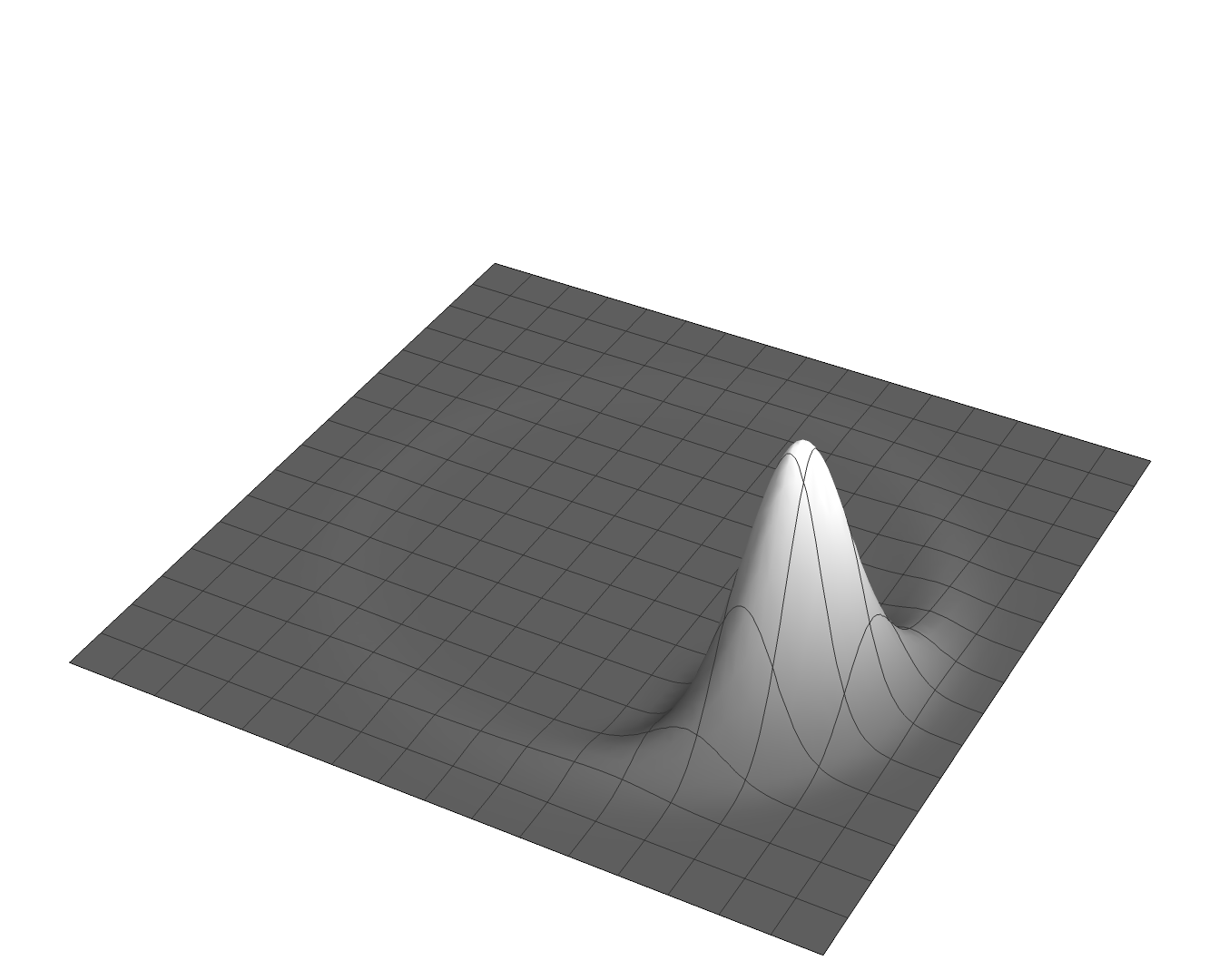}
\end{subfigure}
\begin{subfigure}[t]{.4\textwidth}
  \includegraphics[width=\linewidth]{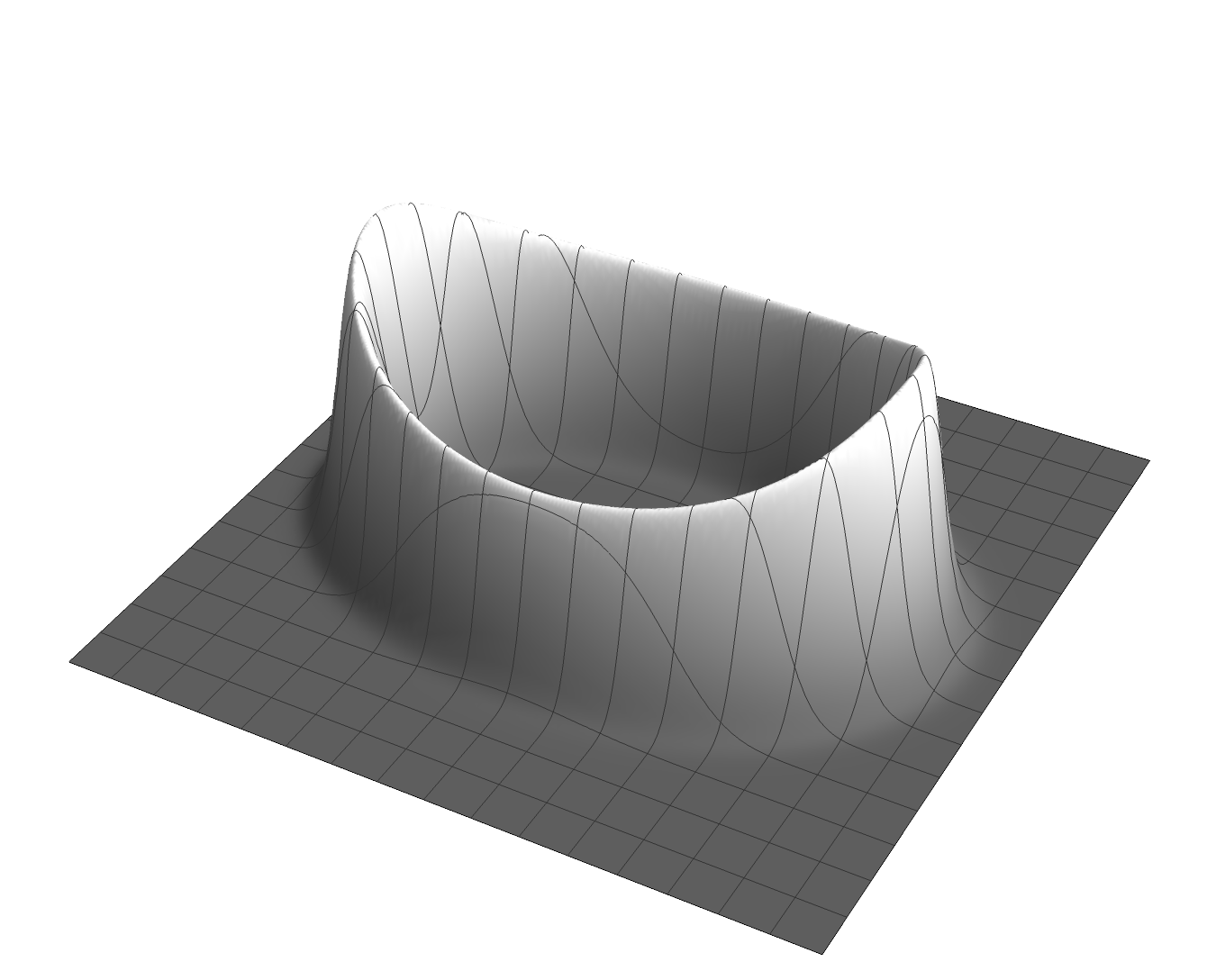}
\end{subfigure}
\caption{{\rm (left)}\, The Berezin density 
$\Kcorrker_m(z_0,z_0)^{-1}|\Kcorrker_m(z,z_0)|^2$ with 
$Q(z)=\frac12|z|^2$ for the boundary point $z_0=1$ and $m=30$. {\rm (right)}\,
The orthogonal polynomial density 
$\lvert P_{m,n}(z)\rvert^2 \e^{-2mQ(z)}$ for $n=25, m=20$ 
and $Q(z)=\tfrac12\vert z\vert^2-\Re(t z^2)$, where $t=0.2$.} 
\label{fig:Berezin-ONP}
\end{figure}

\begin{conj}[boundary universality]\label{conj:erf}
Let $z_0\in\partial\calS_1$ and assume that $\partial\calS_1$ is 
smooth in a neighborhood of $z_0$. 
Then the density $\rho_m$ converges as $m\to\infty$ to the limit
\begin{equation*}
\rho(\xi)=\mathrm{erf}\,(2\Re \xi).
\end{equation*}
\end{conj}
Here, we write $\mathrm{erf}$ for the complex \emph{error function}
\[
\mathrm{erf}\,(z)=\frac{1}{\sqrt{2\pi}}\int_{z}^{\infty}\e^{-{t^2}/{2}}\diff t,
\]
where the integral is taken along a suitable contour from $z$ to 
the origin and then from
the origin to $\infty$ along the positive real line. 
This conjecture has been verified in some specific cases, and partial 
results have appeared recently. In connection with this we
want to mention the work by Ameur, Kang, and Makarov \cite{akm}
who used a limiting form of the Ward identities to show that if $\rho(\xi)$ 
is a priori known to only depend on $\Re\,\xi$, 
then it must necessarily be of the form predicted by Conjecture~\ref{conj:erf}.
The full conjecture however remains open.
In the setting of K\"ahler manifolds, a similar problem appears in the context
of partial Bergman kernels defined by vanishing to high order along a divisor. 
Under the assumption of $S^1$-invariance around the divisor, Ross and Singer 
\cite{RS} obtain the error function asymptotics near the emergent interface
around the divisor (see also the work of Zelditch and Zhou \cite{ZZ1}).
In recent work, Zelditch and Zhou \cite{ZZ2} find that this is a universal
edge phenomenon along interfaces in the context of partial 
Bergman kernels defined by a quantized Hamiltonian.

\begin{figure}
\includegraphics[width=.6\textwidth]{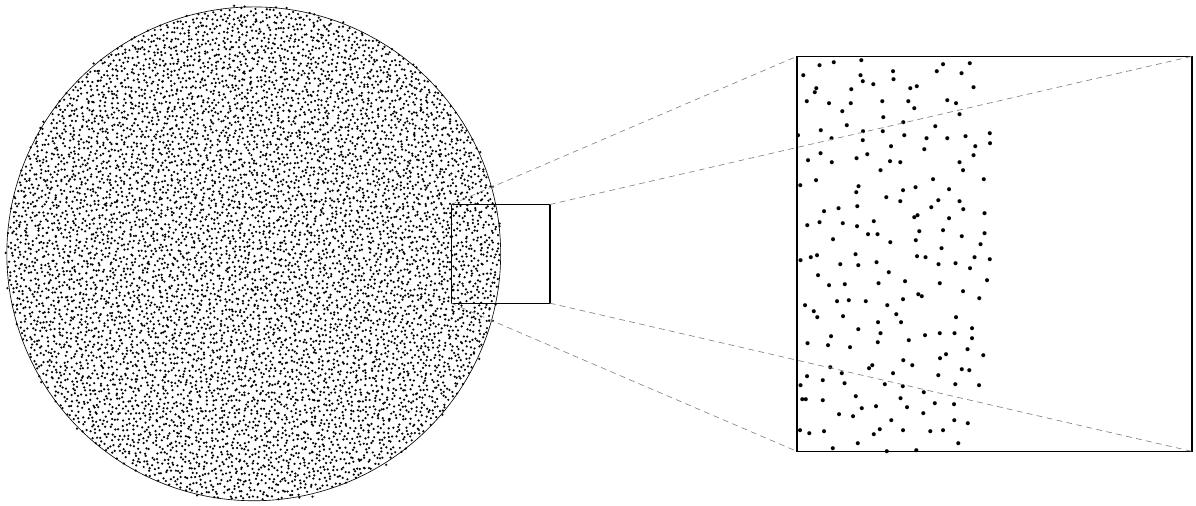}
\caption{The RNM process associated to a quadratic potential 
(The Ginibre ensemble) with
blow-up at a boundary point (courtesy of Nam-Gyu Kang).}
\label{fig:resc}
\end{figure}

Let us briefly motivate why the interface asymptotics for the RNM-ensembles
should be approached via the orthogonal polynomials.
The standard methods to obtain the asymptotics of Bergman kernels are local
in nature, both the peak section approach of Tian (see \cite{Tian})
as well as the microlocal approach of Boutet de Monvel and Sj\"ostrand, 
as explained by Berman, Berndtsson, and Sj\"ostrand \cite{BBS} 
(see also \cite{haimihed}). 
The same applies to older work of H\"ormander \cite{HormActa65} 
and Fefferman \cite{Fef}. 
One reason to expect the boundary universality conjecture to be difficult
is the apparent nonlocality of the correlation kernel. 
To illustrate this, we
consider the Berezin density (associated with secondary quantization 
and complementary to the Palm measure, cf. \cite{BFQ})
\[
B^{\langle z_0\rangle}_m(z)=K_m(z_0,z_0)^{-1}|K_m(z,z_0)|^2\e^{-2mQ(z)}
\]
considered in \cite{ahm1} and find numerically that for boundary points 
$z_0\in\partial \calS_1$, this probability 
density develops a noticeable ridge along the whole 
boundary of the spectral droplet (see Figure \ref{fig:Berezin-ONP} (left)). 
For this reason we focus our analysis on the orthogonal 
polynomials, which have an even more pronounced 
nonlocal behavior (see Figure \ref{fig:Berezin-ONP} (right)).  
Indeed, for rather general potentials $Q$, the mean field 
approximation of the random normal matrix model 
\cite{ahm2, ahm3} supplies information regarding the individual 
orthogonal polynomials, and gives the weak-star convergence of measures 
$$
\lvert P_{m,n}\rvert^2 \e^{-2mQ}\to
\hm(\cdot,\widehat{\C}\setminus\calS_1, \infty),
$$
as $n,m\to\infty$ with $n=m+\Ordo(1)$.
Here, the left-hand side is the density of a probability measure, and the 
right-hand side expression
$\hm(\cdot,\widehat{\C}\setminus\calS_1,\infty)$ stands for harmonic 
measure of the domain $\widehat{\C}\setminus\calS_1$ evaluated at the 
point at infinity, which has the interpretation of hitting probability of
Brownian motion starting at $\infty$. We observe that harmonic
measure is concentrated to the boundary, so that the above convergence 
may be interpreted as {\em boundary concentration}.
Within the random normal matrix model, the addition of a new particle 
has the net effect of adding a term 
$\lvert P_{m,n}\rvert^2 \e^{-2mQ}$ of highest degree. This means that the 
net effect of adding a particle is felt primarily along the droplet boundary.
More generally,
when $\tau=\frac{n}{m}$ is kept fixed, the concentration occurs along the 
boundary $\partial\calS_\tau$ of another
spectral
droplet $\calS_\tau$.

Finally, we mention that the orthogonal polynomial approach has proven 
to be successful in several special cases.
For instance, when $Q(z)=\frac{1}{2}\lvert z\rvert^2+a\Re(z^2)$ 
with $a>0$, Lee and Riser \cite{leeriser} obtain the orthogonal polynomials
in explicit form, and verify Conjecture~\ref{conj:erf} in this 
case.
Along the same lines, in \cite{bblm}, Balogh, Bertola, Lee and McLaughlin 
consider potentials $Q$ which are perturbations of the 
standard quadratic potential of the form
$$
Q(z)=\frac{1}{2}\lvert z\rvert^2-c\log\lvert z-a\rvert^2,
$$
for some $a\in\R$, $c>0$. 
For this $Q$, they obtain an asymptotic expansion of the orthogonal 
polynomials.
For parameters $a$ and $c$ such that the droplet $\calS_\tau$ does not 
divide the plane, 
the expansion is expressed in terms of the properly 
normalized conformal mapping of the complement $\calS_\tau^c$ 
onto the exterior disk $\D_e$, denoted $\phi_\tau$. 
After some rewriting, their formula reads 
\begin{equation}\label{eq:BBLM}
P_{m,n}(z) =\left(\frac{m}{2\pi}\right)^{1/4}
\sqrt{\phi'_\tau(z)}[\phi_\tau(z)]^n\e^{m\mathcal{Q}_\tau(z)}
\left(1+\Ordo(m^{-1})\right),
\end{equation}
valid in a neighborhood of the closed exterior of the droplet for 
$\frac{n}{m}=\tau+\Ordo(m^{-1})$, where $\mathcal{Q}_\tau$ is the bounded
holomorphic 
function on $\calS_\tau^c$, with real part equal to $Q$ on 
the boundary $\partial\mathcal{S}_\tau$, extended analytically 
across the boundary. Using the asymptotics \eqref{eq:BBLM}, they 
verify Conjecture~\ref{conj:erf} for the given collection of potentials. 
The analysis in \cite{bblm} is based on 
Riemann-Hilbert problem methods, which are accessible due to a 
miraculous identity which transforms the Hermitian orthogonality 
over the plane into bilinear orthogonality relations along curves. 
The latter approach should be compared with
the work of Bleher and Kuijlaars \cite{BleherKuijlaars} 
in the context of a cubic potential.
 
At the physical level, it is understood that the asymptotic formula 
\eqref{eq:BBLM} should hold for the wider class of potentials 
of the form $Q(z)=\frac{1}{2}\lvert z\rvert^2+H(z)$, where 
$H$ is harmonic in a neighborhood of the droplet
(the so-called Hele-Shaw potentials) \cite{WZ-conj}. What the higher order 
correction 
terms should look like appears not to be understood even then.

\subsection{Summary of the results}
\label{ss:contrib}
Here, we study the orthogonal polynomials 
with respect to a rather general exponentially varying weight $\e^{-2mQ}$ in
the complex plane. 
To be more precise, we will work with 
potentials $Q$ that are admissible in the sense of the definition below. 
Under $C^2$-smoothness and a minimal growth assumption on $Q$, 
we consider for $\tau>0$ the coincidence set
$$
\calS_\tau^\star:=\{z\in\C\,:\:\checkQ_\tau(z)=Q(z)\},
$$ 
where $\checkQ_\tau$ solves the 
obstacle problem
$$
\checkQ_\tau(z)=\sup\left\{q(z): q\in \operatorname{Subh}_\tau(\C),
\,\, q\leq Q\;\text{on}\;\C\right\}.
$$
Here, $\operatorname{Subh}_\tau(\C)$ denotes the convex body of subharmonic
functions in the plane which grow 
at most like $\tau\log\lvert z\rvert$ at infinity.
The function $\checkQ_\tau$ is $C^{1,1}$-smooth and harmonic outside the
set $\calS_\tau^\star$. Moreover, if $Q$ 
has sufficient growth, $\calS_\tau^\star$ is compact.
For a subset $\mathcal{E}\subset\C$ we write $1_{\mathcal{E}}$ for the
corresponding indicator function.
The support of the probability measure $\mu_\tau$ given by 
\[
  \diff\mu_\tau=2\tau^{-1}1_{\calS_\tau^\star}\hDelta Q\diffA
\] 
is denoted by $\calS_\tau$ and called
the {\em droplet}. Clearly, $\calS_\tau\subset\calS_\tau^\star$, and
$\calS_\tau^\star\setminus\calS_\tau$ is a null-set for the measure
$|\Delta Q|\diffA$. We note that $\mu_\tau$ is the equilibrium measure
for the weighted logarithmic energy problem in the external field $\tau^{-1}Q$.
More details are supplied in Subsection~\ref{ss:prel-pot} below.

\begin{figure}
\centering
\includegraphics[width=.31\linewidth]{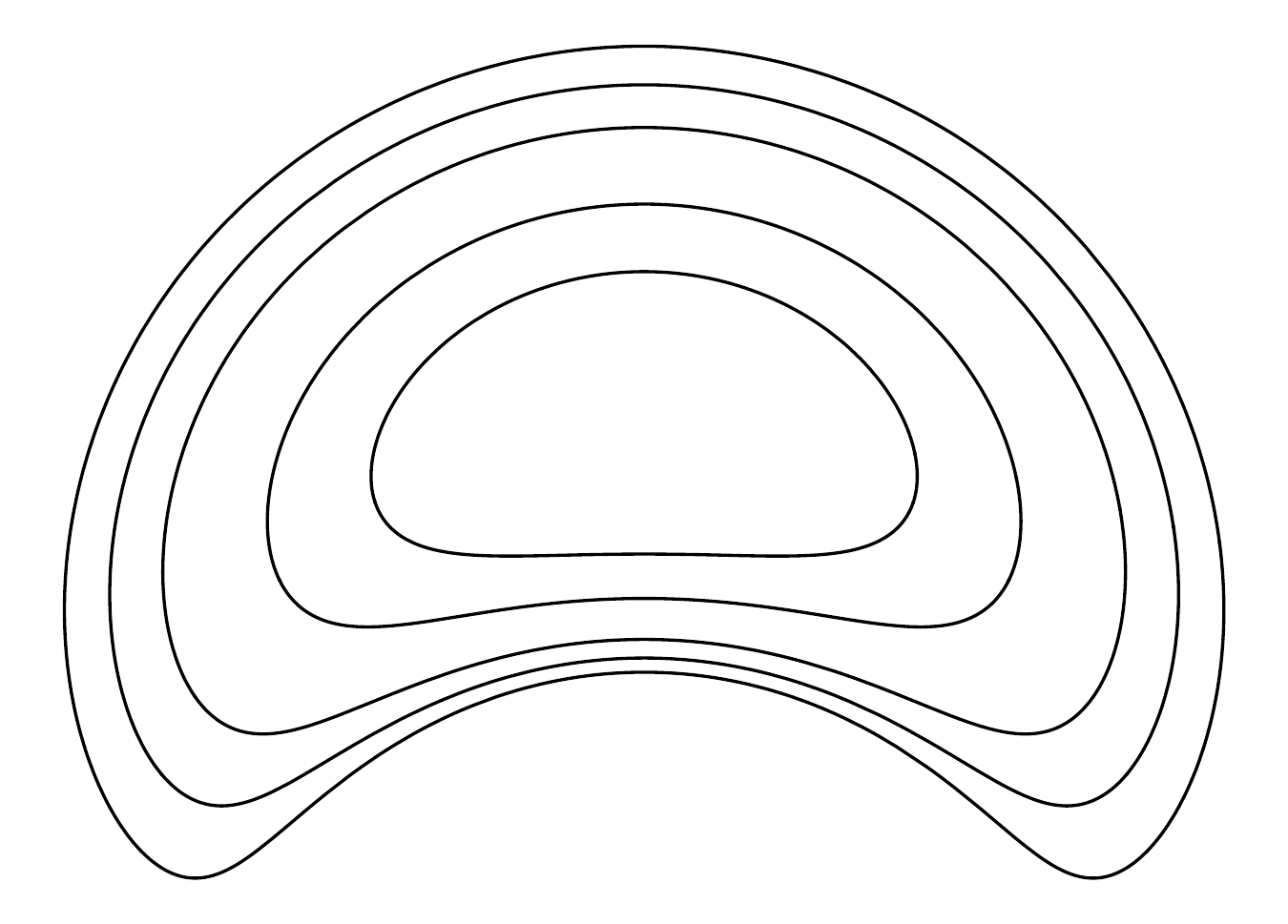}
\caption{Laplacian growth of the compacts $\calS_{\tau}$ 
for the potential $Q(z)=\frac12|z|^2-2^{-\frac12}\log|z+\imag|$ 
(boundary curves indicated).}
\label{fig:HS-flow}
\end{figure}

\begin{defn}\label{def:adm}
The potential $Q:\C\to\R$ is said to be $\tau$-{\em admissible}
at $\tau=\tau_0$ (or, in short, $\tau_0$-admissible) if 
$\calS_{\tau_0}=\calS_{\tau_0}^\ast$ and 
the following conditions are satisfied:
\begin{enumerate}[(i)]
\item $Q$ is $C^2$-smooth in the entire complex plane,
\item $Q$ is real-analytic and strictly subharmonic (i.e.\ $\hDelta Q>0$)
in a neighborhood of $\calS_{\tau_0}$,  
\item $Q$ is grows sufficiently fast at infinity:
\begin{equation}
\liminf_{|z|\to+\infty}\frac{Q(z)}{\log|z|}>\tau_0.
\label{eq-Qcond}
\end{equation}
\item The boundary $\partial\calS_{\tau_0}$ is a smooth Jordan curve. 
\end{enumerate}
\end{defn}
Note that under these conditions, it follows that $\checkQ_{\tau_0}(z)<Q(z)$
on $\calS_{\tau_0}^c$.
As a consequence, we may exclude the immediate birth of additional components of 
$\calS_\tau$ as $\tau$ increases from $\tau_0$.

In the sequel, we consider $\tau_0=1$, and 
{\em assume that $Q$ is $\tau$-admissible} at $\tau=1$.
As observed in Subsection~\ref{ss:varying-weights}, 
the condition \eqref{eq-Qcond} with $\tau=1$ guarantees that all 
polynomials of degree up to $\lceil(1+\epsilon_1)m\rceil-2$ 
belong to the space $L^2(\C,\e^{-2mQ}\diffA)$, 
for some fixed small $\epsilon_1>0$.
As $Q$ is assumed $1$-admissible, the curve $\partial\calS_1$ is smooth,
simple and closed. By known properties of Laplacian growth, this 
assumption implies that the same holds for the boundaries $\partial\calS_\tau$ for 
$\tau\in I_{\epsilon_0}:=[1-\epsilon_0, 1+\epsilon_0]$ for some $\epsilon_0>0$
(cf.\ \cite{HS, HM}).
By considering a smaller $\epsilon_0$, we can make sure
this property holds on the larger interval $I_{2\epsilon_0}$ as well.
Moreover,
the assumption of $1$-admissibility entails that the smooth curves 
$\partial\calS_\tau$ are actually real-analytically smooth for 
$\tau\in I_{\epsilon_0}$. 
This follows from the work of Sakai \cite{Sakai} on 
boundaries with a one-sided Schwarz function, 
as observed in \cite{HS}.

We now proceed to present our main theorem.
To set things up, we denote for $\tau\in I_{\epsilon_0}$ 
by $\phi_\tau$ the conformal mapping
$\phi_\tau:\calS_\tau^c\to\D_\e$,
normalized by
$\phi_\tau(\infty)=\infty$ and $\phi_\tau'(\infty)>0$.
As a consequence of $1$-admissibility, $\phi_\tau$ extends to a conformal
mapping $\calK_{\tau, 0}^c\to\D_\e(0,\rho_{0,0})$, where $0<\rho_{0,0}<1$ and 
$\calK_{\tau, 0}\subset\calS_\tau$ denotes 
an appropriate compact continuum. Here,
we use the notation $\D_\e(0,r):=\{z\in\C:|z|>r\}$ for the
exterior disk of radius $r$ centered at the origin.
We let $\mathcal{Q}_\tau$ denote the bounded holomorphic function on 
$\mathcal{S}^c_\tau$ whose real part equals the potential $Q$ 
along the boundary $\partial\mathcal{S}_\tau$, 
and whose imaginary part vanishes at infinity.
By possibly adjusting $\rho_{0,0}$, we may ensure that $\calQ_\tau$
extends holomorphically to $\calK_{\tau,0}^c$.

For a subset $\mathcal{E}\subset\C$, we use the notation
$\mathrm{dist}_\C(z,\mathcal{E})=\inf_{w\in\mathcal{E}}|z-w|$
for the Euclidean distance from $z$ to the set $\mathcal{E}$.

\begin{thm}\label{thm:main-pw}
Assume that $Q$ is $1$-admissible. Given a positive integer $\kappa$ 
there exist bounded holomorphic functions $\calB_{\tau,j}$ 
defined in a fixed neighborhood of $\calS_\tau^c$
such that for any positive real $A$, the asymptotic formula
$$
P_{m,n}(z)= 
m^{\frac14}[\phi_\tau'(z)]^{\frac12}[\phi_\tau(z)]^n\e^{m\mathcal{Q}_\tau(z)}
\bigg(\sum_{j=0}^{\kappa}m^{-j}\mathcal{B}_{\tau,j}(z)
+\Ordo\left(m^{-\kappa-1}\right)\bigg),
$$
holds, where the error term is uniform over all $z\in\C$ with 
$\mathrm{dist}_\C(z,\calS_\tau^c)\le A(m^{-1}\log m)^{\frac12}$ as 
$n=\tau m\to +\infty$ with $\tau\in I_{\epsilon_0}$.
\end{thm}

In other words, the orthogonal polynomials $P_{m,n}$ enjoy an 
asymptotic expansion
$$
P_{m,n}(z)\sim 
m^{\frac14}
[\phi_\tau'(z)]^{\frac12}[\phi_\tau(z)]^n\e^{m\mathcal{Q}_\tau(z)}
\Big(\mathcal{B}_{\tau,0}(z)+
\tfrac{1}{m}\mathcal{B}_{\tau,1}(z)+\ldots\Big), 
$$
valid provided that $\mathrm{dist}_\C(z,\calS_\tau^c)\le 
A(m^{-1}\log m)^{\frac12}$
as $n=\tau m\to+\infty$ and $\tau\in I_{\epsilon_0}$, 
for any given $A>0$.

\begin{rem}\label{rem:control-outside}
{\rm (a)}\, We derive Theorem~\ref{thm:main-pw} from an $L^2$-version of the
asymptotic expansion, given in Theorem~\ref{main} below.
An advantage of the $L^2$-version is that it holds in
a fixed $\epsilon$-neighborhood of the exterior $\calS_\tau^c$.

\noindent {\rm (b)}\, It is curious to note that the expansion of $P_{m,n}$
contains the factor $(\phi_\tau')^{\frac12}$, rather than $\phi_\tau'$ as one
might expect from Carleman's theorem. The square root is more reminiscent
of Szeg\H{o}'s theorem. We have no satisfactory explanation for this fact, other than
appealing to heuristics based on the steepest descent method.
Naturally, the expansion could be written with $\phi_\tau'$ as a factor, 
by adjusting the terms $\calB_{\tau,j}$ accordingly. However,
the term $\calB_{\tau,0}$ takes on the simplest possible form with the former
choice, as shown in Theorem~\ref{thm:main-coeff}.
\end{rem}

In the context of Theorem \ref{thm:main-pw}, we would like 
to know the coefficient functions $\calB_{\tau, j}$. 
How to do find them is explained
in the following theorem. For the formulation, we need the Szeg\H{o} 
projection $\Pop_{H^2_{-,0}}$ of $L^2(\T)$ onto the conjugate Hardy space 
$H^2_{-,0}=L^2(\T)\ominus H^2$ (cf. Subsection \ref{ss:herglotz} below).
In addition, we need the \emph{modified weight} $R_\tau$ defined in a 
neighborhood of $\bar{\D}_\e$ by
\begin{equation}\label{eq:Rtau}
R_\tau=(Q-\breve{Q}_\tau)\circ\phi_\tau^{-1},
\end{equation}
where we need to explain what is the function $\breve{Q}_\tau$. The solution 
$\hat{Q}_\tau$ to the obstacle problem is a $C^{1,1}$-smooth function which 
equals $Q$ on $\calS_\tau$, while it is strictly smaller and harmonic in
the exterior $\calS_\tau^c$. As a consequence of the smoothness of $Q$
and the boundary curve 
$\partial\calS_\tau$, the restriction $\hat{Q}_\tau|_{\calS_\tau^c}$ to the exterior
extends harmonically across the boundary for each $\tau\in I_{\epsilon_0}$. 
We denote the extended function by $\breve{Q}_\tau$.

\begin{thm}\label{thm:main-coeff}
In the asymptotic expansion of Theorem~\ref{thm:main-pw}, we have 
$\calB_{\tau, 0}=\pi^{-\frac14}
\e^{H_{Q,\tau}}$, where $H_{Q,\tau}$ is bounded and holomorphic 
in $\calS_\tau^c$ and satisfies $\Im H_{Q,\tau}(\infty)=0$, as well as
$$
\Re H_{Q,\tau}=\frac{1}{4}\log\hDelta Q,\quad\text{on}\quad \partial\calS_\tau.
$$
Moreover, if $H_{R_\tau}$ denotes the bounded holomorphic function on $\D_\e$
with
$$
\Re H_{R_\tau}=\frac{1}{4}\log(4\hDelta R_\tau)\quad\text{on}\,\,\,\T,
$$
and $\Im H_{R_\tau}(\infty)=0$,
then for $j=1,2,3,\ldots$, the coefficients 
$\calB_{\tau, j}$ have the form
$$
\calB_{\tau, j}=[\phi_\tau']^{\frac12}\,
B_{\tau, j}\circ\phi_\tau,
$$
where the functions $B_{\tau, j}$ are bounded and holomorphic in 
$\D_\e$, and given by
\[
B_{\tau, j}=c_{\tau, j}\,\e^{H_{R_\tau}}-\e^{H_{R_\tau}}
\Pop_{H^2_{-,0}}[\e^{{\bar H}_{R_\tau}}F_{\tau, j}]
\]
for some real-analytic functions $F_{\tau, j}$ on the circle $\T$ and 
constants $c_{\tau, j}\in\R$. The functions $F_{\tau, j}$ as well as the 
constants $c_{\tau, j}$ may be computed algorithmically in terms of
the potential $R_\tau$ and the functions $B_{\tau, 0},\ldots, B_{\tau, j-1}$, 
where $B_{\tau, 0}=(4\pi)^{-\frac14}\e^{H_{R_\tau}}$.
\end{thm}

\begin{rem}
  \label{rem:hol-ext-comput}
  {\rm (a)}\; In the above theorem, all the functions $\calB_{\tau, j}$, $B_{\tau, j}$ 
  as well as $H_{Q,\tau}$ and $H_{R_\tau}$ 
  extend holomorphically across their respective boundaries.

\noindent {\rm (b)}\; The functions 
$H_{Q,\tau}$ and $H_{R_\tau}$ 
are related by
\[
H_{R_\tau}\circ\phi_\tau=\frac12\log(2 \phi_\tau')+ H_{Q,\tau}. 
\]

\noindent {\rm (c)}\; We point out that 
Theorems \ref{thm:main-pw} and \ref{thm:main-coeff} together
imply that for large enough $m$, and for $\tau=\frac{n}{m}
\in I_{\epsilon_0}$, all the zeros of the polynomial $P_{m,n}(z)$
lie inside $\calS_\tau$, and stay away from the boundary curve
$\partial\calS_\tau$ by 
a distance of at least $A(m^{-1}\log m)^{\frac12}$.
\end{rem}
While Theorem \ref{thm:main-coeff} gives the asymptotic structure of the
orthogonal polynomials, it remains to specify how to algorithmically 
obtain the real-analytic functions $F_{\tau, j}$ and 
the constants $c_{\tau, j}$, for $j=1,2,3,\ldots$.
For $k=0,1,2,\ldots$, let $\Lop_k$ be the differential operator given by
\begin{equation}\label{eq:Lop}
\Lop_k[f]=\sum_{\nu=k}^{3k}\frac{(-1)^{\nu-k}2^{-\nu}}{\nu!(\nu-k)!
[\partial_r^2R_{\tau}(re^{\imag\theta})]^\nu}
\partial^{2\nu}_r\Big(\Big[R_\tau-\frac{1}{2}(r-1)^2
\partial_r^2R_\tau(e^{\imag\theta})\Big]^{\nu-k}f(r\e^{\imag \theta})\Big).
\end{equation}
This is a differential operator of order $6k$, acting on a smooth function $f$ 
defined in a neighborhood of the unit circle. 
We are specifically interested in the restriction 
$\Lop_k[f](r\e^{\imag\theta})\big\vert_{r=1}$, which expression only 
involves derivatives of order at most $2k$. The operator $\Lop_k$
results from the asymptotic analysis of definite integrals using Laplace's method, 
as in Proposition \ref{prop:steep} below.
Later on, in Lemma~\ref{lem:pseudo}, we show the existence of 
differential operators $\Mop_{k}$ with the property that
\[
\int_{\T}\e^{\imag l\theta} \big(\partial^2_r R_\tau(r\e^{\imag\theta})\big)^{-\frac12}
\Lop_{k}[r^{1-l}f(r \e^{\imag\theta})]\bigg\vert_{r=1}
\diff\theta=\int_{\T}\e^{\imag l\theta}\Mop_k [f](\e^{\imag\theta})]
\diff\theta,
\]
for $l=1,2,3,\ldots$. 
We use these operators to rid the left-hand side of any unwanted dependence 
on the parameter $l$.
In terms of the operators $\Lop_k$ and $\Mop_k$, we may now express 
$F_{\tau, j}$ and $c_{\tau, j}$ as follows:
\begin{equation}\label{eq:Fj}
  F_{\tau, j}(\theta)=\sum_{k=1}^{j}\Mop_k[B_{\tau, j-k}](\e^{\imag\theta}),
  \qquad j\geq 1,
\end{equation}
and the real constants $c_{\tau, j}$ are given by 
$c_{\tau, 0}=(4\pi)^{-1/4}$ while for $j=1,2,3,\ldots$,
\begin{equation}\label{eq:cj}
c_{\tau, j}=-\frac12(4\pi)^{\frac14}\sum_{(i,k,l)\in\indset_j}
\int_{\T}
\Mop_k\big[B_{\tau, i}\bar{B}_{\tau, l}\big](\e^{\imag\theta})
\diffs(\e^{\imag\theta})
\end{equation}
where $\indset_j=
\{(i,k,l)\in\N^3\,:\,\,i,l<j,\,\, k\geq 0,\,\, i+k+l=j\}$ and
$\N:=\{0,1,2,\ldots\}$.
The way this algorithm works is that we start with the known function
$B_{\tau,0}$, 
which in its turn gives the function $F_{\tau,1}$ and the constant $c_{\tau,1}$ via
\eqref{eq:Fj} and \eqref{eq:cj}, respectively. This then given $B_{\tau,1}$ 
from the expression in Theorem~\ref{thm:main-coeff}. 
In the next round, we obtain $F_{\tau,2}$ and $c_{\tau,2}$ followed 
by $B_{\tau,2}$ in a similar fashion. An inductive procedure gives 
$F_{\tau,j}$, $c_{\tau,j}$, and $B_{\tau,j}$ for all $j\ge 2$ as well.
Knowing $B_{\tau,j}$ then gives the coefficient function $\calB_{\tau,j}$
as well, by Theorem~\ref{thm:main-coeff}.

As a direct consequence of Theorems~\ref{thm:main-pw} and \ref{thm:main-coeff}, 
we resolve the boundary universality conjecture (Conjecture~\ref{conj:erf}) 
for $1$-admissible potentials.
For the convenience of the reader, we recall some notation. 
For $z_0\in\partial\calS_1$ we denote by $\mathrm{n}$ the 
outward unit normal to $\partial\calS_\tau$ at $z_0$, 
and write $z_m(\xi)$ for the rescaled variable around $z_0$ given by
\eqref{eq:resc1}.

\begin{cor}\label{cor:erf}
Assume that $Q$ is $1$-admissible, and denote by $\kcorrker_m$ the
rescaled kernel
\[
\kcorrker_m(\xi,\eta)=\frac{1}{2m\hDelta Q(z_0)}\Kcorrker_m(z_m(\xi),z_m(\eta)).
\]
Then, there exist unimodular continuous functions $c_m:\C\to \T$ such that
we have the convergence
\[
\lim_{m\to\infty}c_m(\xi)\bar{c}_m(\eta)\kcorrker_m(\xi,\eta)=\kcorrker(\xi,\eta),
\]
locally uniformly on $\C^2$, where the limiting kernel is the
Faddeeva plasma kernel
\[
\kcorrker(\xi,\eta)=
\e^{\xi\bar\eta-\frac{1}{2}(\lvert\xi\rvert^2+\lvert \eta\rvert^2)}
\,\mathrm{erf}\,(\xi+\bar\eta).
\] 
\end{cor}

\begin{rem}
The above kernel convergence has an interpretation in terms of 
determinantal point processes in the plane. More precisely,  
the blow-up of the eigenvalue process for the RNM-ensemble 
around $z_0$ converges to the Fadeeva plasma point field, 
with correlation kernel $\kcorrker(\xi,\eta)$.
The unimodular continuous functions $c_m$ are irrelevant, as they 
do not affect determinantal point processes.
\end{rem}

To complement the present exposition on planar orthogonal polynomials, we
explain in \cite{HW-ose} how the ideas developed here
also apply to give a full asymptotic expansion of
the Bergman kernel for exponentially varying weights when one of the 
variables is away from the corresponding droplet. In that setting, the droplet
arises typically from the repulsive effect of patches where $\hDelta Q<0$.
This result gives error function transition behavior along smooth loops 
of the droplet boundary.

In follow-up work \cite{HW-free}, we intend to explore further 
the implications of Theorem~\ref{main} and \ref{thm:main-coeff} for the 
theory of random normal matrices. In particular, 
we analyze the asymptotics of the free energy $\log \mathcal{Z}_{m,Q}$, 
where $\mathcal{Z}_{m,Q}$ denotes the partition function 
of the RNM-ensemble, and relate the analysis to the planar analogue of the 
classical Szeg\H{o} limit theorem on Toeplitz determinants.

\subsection{Sketch of the main ideas}
\label{ss:idea-sketch}
The first step towards obtaining Theorem~\ref{thm:main-pw} 
is the construction of a family of approximately orthogonal quasipolynomials,
defined outside a compact subset $\calK_\tau$ of the interior of the
droplet $\calS_\tau$. This family of functions have the property
that they are approximately orthogonal to the collection 
of lower degree polynomials, have the correct polynomial 
growth at infinity, but need not
be well-defined globally (i.e.\ on $\calK_\tau$).
In a second step, these quasipolynomials may 
be corrected to true polynomials using
H{\"o}rmander's $\bar\partial$-estimates. 
The actual construction depends on our key lemma 
(Lemma~\ref{lem:main-flow})
which establishes the existence of what we call the 
{\em orthogonal foliation flow}. 

We turn to the underlying ideas for the orthogonal foliation flow.
Our approach will take a slightly different point 
of view than what is used later on. It has the advantage of being more 
intuitively direct.
The approach begins with the following disintegration formula:
let $\{\gamma_{m,n,t}\}_{t}$ denote a smoothly 
varying family of closed simple curves, 
which foliate a region $\Omega_{m,n}$ when $t$ runs through an interval $J_m$.
If $\nu(z)$ denotes the scalar normal velocity 
of the curve flow as it passes through $z$, then
for a suitably integrable function $F$ we have
\begin{equation}\label{eq:disint}
  \int_{\Omega_{m,n}}F(z)\,\e^{-2mQ(z)}\diffA(z)=
  2\int_{J_m}\int_{\gamma_{m,n,t}}F(z)\,\e^{-2mQ(z)}\nu(z)\diffs(z)\diff t
\end{equation}
We consider the weighted arc length measure $\e^{-2mQ}\nu\diffs$ restricted 
to the curve $\gamma_{m,n,t}$, and the associated orthogonal polynomial $P_{m,n,t}$
of degree $n$.
We would like to find a foliation $\{\gamma_{m,n,t}\}_{t}$ 
of the region $\Omega_{m,n}$ 
such that $P_{m,n,t}=c(t)P_{m,n,0}$, 
where $P_{m,n,0}$ is independent of the
flow parameter $t$ and $c(t)$ is an appropriate positive constant. 
As a consequence of \eqref{eq:disint}, the polynomial $P_{m,n,0}$ is 
then orthogonal to $\mathrm{Pol}_n$ with respect to the measure
$1_{\Omega_{m,n}}\e^{-2mQ}\diffA$.
Now, if the foliation covers a sufficiently 
large enough region $\Omega_{m,n}$, 
then the resulting normalized 
orthogonal polynomial ought to be close to $P_{m,n}$ itself. 
In other words, the two-dimensional orthogonality relations
{\em foliate} into lower-dimensional orthogonality relations along
a curve family $\{\gamma_{m,n,t}\}_{t}$. 

The stationarity condition $P_{m,n,t}=c(t)P_{m,n,0}$ is
quite demanding,
and in fact we do not know that such a foliation exists, 
at least if we require it to foliate the entire plane.
Instead, we obtain the foliation in an approximate sense, 
up to any given precision, so that $\Omega_{m,n}$ covers a band around
$\partial\calS_\tau$
of width $\asymp m^{-\frac12}\log m$.
We remark that the stationarity condition may be 
thought of as a Hele-Shaw flow 
condition (see \cite{gvt}, \cite{HS})for the curves 
$\gamma_{m,n,t}$, with respect to the weight 
$|P_{m,n,0}|^2\e^{-2mQ}$. Hele-Shaw flows are 
notorious for singularity formation,
after which the foliation flow cannot be continued.
The requirement not to develop such singularities 
puts a strong requirement 
on the weight $|P_{m,n,0}|^2\e^{-2mQ}$.
This is used in an approximate fashion in 
Section~\ref{sec:flow} to devise an algorithm
to construct $P_{m,n,0}$ together with the 
foliation iteratively 
in a self-improving manner.
For technical reasons, we work with the  
flow curves $\Gamma_{m,n,t}=\phi_\tau(\gamma_{m,n,t})$ 
after applying the conformal mapping $\phi_\tau$, and 
consider quasipolynomials rather than polynomials.

\subsection{An outline of the presentation}
In Section~\ref{sect:prel}, we introduce some auxiliary material which 
will be needed later on. In particular, we 
discuss some aspects of weighted logarithmic potential theory and 
obstacle problems, and introduce the concept of weighted Laplacian growth. 
Moreover, we collect some results on H{\"o}rmander-type $L^2$-estimates 
for the $\bar\partial$-operator, and the asymptotic analysis 
of integrals based on Laplace's method.

In Section~\ref{s:existence-expansion}, we introduce the notion of
quasipolynomials, and state our key lemma on the orthogonal foliation flow
(Lemma~\ref{lem:main-flow}). 
Using H{\"o}rmander $\bar\partial$-techniques
we obtain the $L^2$-analogue of the main theorem
(Theorem~\ref{main}) from the key lemma. The main theorem
(Theorem~\ref{thm:main-pw})
then follows from Theorem~\ref{main} by a weighted Bernstein-Walsh lemma.

In Section~\ref{sect:alg}, we obtain
Theorem~\ref{thm:main-coeff}, which identifies the coefficient functions in
the asymptotic expansion. The proof is based on steepest descent analysis. 
The starting point is the existence of the expansion of
Theorem~\ref{thm:main-pw} which tells us that the probability
distribution $|P_{m,n}|^2\e^{-2mQ}$ is approximately
a Gaussian ridge centered around $\partial\calS_\tau$, so by composing with the
conformal mapping $\phi_\tau$ we obtain a Gaussian ridge around the
unit circle. 
By writing the relevant integrals in polar coordinates and applying
Laplace's method in the radial direction, this structure allows us
to {\em collapse} the orthogonality conditions into integral equations
on the unit circle.
The collapsed orthogonality conditions then reduce to inhomogeneous 
Toeplitz kernel equations. 
The algorithm arises when we solve those equations.

In Section~\ref{sect:erfc}, we supply more details on  
determinantal point processes, and give the proof of 
Corollary~\ref{cor:erf} on boundary universality
in the random normal matrix model for $1$-admissible potentials.

In Section~\ref{sec:flow}, we supply the proof of key lemma on the
existence of the orthogonal foliation flow. The proof is based on
an algorithm, which determines
both the flow and the asymptotic expansion of the approximately orthogonal
quasipolynomials in an iterative and intertwined fashion.
An outline of the algorithm is provided in 
Subsections~\ref{ss:flow1} and \ref{ss:alg-flow-over}.

Finally, in Section \ref{s:RHP}, we connect our orthogonal foliation 
flow with the Its and Takhtajan approach involving soft Riemann-Hilbert
problems ($2\times2$ matrix $\bar\partial$-problems).

\subsection{Acknowledgments} We wish to thank the anonymous referee
for several helpful and insightful comments, which has significantly
improved the manuscript. 
In addition, we would like to thank Gernot Akemann, Nam-Gyu Kang, Simon Larson,
Nikolai Makarov, and Ofer Zeitouni for stimulating discussions.

\subsection{Notation and conventions}
\label{ss:notation}
We denote by $\partial_z$ and $\bar\partial_z$ the standard Wirtinger
derivatives, given by
\begin{equation}\label{eq:def-Wirtinger}
\partial_z=\frac12\big(\partial_x-\imag\partial_y\big)\quad\text{and}\quad 
\bar\partial_z=\frac12\big(\partial_x+\imag\partial_y\big),\qquad z=x+\imag y.
\end{equation}
When the dependence on $z$ is clear we will omit the subscript
and simply write $\partial$ and $\bar\partial$.
The Laplacian factorizes as $\hDelta =\partial\bar\partial$
(notice that this is a quarter of the usual Laplacian).

The Riemann sphere is denoted by $\widehat{\C}$, and
we identify it with the extended complex plane $\widehat{\C}=\C\cup\{\infty\}$
via stereographic projection. If $\Gamma$ is a bounded 
Jordan curve, and $\Omega_\e$ denotes the unbounded 
component of $\C\setminus \Gamma$, then the domain 
$\Omega_\e$ is simply connected if regarded
as a domain on the Riemann sphere $\widehat{\C}$.
As a consequence, the Riemann mapping theorem 
guarantees that there exists a conformal
mapping $\phi:\Omega_\e\to \D_\e$ onto the exterior disk $\D_\e$. 
This mapping is uniquely determined if we require that 
\begin{equation}\label{eq:def-orthostatic}
  \phi(\infty)=\infty \quad \text{and}\quad \phi'(\infty)>0.
\end{equation}
A conformal mapping of unbounded domains which is subject to the normalization 
\eqref{eq:def-orthostatic} at infinity
is called {\em orthostatic}. 
Unless specified otherwise, a conformal mapping $\phi:\Omega_1\to\Omega_2$ is 
tacitly assumed to be onto.

We use the standard Landau notation for control of asymptotic quantities.
Namely, if $f(t)$ and $g(t)$ denote two positive functions
defined for $t\in (0,1]$, 
we say that $f(t)=\Ordo(g(t))$ as $t\to 0$ if
there exists a constant $C$ with $0<C<\infty$ 
such that $f(t)\le C g(t)$ for all $t>0$
sufficiently small.
Moreover, we say that $f(t)=\ordo(g(t))$ as $t\to 0$
if $\lim_{t\to 0} f(t)/g(t)=0$.
Moreover, we use the notation $f(t)\asymp g(t)$ to say that
$f(t)=\Ordo(g(t))$ and $g(t)=\Ordo(f(t))$, as $t\to 0$.
Similar comparisons when $f$ and $g$ are functions defined on more general sets 
are understood analogously.

For a positive Borel measure $\mu$ supported on the set $\Omega \subset \C$,
we denote by $L^2(\Omega,\mu)$ the standard $L^2$-space of square
integrable functions with respect to $\mu$, with inner product 
\[
\langle f,g\rangle_\mu=\int_\Omega f(z)\overline{g(z)}\,\diff\mu(z).
\]
For a domain $\Omega\subset \C$, we define the Bergman space
$A^2(\Omega,\mu)$ 
as the subspace of $L^2(\Omega,\mu)$ consisting of all $f\in L^2(\Omega,\mu)$
which are holomorphic on $\Omega$.
For an integer $n$ and unbounded $\Omega$, we denote by $L^2_n(\Omega,\mu)$
and $A^2_n(\Omega,\mu)$
the subspaces of functions $f$ with
\[
|f(z)|=\Ordo(|z|^{n-1}),\qquad z\in\Omega,\;\,|z|\to+\infty.
\]
If $\Omega=\C$ is the entire complex plane, we drop it from the notation. 
Measures of the form $\diff\mu=\e^{-\phi}\diffA$ play a major role
in our analysis, and for such measures
we use the shorthand notation $A^2_\phi(\Omega)$, $L^2_\phi(\Omega)$, 
$A^2_{\phi,n}(\Omega)$, and $L^2_{\phi,n}(\Omega)$ for the spaces discussed above.
The $L^2$ norm and inner products are denoted by $\lVert \cdot \rVert_\mu$ and 
$\langle \cdot,\cdot\rangle_{\mu}$, or simply by $\lVert \cdot \rVert_\phi$ and 
$\langle \cdot,\cdot\rangle_{\phi}$ in the case of measure of the form
$\diff\mu=\e^{-\phi}\diffA$.

For the convenience of the reader, we supply a list of frequently used notation.
\medskip

\begin{tabular}{ll}
$\C$, $\R$, $\T$& Complex plane, real line, and unit circle, 
respectively. 
\\
$\D$, $\D_\e$ & Open unit disk $\D=\{z:|z|< 1\}$ and exterior disk
$\D_\e=\{z:|z|>1\}$, also
\\ & for arguments $(z_0,r)$ denoting center and radius of the boundary circle.
\\
$\Z$, $\N$, $\Z_+$& Integers, natural numbers 
$\N=\{0,1,2,\ldots\}$ and positive integers \\ 
& $\Z_+=\{1,2,3,\ldots\}$, respectively.
\\
$\mathcal{E}^c, \mathcal{E}^\circ,\overline{\mathcal{E}}$& 
Complement, interior, and closure of the set $\mathcal{E}$. The complement is
\\ & understood as
$\C\setminus \mathcal{E}$, unless specified otherwise.
\\
$1_\mathcal{E}$ & Indicator function for the set $\mathcal{E}$.
\\
$\partial_z, \bar\partial_z$ & Wirtinger derivatives, 
given by $\partial_z=\frac{1}{2}(\partial_x - \imag\partial_y)$, 
$\bar\partial_z=\frac{1}{2}(\partial_x + \imag\partial_y)$, 
\\
& where $z=x+\imag y$.
\\
$\hDelta$& Laplacian, which factorizes as $\hDelta=\partial\bar\partial$. 
N.B.: this equals one-quarter of 
\\
& the usual Laplacian.
\\
$\mathrm{Pol}_n$ & Space of polynomials of degree at most $n-1$.
\\
$Q$, $\checkQ_\tau$ & The potential and the solution to 
obstacle problem with growth $\tau\log|z|$ 
\\ & at infinity, respectively.
\\
$\breve{Q}_\tau$& Harmonic extension of $\checkQ_\tau\big\vert_{\calS_\tau^c}$ 
across $\partial\calS_\tau$.\\
$Q_\tau^\circledast$& Bounded harmonic extension of 
$Q\big\vert_{\partial\calS_\tau}$ to $\calS_\tau^c$. 
\\
$\calQ_\tau$& Holomorphic function on $\calS_\tau^c$ with 
$\Re \calQ_\tau=Q_\tau^\circledast$ and $\Im \calQ_\tau(\infty)=0$.
\\
$\calS_\tau,\calS_\tau^\star$ & The droplet and the coincidence 
set for the obstacle problem, respectively.
\\
& These are equal under the $\tau_0$-admissibility assumption, for
$|\tau-\tau_0|$ small.
\\
$\calK_{0,\tau}$, $\calK_\tau$ & Compact subsets of $\calS_\tau^\circ$
related with the radii $\rho_{0,0}$ and $\rho_0$, respectively.
\\
$I_{\epsilon_0}$ & $I_{\epsilon_0}=[1-\epsilon_0,1+\epsilon_0]$ for a small
parameter $\epsilon_0>0$.
\\
$\phi_\tau$& Conformal mapping $\phi_\tau:\calS_\tau^c\to\D_\e$ with 
$\phi_\tau(\infty)=\infty$ and $\phi'_\tau(\infty)>0$.
\\
$R_\tau$ & The modified potential, given by
$(Q-\breve{Q}_\tau)\circ\phi_\tau^{-1}$.
\\
$\chi_{\tau,0}, \chi_{\tau,1}$ & Smooth cut-off functions related via 
$\chi_{\tau,0}=\chi_{\tau,1}\circ \phi_\tau$.  
\\
$\hm(E,\Omega,z_0)$& Harmonic measure of $E$ relative to $(\Omega,z_0)$.
\\
$H^2$, $H^2_{-}, H^2_{-,0}$ & Hardy spaces, cf.\ Subsection~\ref{ss:herglotz}.
\\
$\Hop_\Omega$&The Herglotz operator for a domain $\Omega$ 
containing the point at infinity. 
\\
$\Pop_{H^2},\Pop_{H^2_{-}}$&Orthogonal projection onto Hardy spaces. 
\\
$\indset$, $\indsetS$, $\indsetT$, $\indsett$ & Index sets, appearing with 
various subscripts and superscripts. 
\\ 
& See pp. 36, 43 and 49.
\\$\Lop_k$, $\Mop_k$ &Differential operators arising in steepest descent
 calculations.
\\
$\calB_{\tau,j}$, $B_{\tau,j}$ & Coefficient functions 
in asymptotic expansions of ONPs, related through 
\\
& the conformal mapping $\phi_\tau$ (see Theorem~\ref{thm:main-coeff}).
\\
$\psi_{s,t}$, $\hat{\psi}_{j,l}$, $b_j$ & Conformal mappings related to the
orthogonal foliation flow, their Taylor \\ & coefficients in $(s,t)$, 
and bounded holomorphic coefficient functions.
\\
$\Gamma_{m,n,t}$, $\calD_{m,n}$ & The curves of the orthogonal foliation
and the foliated region, 
respectively.
\\
$\logdens_{s,t}$, $\hat{\logdens}_{j,l}$ & The logarithmic density in the
master equation and its Taylor coefficients in $(s,t)$, 
see 
\\ & Subsection~\ref{ss:alg-flow-over}.
\\
$\Vop_{m,n}$ & Canonical positioning operator, cf.\ Subsection~\ref{ss:reduct}.
\\
$F_{m,n}^{\langle \kappa\rangle}$, $f_{m,n}^{\langle \kappa\rangle}$ & 
Quasipolynomial and analogous bounded function, related through $\Vop_{m,n}$.
\\
$\delta_m$ & The number $\delta_m=m^{-\frac12}\log m$.
\\
$\hat{\mathbb{A}}(\hrho,\sigma)$ & 
The $2\sigma$-fattened diagonal annulus, cf.\ 
Subsection~\ref{ss:flow0}.
\\
$\prec_{\mathrm{L}}$, $\prec_{\mathrm{OL}}$ & Lexicographic and
order-lexicographic orderings.
\\
$\mathrm{POL}(\cdot)$ & Polynomial complexity classes, 
cf.\ Subsection~\ref{ss:pol-complexity}.
\\
$\calG_{\boldsymbol{\mu},\boldsymbol{\nu}}$, 
$\calH_{\boldsymbol{\mu},\boldsymbol{\nu}}$ &
Non-linear differential expressions for Fa{\`a} di Bruno's formula.
\end{tabular}

\section{Preliminaries}\label{sect:prel}
\subsection{An obstacle problem and logarithmic potential 
theory}\label{ss:prel-pot}
In this section, we follow the presentation of \cite{HM}.
The standard reference for the potential theoretic aspects of 
this material is the monograph \cite{SaffTotik} by Saff and Totik.

For a positive real parameter $\tau$, let $\mathrm{Subh}_\tau(\C)$
denote the convex set of all subharmonic 
functions $q:\C\to\R\cup\{-\infty\}$ on the complex plane $\C$ which
meet the growth bound
\[
q(z)\le \tau \log |z|+\Ordo(1)
\] 
as $\lvert z\rvert\to\infty$. 
For lower semicontinuous potentials $Q$ 
subject to the growth condition \eqref{eq-Qcond}, we 
let $\checkQ_\tau$ be the solution to the obstacle problem
\begin{equation}\label{eq:subh}
\checkQ_\tau(z):=\sup\big\{q(z):\,q\in\mathrm{Subh}_\tau(\C)\,\,\,
\text{and}\,\,\,
\,q\le Q\,\,\,\text{on}\,\,\,\C\big\},
\end{equation}
and observe that trivially $\checkQ_\tau\le Q$, and if we regularize 
$\checkQ_\tau$ on a set of logarithmic capacity $0$ (and keep the 
same notation for the regularized function) 
then $\checkQ_\tau\in\mathrm{Subh}_\tau(\C)$ holds.
Suppose now that $Q$ is $C^2$-smooth.
Standard regularity results then give that
$\checkQ_\tau$ is $C^{1,1}$-smooth, so that the partial derivatives of 
order $2$ of $\checkQ_\tau$ are locally bounded (in the sense of distribution 
theory), see e.g. \cite{Berman} for a simple argument to this effect. 
As a consequence of the growth condition \eqref{eq-Qcond} on $Q$,
the \emph{coincidence set} defined by
\[
\mathcal{S}_\tau^\star:=\{z\in\C:\, \checkQ_\tau(z)=Q(z)\}
\]
is compact, and moreover, a Perron-type argument shows that
$\checkQ_\tau$ is harmonic off $\calS_\tau^\star$.
It now follows from the $C^{1,1}$-smoothness that
$\hDelta \checkQ_\tau=1_{\mathcal{S}_\tau^\star}\hDelta Q$ 
holds in the sense of 
distribution theory (see \cite[p. 53]{KS}).

The above obstacle problem has a direct 
relation with weighted potential theory. 
The weighted logarithmic energy, with respect to a continuous weight 
function $V:\C\to\R$, of a compactly supported finite real
Borel measure $\mu$ is defined as
$$
I_V[\mu]=\int_{\C\times\C} \log\frac{1}{\lvert z-w\rvert}\diff \mu(z)
\diff\mu(w)+2\int_{\C}V(z)\diff\mu(w).
$$
With $V=\tau^{-1}Q$, we set out to minimize the energy $I_{\tau^{-1}Q}[\mu]$ 
over all compactly supported Borel {\em probability} measures $\mu$. 
There is a unique such minimizer, called the \emph{equilibrium measure}, 
which we denote by $\mu_\tau$. The connection with the obstacle problem 
is via the relation   
\begin{equation}
\label{eq:HSevol}
\diff \mu_{\tau}(z)=
2\tau^{-1}
\hDelta\checkQ_\tau\diffA=
2\tau^{-1}
1_{\calS_\tau^\star}
\hDelta Q(z)\diffA.
\end{equation}
As a consequence, we may recover the logarithmic potential for the
equilibrium measure from $\checkQ_\tau$ and a real constant $F_{Q,\tau}$: 
\[
  U^{\mu_\tau}(z):=\int_\C\log\frac{1}{|z-w|}\diff\mu_\tau(z)
  =-\tau^{-1}
  \checkQ_\tau(z)+F_{Q,\tau},\qquad z\in\C.
\]
Since $\mu_\tau$ is a probability measure by definition, we see from
\eqref{eq:HSevol} that $\hDelta Q\ge 0$ a.e.\ on $\calS_\tau^\star$.
So, the coincidence set $\calS_\tau^\star$ will avoid the open subset 
of the plane where $\hDelta Q<0$, which may be nonempty.
We call the support (as a distribution) of the equilibrium 
measure $\mu_\tau$ the \emph{droplet}, and denote it by $\calS_\tau$. 
We alternatively call it the \emph{spectral droplet}, due to the spectral
interpretation as the accumulation set for the eigenvalues of random matrices. 
In general this is a subset of the coincidence (or contact) set 
$\calS_\tau^\star$. However, the difference set 
$\calS_\tau^\star\setminus\calS_\tau$ 
is small, in the sense that it is a null set with respect to the weighted
area measure $|\hDelta Q|\diffA$.
In this presentation, we will assume throughout that the potential 
$Q$ is $1$-admissible. 
Under this assumption, we have the equality $\calS_\tau=
\calS_\tau^\star$ for $\tau\in I_{\epsilon_0}:=[1-\epsilon_0,1+\epsilon_0]$
with some small but positive $\epsilon_0$.

The function $\checkQ_\tau$ is $C^{1,1}$-smooth, with 
$\checkQ_\tau=Q$ on the droplet $\calS_\tau$, whereas in the complement 
$\calS_\tau^c$ it is harmonic and determined by the boundary data  
that $\checkQ_\tau=Q$ on $\partial\calS_\tau$
and the growth $\checkQ_\tau(z)=\tau\log|z|+\Ordo(1)$ as $|z|\to+\infty$. 
We proceed to introduce some further functions related to the potential $Q$.

\begin{defn}\label{def:Q-funct} Assume that $Q$ is $1$-admissible, 
  and let $\tau\in I_{\epsilon_0}$. Then
\begin{enumerate}[(i)]
\item $\breve Q_\tau$ is defined as
the harmonic extension of the restriction of $\checkQ_\tau$ to 
$\calS_\tau^c$ across $\partial\calS_\tau$. 
\item $Q_\tau^\circledast$ is the bounded harmonic harmonic function on 
$\calS_\tau^c$ which equals $Q$ on $\partial\calS_\tau$, extended harmonically
across $\partial\calS_\tau$. 
\item $\calQ_\tau$ is the bounded holomorphic function in 
$\calS_{\tau}^c$ such that $\re \calQ_\tau=Q^\circledast_\tau$ 
on $\calS_\tau^c$ with $\Im \calQ_\tau(\infty)=0$, extended analytically
across $\partial\calS_\tau$. 
\end{enumerate}
\end{defn}
It is clear that the functions $\breve{Q}_\tau$ and $Q^\circledast_\tau$ are
related via
\begin{equation}\label{eq:Q-breve-circledast}
\breve{Q}_\tau=\tau\log|\phi_\tau| +Q^\circledast_\tau.
\end{equation}

\subsection{A weighted Bernstein-Walsh lemma}
The significance of the set $\mathcal{S}_\tau$ in relation to orthogonal 
polynomials is made clear by Proposition~\ref{prop-gengrowth-ext} 
below. We begin with a useful lemma taken from \cite{ahm1}, see Lemma~3.2.

\begin{lem}
Let $u$ be holomorphic in a disk $\D(z,m^{-1/2}\delta)$ Then
\[
\lvert u(z)\rvert^2 \e^{-2mQ(z)}\leq \frac{m\e^{A\delta^2}}{\delta^2}
\int_{\D(z,m^{-1/2})}\lvert u\rvert^2 \e^{-2mQ}\diffA,
\]
where $A$ denotes the essential supremum of $\hDelta Q$ on 
$\D(z,m^{-1/2}\delta)$. 
\label{prop-gengrowth}
\end{lem}

This lemma is used in \cite{ahm1} to obtain growth bounds 
for polynomials of degree at most $n$. The approach works more generally, for
functions of polynomial growth 
in the space $A^2_{2mQ}(\calK^c)$ defined in Subsection~\ref{ss:notation},
where $\calK$ is a compact subset of the interior of the droplet $\calS_\tau$.
The following result generalizes the classical Bernstein-Walsh lemma, see 
e.g.\ Chapter III.$2$ in \cite{SaffTotik}.

\begin{prop}
Let $\tau=\frac{n}{m}$, and suppose $\mathcal{K}$ is a compact subset of the 
interior of $\mathcal{S}_\tau$. 
Then there exists a positive constant $C$ such that for any 
$u\in A^2_{2mQ}(\mathcal{K}^c)$ 
with the polynomial growth control 
$\lvert u(z)\rvert=\Ordo\left(\lvert z\rvert^{n}\right)$ as 
$|z|\to\infty$, we have that
$$
|u(z)|\le Cm^{\frac12}\|u\|_{L^2(\mathcal{K}^c,\e^{-2mQ})}
\e^{m\checkQ_\tau(z)},\qquad \mathrm{dist}_\C(z, \calK)\geq \delta m^{-1/2}.
$$
\label{prop-gengrowth-ext}
\end{prop}
\begin{proof}[Proof]
Assume that $z\in\calS_\tau\setminus\calK$ lies at a distance of at least 
$m^{-1/2}\delta$ from $\calK$. By Lemma~\ref{prop-gengrowth},
we have the estimate
$$
\lvert u(z)\rvert^2 \leq \frac{m\e^{2A\delta^2}}{\delta^2}\e^{2mQ(z)}
\lVert u\rVert^2_{L^2(\mathcal{K}^c,\e^{-2mQ})},
$$
which yields the claim for $z\in\calS_\tau\setminus\calK$ with the 
constant $C=C_\delta=\delta^{-1}\e^{A\delta^2}$.  
Next, suppose that $u$ has norm equal to $1$, and let $q(z)$ be the 
subharmonic function
$$
q(z)=\frac{1}{2m}\log \frac{\lvert u(z)\rvert^2}{mC_\delta^2},
\qquad z\in\calK^c.
$$
It follows from the above estimate on $\lvert u(z)\rvert^2$ that
$q(z) \leq Q$ for $z\in \calS_\tau\setminus\calK$,
and the growth bound on $\lvert u(z)\rvert$ as $|z|\to\infty$ entails that
$q(z) \leq \tau\log\lvert z\rvert+\Ordo(1)$ as $|z|\to\infty$.
Now, we consider the difference $q-\checkQ_\tau$ and 
observe that it is harmonic in $\calS_\tau^c$ and that $q-\checkQ_\tau\le0$
holds on the boundary $\partial\calS_\tau$, since $\checkQ_\tau=Q$ there.
Moreover, we see from the growth bound on $q$ that 
the difference $q-\checkQ_\tau$ is bounded from above in $\calS_\tau^c$.
It now follows from the maximum principle for subharmonic functions that 
$q(z)- \checkQ_\tau(z)\le 0$ 
throughout $z\in\calS_\tau^c$, which completes the proof.
\end{proof}

In particular, Proposition~\ref{prop-gengrowth-ext} 
tells us that $\lvert P_{m,n}(z)\rvert^2\e^{-2mQ}$ 
decays exponentially off the droplet $\mathcal{S}_\tau$ if $\tau=\frac{n}{m}$.
As alluded to in the introduction, it is possible to also locate 
the mass of the probability density 
$\lvert P_{m,n}(z)\rvert^2\e^{-2mQ(z)}$. We recall the notation 
$\hm(\cdot ,\widehat{\C}\setminus\mathcal{S}_t, \infty)$ for the harmonic 
measure of $\widehat{\C}\setminus\mathcal{S}_t$ relative to the point 
at infinity. The following is from \cite{ahm2}.

\begin{thm}
As $m,n\to\infty$ with $\tau=\frac{n}{m}=\tau_0+\Ordo(m^{-1})$ for some 
$\tau_0$ with $0<\tau_0\le1$, we have the convergence
$$
\lvert P_{m,n}\rvert^2\e^{-2mQ}\to\hm(\cdot ,
\widehat{\C}\setminus\mathcal{S}_{\tau_0}, \infty),
$$
in the sense of weak-star convergence of measures.
\end{thm}

See Figure~\ref{fig:Berezin-ONP} (right) above for an 
illustration of this convergence.

\subsection{Weighted Laplacian growth}
Weighted Laplacian growth (or weighted Hele-Shaw flow) describes the movement 
of the boundary of a viscous fluid droplet 
in a porous medium, as fluid is injected into the droplet. The weight 
appears as a result of the variable permeability of the medium, or, 
alternatively, as a result of curved geometry. For the 
mathematical formulation, 
consider a simply connected domain 
$\Omega_0$ on the Riemann sphere $\widehat{\C}:=\C\cup\{\infty\}$
containing the point at infinity.
A smoothly increasing family $\{\Omega_t\}_t$ 
of domains is said to be a Hele-Shaw flow with weight $\omega$, relative 
to the injection point at infinity, if the infinitesimal change of the measure 
$1_{\Omega_\tau}\omega(z)\diffA$
equals harmonic measure (the derivative is as usual taken in the sense of 
distribution theory):
\begin{equation}
\label{eq:HSflow}
\partial_t (1_{\Omega_t}\omega\diffA)=
\diff\hm(\cdot,\Omega_t,\infty).
\end{equation}
Alternatively, we can think in terms of the weak formulation, which amounts to
the requirement that
\[
\int_{\Omega_t\setminus\Omega_s}h\,\omega\diffA
=(t-s)h(\infty), \quad s<t,
\]
holds for all bounded harmonic functions $h$ on $\Omega_t$. 
At times, we prefer to think of the flow of the boundary
loops $\{\partial\Omega_t\}_t$ rather than the flow of domains itself.
A basic reference on Hele-Shaw flow is the book \cite{gvt} by Gustafsson, 
Teodorescu and Vasili'ev. 
The weighted Hele-Shaw flow problem appears to 
have been treated first in the paper \cite{HS} by Hedenmalm and Shimorin,
where the weight was interpreted as a Riemannian metric, motivated by 
considerations in the potential theory of clamped plates \cite{HS2}. 
This line of work 
is continued by \cite{HP}, \cite{HO}. 
In this connection, we mention the work \cite{RWN} by Ross
and Witt-Nystr{\"o}m, which deals with a less regular
situation.

In the present work, weighted Laplacian growth appears for two distinct 
families of weights that arise naturally.
For instance, the complement $\calS_\tau^c$ evolves according to Laplacian
growth with the weight $2\hDelta Q$ and injection point at infinity, 
with $\tau$ as backward time. 
The second type of Laplacian growth occurs with the weight 
$\omega=|P|^2\e^{-2mQ}$, where $P$ is an approximation of the  
orthogonal polynomial $P_{m,n}$, see the discussion in 
Subsection~\ref{ss:idea-sketch}. The latter flow of loops is what we call
the orthogonal foliation flow. 

We will need the following lemma, about the movement of the loops
$\partial\calS_\tau$ as $\tau$ varies.

\begin{lem}
\label{lem:boundary-movement}
Fix $\tau\in I_{\epsilon_0}=[1-\epsilon_0,1+\epsilon_0]$. 
Denote by ${\rm n}_\tau(\zeta)$ the outer unit normal to 
$\partial\calS_\tau$ at a 
point $\zeta\in\partial\calS_\tau$, and let
${\rm n}_\tau(\zeta)\R$ denote the straight line which contains 
${\rm n}_\tau(\zeta)$ and the origin.
Then, if for real $\varepsilon$ the point $\zeta_{\varepsilon}$ 
is closest to $\zeta$ in the intersection
$$
(\zeta-{\rm n}_\tau(\zeta)\R)\cap \partial\calS_{\tau-\varepsilon},
$$
we have as $\varepsilon\to0$ that
$$
\zeta_{\varepsilon}=\zeta-\varepsilon{\rm n}_\tau(\zeta)
\frac{\lvert \phi'_\tau(\zeta)\rvert}{4\hDelta Q(\zeta)}+\Ordo(\varepsilon^2)
$$
and the outer normal ${\rm n}_{\tau-\varepsilon}(\zeta_\varepsilon)$ satisfies
$$
{\rm n}_{\tau-\varepsilon}(\zeta_\varepsilon)
={\rm n}_\tau(\zeta)+\Ordo(\varepsilon).
$$
\end{lem}

\begin{proof}
We recall that the compact sets $\calS_\tau$ evolve according to weighted 
Laplacian growth with respect to the weight $2\hDelta Q$, so that we have
\eqref{eq:HSflow} with $\Omega_\tau=\calS_\tau^c$.
For the details, we refer to Theorem~$5.22$ and Proposition~$6.10$
in \cite{HM}. This means that
\begin{equation}
  \label{eq:evol-eq}
\partial_\tau(1_{\calS_\tau}2\hDelta Q\diffA)=\diff\hm(\cdot,\calS_\tau^c,\infty)
=|\phi_\tau'|\diffs,
\end{equation}
where we recall that $\phi_\tau$ is the (surjective) conformal mapping 
$\calS_\tau^c\to\D_\e$. Informally, the boundary $\partial\calS_\tau$ moves
at local speed $(4\hDelta Q)^{-1}|\phi'_\tau|$ in the exterior normal direction,
where the number $4$ appears in place of $2$ as a result of the different 
normalizations associated with $\diffs$ and $\diffA$. 
It is known by Theorem~$6.2$ in \cite{HS}, which is based on the
Nishida-Nirenberg version of the Cauchy-Kovalevskaya theorem, 
that the loops $\partial\calS_\tau$ deform real-analytically as $\tau$ varies.
In view of this fact and the evolution equation \eqref{eq:evol-eq}, 
the claimed assertions follow from Taylor's formula.
\end{proof}

\subsection{Polynomial
\texorpdfstring{$\bar\partial$}{d-bar}-methods}
Let $\phi$ be a strictly subharmonic function on $\C$. H{\"o}rmander's 
classical result states that the inhomogeneous $\bar\partial$-equation
$$
\bar\partial u=f
$$
can be solved for any datum $f\in L^2_\text{loc}(\C)$ with the estimate
$$
\int_\C \lvert u\rvert^2\e^{-\phi}\diffA\le 
\int_{\C}\lvert f\rvert^2\frac{\e^{-\phi}}{\hDelta \phi}\diffA.
$$
Taking this a starting point, in \cite{ahm1}, Ameur, Hedenmalm, and 
Makarov investigate the case when the solution $u$ is constrained by an 
additional polynomial growth condition at infinity. 
We now describe this result. 
Recall from Subsection~\ref{ss:notation} that $L^2_{\phi,n}(\C)$ denotes 
the subspace of $L^2_\phi(\C)$ subject to the growth restraint
\[
f(z)=\Ordo(|z|^{n-1})
\]
near infinity. The polynomial growth Bergman space $A^2_{\phi,n}(\C)$ is 
analogously defined there. We will consider these spaces with $\phi=2m Q$.

The following is a direct consequence of Theorem~4.1 in \cite{ahm1}. 

\begin{prop}\label{bh}
Let $f\in L^\infty(\calS_\tau)$. Then the $L^2_{2mQ,n}(\C)$-minimal 
solution $u_{0,n}$ to the problem
$$
\bar\partial u_{0,n}=f
$$ 
satisfies
\begin{equation}
\int_{\C}\lvert u_{0,n}\rvert^2 \e^{-2mQ}\diffA\leq \frac{1}{2m}\int_{\calS_\tau}
\lvert f\rvert^2\frac{\e^{-2mQ}}{\hDelta Q}\diffA,
\end{equation}
provided that the right-hand side is finite.
\end{prop}

\begin{proof}
We apply Theorem~4.1 of \cite{ahm1} with $\calT=\calS_\tau$, $\phi=2mQ$, 
$\varrho=0$, and  
$$
\hat{\phi}=2m\Big(1-\frac{\varepsilon}{\tau}\Big)
\checkQ_\tau+\varepsilon m\log(1+\lvert z\rvert^2).
$$
Then all conditions except $(ii)$ are trivially satisfied with 
$a,b=\mathrm{o}(1)$ as $\varepsilon\to 0^+$. To see why $(ii)$ holds,
it is enough to observe that
$$
\hat{\phi}(z)=2m\tau\big(1-\frac{\varepsilon}{\tau}\big)\log \lvert z\rvert 
+2\varepsilon m\log \lvert z\rvert+\Ordo(1) =\log(\lvert z\rvert^{2n})+\Ordo(1)
$$
as $|z|\to\infty$. Hence the inclusion $A^2_{\hat\phi}\subset
\operatorname{Pol}_{n}$ follows.
Letting $\varepsilon\to0^+$ for fixed $m$ and $n$ completes the proof.
\end{proof}

\begin{rem}
In Theorem~4.1 of \cite{ahm1} there is an additional freedom to modify
the weight with a function $\varrho$, which we set to equal $\varrho=0$
in the above.
The conditions on $\varrho$ are such that there is flexibility in the interior
direction inside the droplet, but none in the exterior or along the boundary.
As $\varrho$ is used to control the norm-minimal solution to the
$\bar\partial$-equation, this flexibility tells us that decay of the datum $f$
in the interior of the droplet translates to a corresponding decay of the
solution $u_{0,n}$.
On the other hand, decay of the datum near a boundary 
point in the tangential direction will not
necessarily have the same effect.
\end{rem}

\subsection{Holomorphic boundary value problems and 
Toeplitz operators}
\label{ss:herglotz}
For the reader's convenience, we include some elementary facts from 
the theory of Herglotz kernels and Hardy spaces. 
Let $f$ be holomorphic in the unit disk $\D$ with continuous extension 
to the boundary. 
The classical Herglotz integral formula \cite[pp. 52]{harmonicmeasure} 
asserts that
$$
f(z)=\int_{\T}\frac{\zeta+ z}{\zeta- z}\Re(f(\zeta))\,\diffs(\zeta) + 
\Im(f(0)),\qquad z\in \D.
$$
If $F\in L^1(\T)$ is real-valued, this allows us to 
solve the boundary value problem
$$
\Re f\big\vert_{\T}=F
$$
where $f$ is holomorphic in the disk by the integral formula
$$
f(z)=\Hop_{\D}F(z):=\int_{\T}\frac{\zeta+ z}{\zeta -z}F(\zeta)\,
\diffs(\zeta),\quad z\in\D.
$$
Moreover, the solution is unique up to an additive imaginary constant. 
For us, it is more natural to work in the exterior disk. 
By reflection in the unit circle, we obtain the formula
$$
f(z)=\Hop_{\,\D_\e}F(z):=\int_{\T}\frac{z+\zeta}{z-\zeta}F(\zeta)\,
\diffs(\zeta),\quad z\in\D_\e,
$$
which we refer to as the {\em Herglotz transform} of $F$. If $F$ is
$L^2(\T)$-integrable, its Herglotz transform is in the Hardy space $H^2$.
If we assume slightly more smoothness, e.g. that $F$ is $C^1$-smooth, then
its Herglotz transform is continuous and bounded in the closed exterior disk
$\bar\D_\e$. Analogously, if we have a lot of smoothness, e.g. 
$F$ is $C^\omega$-smooth, then its Herglotz transform extends to a bounded
analytic function on a slightly bigger exterior disk $\D_\e(0,\rho)$ with
$\rho<1$. 
We recall the definition of the Hardy space $H^2=H^2(\D)$ mentioned above.
A function $f$ is in $H^2$ if it is holomorphic in $\D$ with
$$
\sup_{0<r<1}\int_{\T}\lvert f(r\zeta)\rvert^2\diffs(\zeta)<+\infty.
$$
Alternatively, in terms of the boundary values, $H^2$ is the closed subspace of 
$L^2(\T)$ defined by the property that the Fourier coefficients with negative
index all vanish. 
The conjugate Hardy space $H^2_{-}$ consists of all functions of 
the form $\bar{f}$, where $f\in H^2$,
which may also be viewed as Hardy space on the exterior disk $\D_\e$.
In a similar fashion, the standard $H^p$-spaces can be defined as well. 
For instance, for $p=\infty$ the space $H^\infty$ consists of the 
bounded holomorphic functions in the unit 
disk $\D$ equipped with the supremum norm.

Associated with the Hardy and conjugate Hardy subspaces of $L^2(\T)$
there are the orthogonal projections $\Pop_{H^2}:L^2(\T)\to H^2$ and
$\Pop_{H^2_-}:L^2(\T)\to H^2_-$. These are associated with the Szeg\H{o}
integral kernel:
\[
\Pop_{H^2}f(z)=\int_\T\frac{f(\zeta)}{1-z\bar\zeta}\diffs(\zeta),\qquad
z\in\D,
\]
and
\[
\Pop_{H^2_-}f(z)=\int_\T\frac{zf(\zeta)}{z-\zeta}\diffs(\zeta),\qquad
z\in\D_\e.
\]
We will also be interested in the subspace $H^2_{-,0}$ of $H^2_-$ consisting
of all functions that vanish at infinity (or equivalently, have average $0$
on the unit circle). The associated projection is
 \[
\Pop_{H^2_{-,0}}f(z)=\int_\T\frac{\zeta f(\zeta)}{z-\zeta}\diffs(\zeta),\qquad
z\in\D_\e.
\] 
It is clear from the above concrete formul\ae{} that the Herglotz 
transform $\Hop_{\D_\e}$ can be expressed in terms of projections:
$\Hop_{\D_\e}=\Pop_{H^2_-}+\Pop_{H^2_{-,0}}$. 
For an $L^\infty(\T)$-function $\Theta$, we define the (exterior) 
Toeplitz operator
$\mathbf{T}_\Theta:H^2_{-}\to H^2_{-}$ by
$$
\mathbf{T}_\Theta f=\Pop_{H^2_{-}}[\Theta f],\qquad f\in H^2_{-}.
$$
The nullspace (kernel) of this operator consists of all solutions in
$H^2_{-}$ to $\mathbf{T}_\Theta f=0$. Assuming that $\Theta$ is nonzero
almost everywhere on the circle $\T$, it follows that the condition
that $f$ belongs to the nullspace is 
equivalent to $f\in H^2_{-}\cap\Theta^{-1} H^2_0$, where $H^2_0$ consists of the 
functions in $H^2$ with mean $0$. If we implicitly define the function 
$\vartheta$ 
by $\Theta(z)=z\vartheta(z)$, we may rephrase this condition as
\begin{equation}\label{eq:toeplitz-hom}
f\in H^2_{-}\cap \vartheta^{-1}H^2,
\end{equation}
which we refer to as a {\em homogeneous (exterior) Toeplitz kernel condition}. 
For a function $F$ in the space $L^2(\T)$, we also consider the 
related condition
\begin{equation}\label{eq:toeplitz-inhom}
f\in H^2_{-}\cap \vartheta^{-1}(-F+H^2),
\end{equation}
which we refer to as an {\em inhomogeneous Toeplitz kernel condition}.
In terms of Toeplitz operators, this condition may be
written as $\mathbf{T}_{z\vartheta}f+\Pop_{H^2_{-}}[zF]=0$. The following 
proposition
provides the structure of solutions to the homogeneous and inhomogeneous 
Toeplitz kernel conditions for sufficiently regular symbols $\vartheta$.

\begin{prop}\label{prop:toeplitz-ker}
Suppose that $\vartheta$ can be written in the form $\vartheta=\e^{u+\bar v}$,
where $u$ and $v$ are in $H^\infty$, and let $F$ be a function in $L^2(\T)$. 
Then $f$ solves
$$
f\in H^2_{-}\cap \vartheta^{-1}(-F+H^2)
$$
if and only if
$$
f=C\e^{-\bar v}-\e^{-\bar v}\Pop_{H^2_{-,0}}[\e^{-u}F],
$$
for some constant $C$.
\end{prop}

\begin{proof}
That $f\in H^2_{-}\cap \vartheta^{-1}(-F+H^2)$ is equivalent to having
\begin{equation}
\label{eq-Tkcond1.01}
\e^{\bar v}f\in \e^{\bar v}H^2_{-}\cap (-\e^{-u}F+\e^{-u}H^2)=
H^2_{-}\cap (-\e^{-u}F+H^2).
\end{equation}
Since $\e^{\bar v}f\in H^2_{-}$, an application of the projection 
$\Pop_{H^2_{-,0}}$ gives
$$
\Pop_{H^2_{-,0}}[\e^{\bar v}f]=\e^{\bar v}f-C
$$
for some constant $C$. On the other hand, 
since $\e^{\bar v}f\in -e^{-u}F+H^2$ holds by \eqref{eq-Tkcond1.01},
it is immediate that
$$
\Pop_{H^2_{-,0}}[\e^{\bar v}f]= -\Pop_{H^2_{-,0}}[\e^{-u}F],
$$
since $H^2$ projects to $\{0\}$. It follows that
$$
\e^{\bar v}f=C+\Pop_{H^2_{-,0}}[\e^{\bar v}f]=C-\Pop_{H^2_{-,0}}[\e^{-u}F],
$$
as claimed.
\end{proof}

\begin{rem}
The Toeplitz kernel equation \eqref{eq-Tkcond1.01} may be viewed as a scalar
Riemann-Hilbert problem with jump from the inside $\D$ to the outside $\D_\e$
equal to $\e^{-u}F$. Later, we will use the conformal mapping from the 
complement of the droplet $\calS_\tau^c$ to the exterior disk $\D_\e$, 
and the interpretation of the
Toeplitz kernel equation in that context is as a scalar Riemann-Hilbert problem
on the Schottky double of $\calS_\tau^c$. 
\end{rem}

\subsection{Steepest descent analysis}
For our computational algorithm in Section~\ref{sect:alg},
we will need the following result (\cite{Horm}, p. 220, 
Theorem 7.7.5). The formulation requires some notation. For an open subset 
$\Omega$ of $\R$, we let $C^k(\Omega)$ denote the space of $k$ times 
differentiable 
functions on $\Omega$, and for a compact subset $K$ of $\R$, we let $C^k_0(K)$ 
denote the space $k$ times differentiable, compactly supported functions on 
$\R$ whose support is contained in $K$. The norm in the space $C^k(\Omega)$
is defined as
$$
\lVert u\rVert_{C^k(\Omega)}=
\sum_{j=0}^k\lVert u^{(j)}\rVert_{L^\infty(\Omega)},
$$
and the norm in $C^k_0(K)$ is analogously defined.

\begin{prop}\label{prop:steep}
Let $\mathrm{K}\subset\R$ be a compact interval, 
$\Omega$ an open neighborhood of 
$\mathrm{K}$, $x_0$ an interior point of $K$, 
and $k$ a positive integer. 
If $u\in C^{2k}_0(\mathrm{K})$, 
$V\in C^{3k+1}(\Omega)$ and $V\ge0$ in $\Omega$, $V'(x_0)=0$,
$V''(x_0)>0$, and $V'\ne0$ in $\mathrm{K}\setminus\{x_0\}$, then, for 
$\omega>0$, we have
\begin{equation}\label{eq:steep}
\bigg|
\e^{\omega V(x_0)}\int_{\mathrm{K}}u(x)\e^{-\omega V(x)}\diff x-
\bigg(\frac{2\pi}{\omega V''(x_0)}\bigg)^{\frac12}
\sum_{j=0}^{k-1}\omega^{-j}
\Lop_j u(x_0)\bigg|
\le C\omega^{-k}
\|u\|_{C^{2k}(\mathrm{K})}.
\end{equation}
Here, $C$ is bounded when $V$ stays in a 
bounded set in $C^{3k+1}(\Omega)$, 
and $|x-x_0|/|V'(x)|$ has a uniform bound. With 
\[
W_{x_0}(x):=V(x)-V(x_0)-\frac{1}{2}(x-x_0)^2V''(x_0),
\]
we have 
\[
\Lop_ju(x):= \sum_{(k,l):l-k=j,\,2l\ge 3k}\frac{(-1)^{k}2^{-l}}
{k!l![V''(x_0)]^l}\partial_x^{2l}(W^k_{x_0}u)(x).
\]
\end{prop}

In the definition of the above differential operator $\Lop_j$, it is 
implicit that the summation takes place over nonnegative integers 
$k$ and $l$. The differential operator \eqref{eq:Lop} 
mentioned in connection with 
Theorem~\ref{thm:main-coeff} is obtained from 
this formula.

The following proposition is tailored to our needs, based on Proposition
~\ref{prop:steep}.

\begin{prop}\label{prop:steep2}
Let three reals $\rho_0, \rho_1,\rho_2$ be given, with 
$0<\rho_0<1<\rho_1<\rho_2$. 
Assume that $V:[\rho_0,\infty)\to\R$ is $C^{3k+1}$-smooth, and that $V$ 
has a unique
minimum at $1$, with $V(1)=V'(1)=0$. Suppose furthermore that we have 

\noindent {\rm (a)}\; the convexity bound $V''\ge \alpha$ on 
$(\rho_0,\rho_2)$ for some 
real $\alpha>0$, 

\noindent {\rm (b)} and that $V$ has a bound from below of the form 
$V(x)\ge \vartheta\log x$ on the interval $[\rho_1,\infty)$, for some 
real constant $\vartheta>0$. 

\noindent If the function $u:(\rho_0,\infty)\to\C$ is bounded 
and continuous throughout, and in addition $u$ is $C^{2k}$-smooth on the 
interval
$[0,\rho_2]$ and vanishes on $[0,\rho_0]$,
then we have
$$
\int_{\rho_0}^\infty u(x)\e^{-\omega V(x)}\diff x=\bigg(\frac{2\pi}
{\omega V''(1)}\bigg)^{\frac12}\sum_{j=0}^{k-1}\omega^{-j}\Lop_j [u](1) +E,
$$
where the error term $E=E(\omega, k,u,\vartheta, \rho_0,\rho_1,\rho_2)$ 
enjoys the bound
$$
\lvert E\rvert \le C_1\omega^{-k}\lVert u\rVert_{C^{2k}([\rho_0,\rho_2])}+
\lVert u\rVert_{L^\infty([\rho_1,\infty))}\rho_1^{-\omega\vartheta+1},
$$
provided that $\omega>\frac{2}{\vartheta}$, where $C_1$ remains uniformly 
bounded when $V$ stays in a bounded
set of $C^{3k+1}([\rho_0,\rho_2])$.
\end{prop}

\begin{proof}[Sketch of proof]
Let $\chi$ be a smooth cut-off function with 
$0\le \chi\le 1$ throughout, which equals
$1$ on the interval $[\rho_0, \rho_1]$, and vanishes on $[\rho_2,\infty)$. 
We use the cut-off function to split the integral
$$
\int_{\rho_0}^\infty u(x)\e^{-\omega V(x)}\diff x
=\int_{\rho_0}^{\rho_2}\chi(x) u(x)
\e^{-\omega V(x)}\diff x
+\int_{\rho_1}^{\infty}(1-\chi(x))u(x)\e^{-\omega V(x)}\diff x.
$$
The first integral gives the main contribution, which is estimated using 
Proposition~\ref{prop:steep}. The other two integrals are estimated using 
the given bounds from below on $V$. The details are omitted.
\end{proof}

\section{Existence of an asymptotic expansion}
\label{s:existence-expansion}
\subsection{An \texorpdfstring{$L^2$}{L2}-version of the main theorem}
\label{ss:L2-expansion}
The proof of Theorem~\ref{thm:main-pw} goes via an expansion valid 
in weighted $L^2$-space, which is of independent interest. 
Modulo the key lemma (Lemma~\ref{lem:main-flow}) concerning the orthogonal
foliation flow, we first obtain
the weighted $L^2$-expansion, and then obtain Theorem~\ref{thm:main-pw} 
as a consequence.
The proof of the key lemma is deferred to Section~\ref{sec:flow}. 

For two sets $\mathcal{E},\calF\subset \C$ we define
the distance between them as
\[
  \mathrm{dist}_\C(\mathcal{E},\calF)
  =\inf_{z\in\mathcal{E},\,w\in\calF}|z-w|. 
\]
We shall need the following notion. 
\begin{defn}\label{def:intermediate} If $\calK$ and $\calS$ are compact sets in 
the plane with $\calK\subset\calS$ and 
\[
\mathrm{dist}_\C(\calK, \calS^c)= \varepsilon,
\]
we say that a compact set $\mathcal{X}$ is {\em intermediate} between 
$\calK$ and $\calS$ if $\calK\subset\mathcal{X}\subset\calS$ with
\[
\mathrm{dist}_\C(\calK, \mathcal{X}^c)\ge \frac{\varepsilon}{1000}\quad
\text{and}\quad \mathrm{dist}_\C(\mathcal{X}, \calS^c)\ge 
\frac{\varepsilon}{1000}.
\]
\end{defn}

We recall from the discussion following Definition~\ref{def:adm} 
the notation $I_{\epsilon_0}=[1-\epsilon_0, 1+\epsilon_0]$, where 
$\epsilon_0$ is fixed
and positive, with the property that the curves $\partial\calS_\tau$ form 
a smooth flow of simple loops for 
$\tau\in I_{\epsilon_0}$. 
\begin{thm}\label{main}
Assume that $Q$ is $1$-admissible, and fix the precision parameter 
$\kappa\in\mathbb{N}$. 
Then, for each $\tau\in I_{\epsilon_0}$ there exists a compact subset
$\calK_{\tau}\subset\calS_\tau$ with 
$\mathrm{dist}_\C(\calK_\tau,\partial\calS_\tau)\ge\varepsilon$ 
for some positive real number $\varepsilon$, such that the following holds. 
On the complement $\calK_\tau^c$, there are 
bounded holomorphic functions $\calB_{\tau, j}$ such that the 
associated function
$$
F_{m,n}^{\langle \kappa\rangle}=m^{\frac{1}{4}}
\sqrt{\phi_\tau'}\,[\phi_\tau]^n
\e^{m\mathcal{Q}_\tau}
\sum_{j=0}^{\kappa}m^{-j}\calB_{\tau, j},
$$
approximates well the normalized orthogonal polynomials $P_{m,n}$ in the 
sense that we have the norm control
$$
\big\lVert P_{m,n}-\chi_{\tau, 0}F_{m,n}^{\langle \kappa\rangle}\big\rVert_{2mQ}
=\Ordo(m^{-\kappa-1})
$$
as $n,m\to\infty$ while $\tau=\frac{n}{m}\in I_{\epsilon_0}$.
Here, $\chi_{\tau, 0}$ denotes a smooth cut-off function with 
$0\le\chi_{\tau, 0}\le 1$ and uniformly bounded gradient. In addition the function
$\chi_{\tau, 0}$ vanishes on $\calK_\tau$,
and equals $1$ on the set
$\mathcal{X}_\tau^c$ where $\mathcal{X}_\tau$ is an intermediate set
between $\calK_\tau$ and $\calS_\tau$.  In the above estimate,
the implicit constant is uniform for $\tau\in I_{\epsilon_0}$.
\end{thm}

In the above theorem, the products $\chi_{\tau, 0} F_{m,n}^{\langle \kappa\rangle}$
are understood to vanish on the set $\calK_\tau$, where 
$F_{m,n}^{\langle \kappa\rangle}$ may be undefined. 

\begin{rem}\label{rem:im-set}
{\rm (a)\;} By inserting a further family 
$\mathcal{X}_\tau'$ of intermediate sets
between $\calK_\tau$ and $\calS_\tau$ such that
$\mathcal{X}_\tau$ is intermediate 
between $\mathcal{X}_\tau'$ and $\calS_\tau$,
we can make sure that the cut-off 
function $\chi_{\tau, 0}$ vanishes on 
$\mathcal{X}_\tau'$ (and not just on $\calK_\tau$). 
We mention that the compact sets $\calK_\tau$, 
$\mathcal{X}_\tau'$ and $\mathcal{X}_\tau$ may be 
obtained, e.g., as the complements of the 
conformal images under $\phi_\tau^{-1}$ of the exterior disks 
$\D_\e(0,\rho)$ with $\rho=\rho_0$, $\rho_{0,1}$ and $\rho_{0,2}$, 
where $0<\rho_0<\rho_{0,1}<\rho_{0,2}<1$. 
As for the intermediate property of Definition~\ref{def:intermediate}
regarding the sets $\calK_\tau$, $\mathcal{X}_\tau'$, $\mathcal{X}_\tau$, 
and $\calS_\tau$, this is a little subtle, and depends on making a correct 
choice of the parameters $\rho_0$, $\rho_{0,1}$, and $\rho_{0,2}$.
At our disposal, we have the Koebe distortion theorem 
and the fact that $\log (\phi_\tau^{-1})'$ is a Lipschitz function in the
hyperbolic metric with known Lipschitz constant 
(see, e.g.,\ Corollary 1.4 and Proposition~1.2 in \cite{Pommerenke}, 
respectively). We omit the necessary details.

\noindent {\rm (b)} Without loss of generality, we may assume that the 
cut-off 
function $\chi_{\tau, 0}$ is uniformly smooth in the sense that for 
any fixed positive
integer $k$ the $C^k(\C)$-norm of $\chi_{\tau, 0}$ is uniformly bounded for 
$\tau\in I_{\epsilon_0}$.

\noindent {\rm (c)} Our method of  proof involves Toeplitz kernel
problems and the construction of an approximate orthogonal foliation flow
of loops. The underlying idea is inspired by an approach to the local
expansion of Bergman kernels, which involves a flow of loops emanating
from the point of expansion \cite{GHS}.  
\end{rem}

\subsection{Introduction of quasipolynomials}\label{ss:quasipol} 
We turn to the \emph{approximate orthogonal 
quasipolynomials} $F_{m,n}$, by which we mean certain functions which 
behave like 
orthogonal polynomials with respect to the measure $\e^{-2mQ}\diffA$, 
in a sense specified below. 
Let $\calK_{\tau}$ be an appropriately chosen compact subset of the droplet
$\calS_\tau$, which lies at a fixed positive distance from $\partial\calS_\tau$. 
Moreover, we require that the conformal mapping 
$\phi_\tau:\calS_\tau\to\D_\e$ extends to a (surjective) conformal mapping
$$
\phi_\tau\colon\calK_{\tau}^c\to\D_\e(0,\rho_0), \qquad \tau\in I_{\epsilon_0},
$$ 
for some $\rho_0$ with $0<\rho_{0,0}<\rho_0<1$, where we recall that $\rho_{0,0}$
was defined in the discussion preceding Theorem~\ref{thm:main-pw}.
In what follows, we will disregard the behavior on the compact set 
$\calK_{\tau}$. 
We will justify this a posteriori, 
using $\bar\partial$-methods.

\begin{defn}\label{def:quasipol}
We say that a function $F$ is a {\em quasipolynomial} on $\calK_\tau^c$ 
of degree $n$ if it is defined and holomorphic on $\calK_{\tau}^c$, 
with polynomial growth near infinity: 
$\lvert F(z)\rvert \asymp \lvert z\rvert^n$ as 
$|z|\to\infty$.
\end{defn}

In the context of this definition, a quasipolynomial $F$ of degree $n$ has
$F(z)=a z^n+\Ordo(\lvert z\rvert^{n-1})$ near infinity, for some complex number 
$a\neq 0$. We refer to the number $a$
as the {\em leading coefficient} of the quasipolynomial $F$.

We now fix a positive integer $\kappa$, which we think of as an 
accuracy parameter.
Moreover, we denote by $\chi_{\tau, 0}$ a smooth cut-off function that 
vanishes on $\mathcal{X}_\tau'$ and equals $1$ on $\mathcal{X}_\tau^c$, where
$\mathcal{X}_\tau'$ denotes an intermediate set between $\calK_\tau$
and $\calS_\tau$, while $\mathcal{X}_\tau$ is an intermediate set
between $\mathcal{X}_\tau'$ and $\calS_\tau$. In addition, we 
shall require that the $C^{2(\kappa+1)}$-norm of $\chi_{\tau, 0}$ remains uniformly 
bounded for $\tau\in I_{\epsilon_0}$.
 
\begin{defn}\label{def:quasi}
We say that a sequence $\{F_{m,n}\}_{m,n}$ of quasipolynomials of degree $n$ 
on $\calK_\tau^c$ is {\em normalized and approximately orthogonal 
(of accuracy $\kappa$)}
if the following asymptotic conditions (i)-(iii) are met as $m\to\infty$ while 
$\tau=\tfrac{n}{m}\in I_{\epsilon_0}$:

\noindent (i) we have the approximate orthogonality
\[
\forall p\in\mathrm{Pol}_n:\quad \int_{\C}\chi_{\tau, 0}
F_{m,n}(z)\overline{p(z)}\,\e^{-2mQ(z)}
\dA(z)=\Ordo\left(m^{-\kappa-\frac13}\|p\|_{2mQ}\right),
\]
\noindent (ii) the quasipolynomials $F_{m,n}$ have approximately unit norm,
\[
\int_{\C} \chi_{\tau, 0}^2(z)|F_{m,n}(z)|^2\e^{-2mQ(z)}\dA(z)
=1+\Ordo(m^{-\kappa-\frac13}),
\]
\noindent (iii) and the quasipolynomial $F_{m,n}$ has leading 
coefficient $c_{m,n}$
at infinity
which is approximately real and positive, in the sense that
$$
\frac{\Im c_{m,n}}{\Re c_{m,n}}=\Ordo(m^{-\kappa-\frac{1}{12}})
$$
where all the implied constants are uniform.
\end{defn}

In terms of the above definition, Theorem~\ref{main} implies in 
particular that
$F_{m,n}^{\langle \kappa\rangle}$ is a sequence of approximately orthogonal
quasipolynomials with accuracy $\kappa$. The fraction $\frac13$ which appears
in the definition is convenient in our calculations. The concept would 
be meaningful even if this number were replaced by e.g. $\frac15$. 

\medskip

\subsection{The renormalizing ansatz}\label{ss:reduct}
Since $Q$ is assumed $1$-admissible, the curves $\Gamma:=\partial\calS_\tau$ 
remain real-analytically smooth and simple for 
$\tau\in I_{\epsilon_0}= [1-\epsilon_0,1+\epsilon_0]$. 
In view of the requirement that $\calK_{\tau,0}\subset\calK_\tau$, the 
functions $Q^\circledast_\tau$ and $\breve{Q}_\tau$
are harmonic while $\calQ_\tau$ is holomorphic in the domain $\calK_{\tau}^c$
(see Definition~\ref{def:Q-funct}).
We define the operator $\Vop_{m,n}$ by
\begin{equation}\label{eq:Vop}
\Vop_{m,n}f(z):= \phi'_\tau(z)\,[\phi_\tau(z)]^n\e^{m\calQ_\tau(z)}\,
(f\circ\phi_\tau)(z),\qquad \tau=\frac{n}{m}.
\end{equation}
If $f,g$ are well-defined in $\D_\e(0,\rho_0)$, then 
$\Vop_{m,n}f$ and $\Vop_{m,n}g$ are 
well-defined in $\calK_{\tau}^c$. We observe that by a change-of-variables,
\begin{multline}\label{eq:Visom}
\int_{\calK_\tau^c}\Vop_{m,n}f\,\overline{\Vop_{m,n}g}\,\e^{-2mQ}\diffA=
\int_{\calK_\tau^c} (f\circ\phi_\tau) \overline{(g\circ\phi_\tau)}
\,\e^{-2m(Q-\tau\log\lvert \phi_\tau\rvert-\Re\calQ_\tau)}
\lvert \phi_\tau'\rvert^2\diffA
\\
=\int_{\D_\e(0,\rho_0)} f\,\overline{g}\,
\e^{-2mR_\tau}\diffA,
\end{multline}
where we write 
\[
R_\tau:=(Q-\breve Q_\tau)\circ\phi^{-1}_{\tau},
\] 
and the first equality holds by \eqref{eq:Q-breve-circledast}.

The function $R_\tau$ is a central object in our analysis, 
and we turn to some of its basic properties.
\begin{prop}\label{prop:R}
The function $R_\tau$ is defined on $\D_\e(0,\rho_0)$, and is real-analytic 
in a neighborhood of $\T$. Moreover, near the unit circle $R_\tau$ satisfies
\[
R_\tau(r\e^{\imag\theta})
=2\hDelta R_\tau(\e^{\imag\theta})(1-r)^2+\Ordo((1-r)^3), 
\qquad r\to 1,
\]
where the implied constant is uniform for $\e^{\imag\theta}\in\T$ and 
$\tau\in I_{\epsilon_0}$. Furthermore, $R_\tau$ has the growth bound from below
$$
R_\tau(z)\ge \vartheta\log\lvert z\rvert,\qquad z\in\D_\e(0,\rho_1),
$$
for some real parameters $\vartheta>0$ and $\rho_1>1$, which do not depend
on $\tau\in I_{\epsilon_0}$.
\end{prop}

\begin{rem}
\label{rem:growth-R} 
In particular, $R_\tau(z)\asymp(1-|z|)^2$ near the unit circle.
Indeed, since $\breve{Q}_\tau$ is harmonic on $\calK_\tau^c$, we find that
$$
\hDelta R_\tau=\hDelta (Q-\breve{Q}_\tau)\circ\phi_\tau^{-1}=\lvert 
(\phi_\tau^{-1})'\rvert^2\,(\hDelta Q)\circ\phi_\tau^{-1} ,
$$
which shows that near the circle $\T$, we have uniform bound of
$\hDelta R_\tau$ from below by a positive constant. As a consequence,
the same holds for 
$\partial^2_r R_\tau(r\e^{\imag\theta})$ for $r$ close to $1$, which will be useful 
in the context of Proposition~\ref{prop:steep2}.
\end{rem}

\begin{proof}[Sketch of proof]
The assertion on the local behavior near the circle $\T$ results from an 
application of Taylor's formula, using that along the 
boundary $\partial\calS_\tau$ we have $Q=\breve Q_\tau$, 
$\nabla Q=\nabla\breve{Q}_\tau$ while
$$
\partial^2_{\rm n}(Q-\breve{Q}_\tau)=(\partial^2_{\rm n}+\partial^2_{\rm t})
(Q-\breve{Q}_\tau)=4\hDelta Q.
$$
Here,  $\partial_{\rm n}$ and $\partial_{\rm t}$ denote 
the normal and tangential derivatives, respectively.
We turn to the global estimate from below on $R_\tau$. By the assumption 
\eqref{eq-Qcond} with $\tau=1$ on the growth of $Q$ near infinity, and 
the growth
control 
$$
\breve{Q}_\tau(z)=\checkQ_\tau(z)=\tau\log\lvert z\rvert+\Ordo(1),\qquad 
\text{as}\;\;\lvert z\rvert\to\infty,
$$
it follows that
$$
\liminf_{z\to\infty}\frac{(Q-\breve{Q}_\tau)(z)}{\log\lvert z\rvert}>0
$$
for $\tau\in I_{\epsilon_0}$, provided that $\epsilon_0$ is small enough.
Since $\lvert \phi_\tau^{-1}(z)\rvert\asymp \lvert z\rvert$ near infinity, 
we see that
$$
\lim_{\lvert z\rvert\to\infty}\frac{R_\tau(z)}{\log\lvert z\rvert}>0.
$$
There is no point in $\D_e$ where $R_\tau$ vanishes, since the coincidence set 
(where $\checkQ_\tau$ and $Q$ coincide) equals $\calS_\tau$ 
(see Definition~\ref{def:adm}). We may conclude that the ratio
$\frac{R_\tau(z)}{\log\lvert z\rvert}$
is bounded below by a positive constant $\vartheta$ on the exterior disk
$\D_\e(0,\rho_1)$. A careful analysis of this argument shows that we may assume 
that $\vartheta$ does not depend on $\tau$, as long as $\tau\in I_{\epsilon_0}$.
\end{proof}

Informally, Proposition~\ref{prop:R} tells us that 
near the unit circle, the function 
$\e^{-2mR_\tau}$ may be thought of as a 
Gaussian wave around the unit circle $\T$. 

We return to the operator $\Vop_{m,n}$, defined in \eqref{eq:Vop}. It 
renormalizes the weight, and transports 
holomorphic functions in the exterior disk $\D_\e(0,\rho_0)$ 
to holomorphic functions in the region $\calK_\tau^c$.
In the sequel, we will refer to $\Vop_{m,n}$ as {\em the canonical 
positioning operator}. Its basic properties are summarized in the following 
proposition, which involves the spaces 
$L^2_\phi(\mathcal{X}^c)$ and $A^2_\phi(\mathcal{X})$, 
as well as the restricted growth subspaces 
$L^2_{\phi,k}(\mathcal{X}^c)$ and $A^2_{\phi,k}(\mathcal{X}^c)$, 
all defined in Subsection~\ref{ss:notation}. 
Below, these appear for various choices of the weight $\phi$,
the parameter $k$, and the compact set $\mathcal{X}$.

\begin{prop}\label{prop:Vop}
The canonical positioning operator $\Vop_{m,n}$ is an isometric isomorphism
$L^2_{2mR_\tau}(\D_\e(0,\rho_0))\to L^{2}_{2mQ}(\calK_{\tau}^c)$,
and the inverse operator
is given by
\begin{equation*}
\Vop_{m,n}^{-1}g(z)=z^{-n}[\phi_\tau^{-1}]'(z)\,\e^{-m(\calQ_\tau\circ
\phi_\tau^{-1})(z)}
(g\circ\phi_\tau^{-1})(z),\qquad g\in L^{2}_{2mQ}(\calK_{\tau}^c).
\end{equation*}
Moreover, the operator $\Vop_{m,n}$ preserves holomorphicity, and in addition, 
it maps the 
subspace $A^2_{2mR_\tau,0}(\D_\e(0,\rho_0))$ onto 
$A^{2}_{2mQ,n}(\calK_{\tau}^c)$.
\end{prop}

\begin{proof}
As direct consequence of the \eqref{eq:Visom}, we see that 
$L^2_{2mR_\tau}(\D_\e(0,\rho_0))$ is mapped isometrically 
into $L^{2}_{2mQ}(\calK_{\tau}^c)$, and moreover if $\Vop_{m,n}^{-1}$ 
is given by the above formula, we see that it is actually the 
inverse to $\Vop_{m,n}$. By definition, $\Vop_{m,n}f$ is holomorphic in 
$\calK_\tau^c$, if $f$ is 
holomorphic in $\D_\e(0,\rho_0)$.  It follows that $\Vop_{m,n}$ is actually 
an isometric isomorphism $A^2_{2mR_\tau}(\D_\e(0,\rho_0))\to A^{2}_{2mQ}
(\calK_{\tau}^c)$. It remains to note that $\Vop_{m,n}$ maps bijectively 
$$ 
A^2_{2mR_\tau,0}(\D_\e(0,\rho_0))\to A^{2}_{2mQ,n}(\calK_{\tau}^c),
$$
which is a direct consequence of the fact that $\lvert \phi_\tau(z)\rvert\asymp 
\lvert z\rvert$ as $\lvert z\rvert\to\infty$.
\end{proof}

\subsection{The orthogonal foliation flow}
\label{ss:prel-flow}
We will obtain our main result (Theorem~\ref{main})
as a consequence of the 
existence of 
what we call the {\em approximate orthogonal foliation flow of simple loops} 
$\Gamma_{m,n,t}$, parameterized by the parameter $t$. For a brief sketch 
of the intuition that lies behind the construction of this
flow of curves, we refer to the discussion in
Subsection~\ref{ss:idea-sketch} above.

We recall from Subsection~\ref{ss:notation} that a conformal mapping
$\psi$ of the exterior disk $\D_\e$ onto a domain containing the point
at infinity is said to be {\em orthostatic} if it maps $\infty$ to
$\infty$, and has $\psi'(\infty)>0$.
Given a smooth family $\psi_t$ of orthostatic conformal mappings on the
exterior disk, indexed by a real parameter $t$ close to $0$, such that 
the image domains $\Omega_t:=\psi_t(\D_\e)$ increase with $t$, 
we put $\Gamma_t=\psi_t(\T)$ and denote by 
$\calD=\bigcup_t\Gamma_t$ the region covered by the flow.
We may form the {\em foliation mapping} $\Psi$ by the formula
$$
\Psi(z)=\psi_{1-\lvert z\rvert}\Big(\frac{z}{\lvert z\rvert}\Big),
$$
for $z$ in some annulus $\mathbb{A}$ containing the unit circle. The foliation 
mapping $\Psi$ maps $\mathbb{A}$ onto the domain $\mathcal{D}$ covered by 
the boundaries.
Moreover, the Jacobian $J_\Psi$ of the foliation mapping is given by
\begin{equation}\label{eq:jac}
J_\Psi(r\zeta)=-\frac{1}{r}\Re\big\{\bar \zeta\partial_t\psi_t
\big(\zeta\big)
\overline{\psi'_{t}\big(\zeta\big)}
\big\}\big\vert_{t=1-r},\qquad \zeta\in\T,
\end{equation}
for $r$ near $1$.
We may integrate over a flow encoded by a 
foliation mapping $\Psi$ as follows:
If we denote by $\mathbb{A}_\epsilon$ the annulus 
$\mathbb{A}_\epsilon=\D(0,1+\epsilon)\setminus\bar{\D}(0,1-\epsilon)$, 
we have for integrable $f$,
\begin{align}\label{eq:int-flow}
\int_{\Psi(\mathbb{A}_\epsilon)}f\diffA&=
\int_{\mathbb{A}_\epsilon}f\circ\Psi\;J_{\Psi}\diffA\\
\nonumber
&=2\int_{-\epsilon}^{\epsilon}\int_{\T}f\circ\psi_t(\zeta) (1-t)
J_\Psi\big((1-t)\zeta\big)
\diffs(\zeta)\diff t.
\end{align}

The existence of the foliation flow may be phrased as follows. We call
the relation \eqref{eq:main-flow} below the 
\emph{master equation for the orthogonal foliation flow}.
For convenience of notation, let $\delta_m$ be the number
$$
\delta_m:=m^{-1/2}\log m.
$$

\begin{lem}\label{lem:main-flow}
Fix the precision parameter $\kappa$ to be a positive integer. 
For $\tau=\frac{n}{m}\in I_{\epsilon_0}$, 
there exist $0<\rho_0<1$ and
bounded holomorphic functions $B_{\tau, j}$ on $\D_\e(0,\rho_0)$
for $j=0,\ldots,\kappa$, such that the the following properties hold.
The function $B_{\tau,0}$ is bounded away 
from zero with $B_{\tau,0}(\infty)>0$, 
while for $j=1,\ldots,\kappa$ we have $\Im B_{\tau,j}(\infty)=0$. 
Moreover, there exists a smooth family of orthostatic conformal mappings 
$\{\psi_{m,n,t}\}_{m,n,t}$ on $\bar{\D}_\e$, such that if we write
$f_{m,n}^{\langle \kappa\rangle}=\sum_{j=0}^{\kappa} m^{-j}B_{\tau,j}$, we have 
that
\begin{multline}
m^{\frac12}\big\lvert 
f_{m,n}^{\langle \kappa\rangle}\circ\psi_{m,n,t}(\zeta)\big\rvert^2
\e^{-2m(R_\tau\circ\psi_{m,n,t})(\zeta)} \,(1-t)J_{\Psi_{m,n}}((1-t)\zeta)
\\
=\frac{m^{\frac12}}{(4\pi)^{\frac12}}\e^{-mt^2}
\big(1 +\Ordo\big(m^{-\kappa-\frac13}\big)\big), \qquad \zeta\in\T,
\label{eq:main-flow}
\end{multline}
provided that $|t|\le \delta_m$.
Here, the implicit constant 
is uniform in $\tau\in I_{\epsilon_0}$. 
Moreover, if 
$\calD_{m,n}$ denotes the union 
$\calD_{m,n}=\bigcup_{\lvert t\rvert\leq\delta_m}\psi_{m,n,t}(\T)$, 
then $\mathrm{dist}_{\C}(\calD_{m,n}^c,\T)
\geq c_0\delta_m$ for some positive constant $c_0$.
\end{lem}

\begin{rem}
The equation \eqref{eq:main-flow} may be understood as an approximate
{\em weighted Polu\-bar\-in\-ova-Galin equation} with weight 
$\lvert f_{m,n}^{\langle\kappa\rangle}\rvert^2\e^{-2mR_\tau}$, and variable 
speed of expansion. 
Indeed, we should compare with equation (6.11) in \cite{HS}, which states 
in a similar context that along concentric circles,
$$
J_\Psi=\omega^{-1}\circ{\Psi},\qquad 
$$
where $\Psi$ is a foliation mapping, and $\omega$ denotes a weight. In 
comparison, our factor $(4\pi)^{-\frac12}\e^{-mt^2}$ appears as 
consequence of the variable speed.
\end{rem}

In what follows, we take this key lemma for granted. 
The proof is supplied in Section~\ref{sec:flow}.

\subsection{The \texorpdfstring{$L^2$}{L2}-expansion
for quasipolynomials}
We first find a sequence of approximately orthogonal quasipolynomials with 
an asymptotic expansion.

\begin{lem}\label{lem:quasi}
Let $\kappa\in\mathbb{N}$ be given and let 
$f_{m,n}^{\langle \kappa\rangle}=\sum_{j=0}^{\kappa}m^{-j}B_{\tau, j}(z)$
be the functions defined in Lemma~\ref{lem:main-flow}. 
Then the functions
\begin{equation}\label{eq:Fnm-def}
F_{m,n}^{\langle\kappa\rangle}(z)
=m^{\frac14}\Vop_{m,n}[f_{m,n}^{\langle \kappa\rangle}]
=m^{\frac14}\phi_\tau'(z)[\phi_\tau(z)]^n
\e^{m\mathcal{Q}_\tau(z)}(f_{m,n}^{\langle\kappa\rangle}\circ\phi_\tau)(z)
\end{equation}
constitute a family of approximately orthogonal quasipolynomials to
accuracy $\kappa$ in the sense of Definition~\ref{def:quasi}.
\end{lem}

\begin{proof}
We denote by $\chi_{\tau, 1}$ a radial
smooth cut-off function which vanishes on $\D(0,\rho_{0,1})$ and equals 1 on 
$\D_\e(0,\rho_{0,2})$, where the parameters $0<\rho_0<\rho_{0,1}<\rho_{0,2}<1$ 
are chosen in accordance 
with Remark~\ref{rem:im-set}.
The cut-off function $\chi_{\tau,0}$ is then given by 
$\chi_{\tau,0}=\chi_{\tau,1}\circ\phi_\tau$. 
The intermediate sets $\mathcal{X}_\tau'$ and
$\mathcal{X}_\tau$ are given as the complements 
of the conformal images of $\D_\e(0,\rho_{0,1})$ 
and $\D_\e(0,\rho_{0,2})$ under $\phi_\tau$, respectively.

By Lemma~\ref{lem:main-flow}, the functions $f_{m,n}^{\langle \kappa\rangle}$ 
are bounded and holomorphic on the exterior disk $\D_\e(0,\rho_0)$, 
with $f_{m,n}^{\langle \kappa\rangle}(\infty)>0$. 
As the leading term $B_{\tau, 0}$ is bounded away from $0$ on 
$\D_\e(0,\rho_0)$, it follows that for large enough $m$, the same can be
said for $f_{m,n}^{\langle \kappa\rangle}$. 
In view of this, the functions $F_{m,n}^{\langle \kappa\rangle}$ given by
\eqref{eq:Fnm-def} are 
quasipolynomials of order $n$ on 
$\calK_\tau^c:=\phi_\tau^{-1}(\D_\e(0,\rho_0))$ 
in the sense of Definition~\ref{def:quasipol}. 

It remains to verify the properties (i), (ii), and 
(iii) of Definition~\ref{def:quasi}.
To this end, we recall the definition of the domain $\calD_{m,n}$ from 
Lemma~\ref{lem:main-flow},
which is a certain closed neighborhood of the unit circle which arises from 
our orthogonal foliation flow. 
We recall that 
$$
\mathrm{dist}_\C(\calD_{m,n}^c,\T))\ge c_0\delta_m
$$ 
holds for some fixed  
constant $c_0>0$, where $\delta_m=m^{-\frac12}\log m$. 
We first check property (ii) of Definition~\ref{def:quasi}.
As a step in this direction, we claim that most of the weighted $L^2$-mass 
of the function $\chi_{\tau, 1}f_{m,n}^{\langle \kappa\rangle}$ 
lies in the domain $\calD_{m,n}$.
Indeed, a computation based on the change-of-variables 
formula \eqref{eq:int-flow} reveals that
\begin{multline}
  \label{eq:norm-fmn-Dmn}
m^{\frac12}\int_{\calD_{m,n}}\lvert f_{m,n}^{\langle \kappa\rangle}
\rvert^2\e^{-2mR_\tau}\diffA
\\=2m^{\frac12}\int_{-\delta_m}^{\delta_m}\int_{\T}
\big\lvert f_{m,n}^{\langle \kappa\rangle}\circ\psi_{m,n,t}
(\zeta)
\big\rvert^2\e^{-2m\,R_\tau\circ\psi_{m,n,t}(\zeta)}
\Re\big(-\bar{\zeta}
\partial_t\psi_{m,n,t}(\zeta)
\overline{\psi_{m,n,t}'(\zeta)}\big)\diffs(\zeta)\diff t
\\
=2m^{\frac12}\int_{-\delta_m}^{\delta_m}\big((4\pi)^{-\frac12}+
\Ordo(\delta_m^{2\kappa+1})\big)\e^{-mt^2}\diff t=1+\Ordo(\delta_m^{2\kappa+1})=
1+\Ordo(m^{-\kappa-\frac13}),
\end{multline}
where we move the integration to the flow coordinates 
$(t,\zeta)\in[-\delta_m,\delta_m]\times\T$. 

We know that the functions 
$f_{m,n}^{\langle \kappa\rangle}$ are bounded uniformly
in $\D_\e(0,\rho_0)$ independently of $m$ and $n$ 
while $\tau\in I_{\epsilon_0}$, so 
that 
\begin{equation}
\chi_{\tau, 1}\lvert f_{m,n}^{\langle \kappa\rangle}\rvert\le C_0
\label{eq:est00.11}
\end{equation}
holds in the whole plane $\C$, for some constant $C_0$.
Let $\calD_\circledast$ denote a fixed bounded domain which 
contains $\D\cup\calD_{m,n}$, such that the bound from below 
$R_\tau(z)\ge \theta_0\log\lvert z\rvert$ 
holds outside $\calD_\circledast$, for some 
$\theta_0>0$ and all $\tau\in I_{\epsilon_0}$. 
That such a domain exists for sufficiently large $m$ is 
shown in Proposition~\ref{prop:R}. 
On the other hand, in view of Remark~\ref{rem:growth-R}
we have the estimate
$$
\e^{-2mR_\tau}\le \e^{-\alpha_0(\log m)^2},\qquad \text{on }\;
\calD_\circledast\cap\D_\e(0,\rho_0)\setminus\calD_{m,n}
$$
for some constant $\alpha_0>0$ (if necessary we 
adjust $\rho_0$ and $\calD_\circledast$). 
As a consequence, we have
\begin{multline}\label{eq:0001}
m^{\frac12}\int_{\C\setminus\calD_{m,n}}\chi_{\tau, 1}^2\lvert f_{m,n}
^{\langle \kappa\rangle}\rvert^2\e^{-2mR_\tau}\diffA\le 
C_0^2m^{\frac12}\int_{\C\setminus\calD_{\circledast}}
\e^{-2m\theta_0\log \lvert z
\rvert}\diffA
\\
+C_0^2m^{\frac{1}{2}}\int_{\calD_{\circledast}
\cap\D(0,\rho_0)\setminus\calD_{m,n}}
\e^{-\alpha_0(\log m)^2}\diffA=\Ordo(m^{\frac12}\e^{-\alpha_0(\log m)^2})
=\Ordo(m^{-\alpha_0\log m+\frac12}).
\end{multline}
It now follows from \eqref{eq:norm-fmn-Dmn} and \eqref{eq:0001} that
\begin{multline*}
m^{\frac12}\int_{\C}\chi_{\tau, 1}^2\lvert f_{m,n}
^{\langle \kappa\rangle}\rvert^2\e^{-2mR_\tau}\diffA
=m^{\frac12}\int_{\calD_{m,n}}\lvert f_{m,n}
^{\langle \kappa\rangle}\rvert^2\e^{-2mR_\tau}\diffA
\\
+m^{\frac12}\int_{\C\setminus\calD_{m,n}}\chi_{\tau, 1}^2
\lvert f_{m,n}^{\langle \kappa\rangle}\rvert^2\e^{-2mR_\tau}\diffA=
1+\Ordo(m^{-\kappa-\frac13}),
\end{multline*}
where we use that $\chi_{\tau, 1}=1$ holds on $\calD_{m,n}$ together with
our foliation flow Lemma~\ref{lem:main-flow} and the estimate \eqref{eq:0001}. 
Hence, by the isometric property of $\Vop_{m,n}$ from 
Proposition~\ref{prop:Vop}, it follows that 
$$
\int_{\C}\chi_{\tau, 0}^2\lvert F_{m,n}^{\langle \kappa\rangle}
\rvert^2\e^{-2mQ}\diff A=1+\Ordo(m^{-\kappa-\frac13}),
$$
as required by property (ii) of Definition \ref{def:quasi}.

We turn to property (i) of Definition \ref{def:quasi}, the approximate 
orthogonality property. For a polynomial $p\in\operatorname{Pol}_n$ of 
degree at most $n-1$, we put $g=\Vop_{m,n}^{-1}[p]$ and note that
$g(\infty)=0$. 
For all large enough $n$ and $m$ with $\tau=\frac{n}{m}\in I_{\epsilon_0}$, 
the function $f_{m,n}^{\langle \kappa\rangle}$ is zero-free in a 
neighborhood of the extended exterior disk $\bar{\D}_\e\cup\{\infty\}$, 
which we may assume to be a fixed exterior disk $\D_\e(0,\rho_0)\cup\{\infty\}$ 
for some fixed $\rho_0<1$.
By the isometric property of $\Vop_{m,n}$, we find that
\begin{multline}
\int_\C\chi_{\tau, 0}\,p\, \overline{F_{m,n}^{\langle \kappa\rangle}}
\e^{-2mQ}\diffA
=m^{\frac14}\int_{\C}\chi_{\tau, 1}\,g\,\overline{f_{m,n}^{\langle \kappa\rangle}}
\e^{-2mR_\tau}\diffA(z)
\\
=m^{\frac14}\int_{\calD_{m,n}}\frac{g}{f_{m,n}^{\langle \kappa\rangle}}
\lvert f_{m,n}^{\langle \kappa\rangle}\rvert^2 \e^{-2mR_\tau}\diffA+
\Ordo(m^{-\frac{\alpha_0}{2}\log m+\frac34}\lVert p\rVert_{2mQ}),
\label{eq:est00.12}
\end{multline}
where we are required to justify the indicated error term estimate.
To do this, we need Proposition~\ref{prop-gengrowth-ext}, or more accurately, 
Lemma 3.5 in \cite{ahm1}, which gives the estimate for $p\in\mathrm{Pol}_n$
\begin{equation}
|p|\le C_1 m^{\frac12}\|p\|_{2mQ}\e^{m\checkQ_\tau}
\label{eq:est00.13}
\end{equation}
in the whole plane $\C$ for some constant $C_1$, independent of 
$\tau=\frac{n}{m}\in I_{\epsilon_0}$. The missing term on the right-hand side
of \eqref{eq:est00.12} equals
\[
m^{\frac14}\int_{\C\setminus\calD_{m,n}}\chi_{\tau, 1}\,g\,
\overline{f_{m,n}^{\langle \kappa\rangle}}
\e^{-2mR_\tau}\diffA=\int_{\C\setminus\phi_\tau^{-1}(\calD_{m,n})}
\chi_{\tau, 0}\,p\,\overline{F_{m,n}^{\langle \kappa\rangle}}
\e^{-2mQ}\diffA,
\] 
and if we apply the pointwise estimate \eqref{eq:est00.13}, we obtain
\begin{multline*}
\int_{\C\setminus\phi_\tau^{-1}(\calD_{m,n})}
\chi_{\tau, 0}|p\,F_{m,n}^{\langle \kappa\rangle}|\,\e^{-2mQ}\diffA\le
C_1m^{\frac12} \|p\|_{2mQ} \int_{\C\setminus\phi_\tau^{-1}(\calD_{m,n})}
\chi_{\tau, 0}|F_{m,n}^{\langle \kappa\rangle}|\,\e^{-2mQ+m\checkQ_\tau}\diffA
\\
=C_1m^{\frac34} \|p\|_{2mQ} \int_{\C\setminus\calD_{m,n}}
\chi_{\tau, 1}|f_{m,n}^{\langle \kappa\rangle}|
\,\e^{m(\checkQ_\tau-Q)\circ\phi_\tau^{-1}-mR_\tau}\diffA
\\
\le C_0C_1 m^{\frac34} \|p\|_{2mQ} \int_{\D_\e(0,\rho_0)\setminus\calD_{m,n}}
\,\e^{-mR_\tau}\diffA ,
\end{multline*}
where in the last step, we applied the estimate \eqref{eq:est00.11} and
the fact that $\checkQ_\tau\le Q$. The rest of the argument that gives 
\eqref{eq:est00.12} involves splitting the domain of integration 
using the set $\calD_\tau$, and proceeds as in \eqref{eq:0001}.
This establishes \eqref{eq:est00.12}, although we still need to control 
the main term on the right-hand side. 
To this end, we denote by $h$ the ratio 
$h={g}/{f_{m,n}^{\langle\kappa\rangle}}$. 
In view of the stated properties of $f_{m,n}^{\langle\kappa\rangle}$ and $g$, 
the function $h$ is holomorphic in the exterior disk $\D_\e(0,\rho_0)$ 
and vanishes at infinity.
Using the foliation flow as coordinates
on $\calD_{m,n}$ in terms of $(t,\zeta)\in[-\delta_m,\delta_m]\times\T$, 
we find from Lemma~\ref{lem:main-flow} that 
\begin{multline}
m^{\frac14}\int_{\calD_{m,n}}h(z)\lvert 
f_{m,n}^{\langle\kappa\rangle}(z)\rvert^2 
\e^{-2mR_\tau(z)}\diffA(z)
\\
=2m^{\frac14}\int_{-\delta_m}^{\delta_m}\int_{\T}
h\circ\psi_{m,n,t}
(\zeta)\big\lvert f_{m,n}^{\langle \kappa\rangle}\circ\psi_{m,n,t}
(\zeta)
\big\rvert^2\e^{-2m\,R_\tau\circ\psi_{m,n,t}(\zeta)}
\\
\times
\Re\big\{-\bar{\zeta}
\partial_t\psi_{m,n,t}(\zeta)
\overline{\psi_{m,n,t}'(\zeta)}\big\}\diffs(\zeta)\diff t
\\
=2m^{\frac14}\int_{-\delta_m}^{\delta_m}
\int_{\T}h\circ\psi_{m,n,t}(\zeta)\Big\{(4\pi)^{-\frac12}\e^{-mt^2}
+\Ordo\big(m^{-\kappa-\frac13}\e^{-mt^2}\big)\Big\}\diffs(\zeta)\diff t
\\
=\Ordo\bigg(m^{-\kappa-\frac1{12}}\int_{-\delta_m}^{\delta_m}\int_{\T}\lvert h
\circ\psi_{m,n,t}(\zeta)\rvert \diffs(\zeta)\,\e^{-mt^2}\diff t\bigg).
\label{eq:est00.14}
\end{multline}
Here, the crucial reduction in the last step of \eqref{eq:est00.14} 
is based on the fact that the function 
$h\circ\psi_{m,n,t}$ is holomorphic in $\bar\D_\e$ and vanishes at 
infinity, so that by the mean value property
\[
\int_{\T}h\circ\psi_{m,n,t}\,\diffs=0.
\]
Now that \eqref{eq:est00.14} is established, we need to simplify
the error term further.  
We will use the observation that all the steps before the last in 
\eqref{eq:est00.14} apply to a fairly general sufficiently integrable function 
in place of $h$, for instance $|h|$ will work. 
It then follows from \eqref{eq:est00.14} with $|h|$ instead that large enough
$m$, we have
\begin{multline*}
\int_{-\delta_m}^{\delta_m}
\int_{\T}|h\circ\psi_{m,n,t}(\zeta)|
\,\e^{-mt^2}\diffs(\zeta)\diff t
\le2\int_{\calD_{m,n}}|h(z)|\,
\lvert f_{m,n}^{\langle\kappa\rangle}(z)\rvert^2 
\e^{-2mR_\tau(z)}\diffA(z)
\\
=2\int_{\calD_{m,n}}\lvert g(z)\,
f_{m,n}^{\langle\kappa\rangle}(z)\rvert 
\e^{-2mR_\tau(z)}\diffA(z)
\le2 C_0\int_{\calD_{m,n}}|g(z)|\,
\e^{-2mR_\tau(z)}\diffA(z), 
\end{multline*}
where in the last step we applied the bound \eqref{eq:est00.11}. 
Finally, we apply the Cauchy-Schwarz inequality, and recall that 
recall that $g=\Vop_{m,n}^{-1}[p]$ where 
$\Vop_{m,n}$ has the isometry property
of Proposition~\ref{prop:Vop}:
\begin{multline}
\int_{-\delta_m}^{\delta_m}
\int_{\T}|h\circ\psi_{m,n,t}(\zeta)|
\,\e^{-mt^2}\diffs(\zeta)\diff t
\le2 C_0\int_{\calD_{m,n}}|g(z)|\,
\e^{-2mR_\tau(z)}\diffA(z)
\\
\le2C_0\|g\|_{L^2(\calD_{m,n},\e^{-2mR_\tau})}
\bigg\{\int_{\calD_{m,n}}\e^{-2mR_\tau}\diffA\bigg\}^{1/2}
=\Ordo\big(m^{-\frac14}
\|p\|_{2mQ}\big). 
\label{eq:est00.16}
\end{multline} 
Here, we used a simple decay estimate of the integral 
of the {\em Gaussian ridge} $\e^{-2mR_\tau}$. Next, 
we write $g/f_{m,n}^{\langle\kappa\rangle}$ in place of $h$,
and combine the estimates \eqref{eq:est00.14} and
\eqref{eq:est00.16}, and arrive at
\begin{multline}
m^{\frac14}\int_{\calD_{m,n}}g\,\overline{f_{m,n}^{\langle\kappa\rangle}}
\e^{-2mR_\tau(z)}\diffA(z)
=m^{\frac14}\int_{\calD_{m,n}}h(z)\lvert f_{m,n}^{\langle\kappa\rangle}(z)\rvert^2 
\e^{-2mR_\tau(z)}\diffA(z)
\\
=\Ordo\big(m^{-\kappa-\frac13}\|p\|_{2mQ}\big).
\label{eq:est00.17}
\end{multline}
In view of \eqref{eq:est00.12} and \eqref{eq:est00.17}, we find that for
all polynomials $p\in\mathrm{Pol}_n$,
\begin{equation}
\int_\C\chi_{\tau, 0}\,p\, \overline{F_{m,n}^{\langle \kappa\rangle}}
\e^{-2mQ}\diffA
=\Ordo\big(m^{-\kappa-\frac13}\|p\|_{2mQ}\big),
\label{eq:est00.018}
\end{equation}
as required. 
Since in addition, $f_{m,n}^{\langle\kappa\rangle}(\infty)>0$, 
while $\calQ_\tau(\infty)\in\R$ and 
$\phi_\tau'(\infty)>0$ hold, the leading coefficient
of the quasipolynomial $F_{m,n}^{\langle\kappa\rangle}$ is now positive, which 
settles property (iii) of Definition \ref{def:quasi} as well.  
This completes the proof.
\end{proof}

\subsection{Polynomialization of the quasipolynomials and
proof of Theorem~\ref{main}}
We have applied Lemma~\ref{lem:main-flow} to obtain the existence of 
quasi\-polynomials $F_{m,n}^{\langle \kappa\rangle}$,
of degree $n$ and accuracy $\kappa$ with an asymptotic expansion, 
and shown that they are approximately 
orthogonal and normalized. To obtain the full $L^2$-expansion, 
it remains to show that they are indeed good 
approximations of the true normalized orthogonal polynomials $P_{m,n}$.
 
\begin{proof}[Proof of Theorem~\ref{main}]
We retain the above notation, and consider the $\dbar$-problem 
\[
\dbar_z u(z)=F_{m,n}^{\langle \kappa\rangle}(z)\dbar_z\chi_{\tau, 0}(z).
\]
In view of Proposition~\ref{bh}, the $L^2_{2mQ,n}$-norm minimal solution 
$u_{0}$, which then has the growth $u_{0}(z)=\Ordo(|z|^{n-1})$ near 
infinity, enjoys the norm bound
\begin{equation}
\int_\C|u_{0}|^2\e^{-2mQ}\dA\le \frac{1}{\alpha_1 m}
\int_{\calS_\tau}|F_{m,n}^{\langle \kappa\rangle}|^2
|\dbar\chi_{\tau, 0}|^2 \e^{-2mQ}\dA,
\label{eq-dbarsol1}
\end{equation}
where $\alpha_1>0$ stands for the minimum of $\hDelta Q$ on the biggest 
droplet $\calS_\tau$ with $\tau\in I_{\epsilon_0}$ (which is attained for 
the rightmost endpoint $\tau=1+\epsilon_0$). 
Next, given that the quasipolynomials of degree $n$ are of the form
$F_{m,n}^{\langle \kappa\rangle}=m^{\frac14}
\Vop_{m,n}[f_{m,n}^{\langle \kappa\rangle}]$, 
where the functions $f_{m,n}^{\langle \kappa\rangle}$ 
are uniformly bounded in $\D_\e(0,\rho_0)$ for some radius $\rho_0<1$,
we find that
\begin{multline}
\int_{\calS_\tau}|F_{m,n}^{\langle \kappa\rangle}|^2|\dbar\chi_{\tau, 0}|^2 
\e^{-2mQ}\diffA=m^{\frac12}
\int_{\D}|f_{m,n}^{\langle\kappa\rangle}|^2|\dbar\chi_{\tau, 1}|^2
|\phi_\tau'\circ\phi_\tau^{-1}|^2 \e^{-2mR_\tau}\diffA
\\
=\Ordo(m^{\frac12}\e^{-\alpha_2m})
\label{eq:expdecay00.1}
\end{multline}
for some $\alpha_2>0$ such that $2R_\tau\ge{\alpha_2}$ on the support 
of $\bar\partial\chi_{\tau, 1}$. This exponential decay estimate is possible 
since the support of $\bar\partial\chi_{\tau, 1}$ is located inside $\D$ 
away from the boundary.
Note that in the context of the estimate \eqref{eq:expdecay00.1}
it is important as well that the expression $|\phi_\tau'\circ\phi_\tau^{-1}|^2$ 
is uniformly bounded on the support of $\bar\partial\chi_{\tau, 1}$ as well.
If we combine the above estimates \eqref{eq-dbarsol1} and 
\eqref{eq:expdecay00.1}, we find that
\begin{equation}
\int_\C|u_{0}|^2\e^{-2mQ}\dA=\Ordo(m^{-\frac12}\e^{-\alpha_2m}),
\label{eq-dbarsol1.1}
\end{equation}
as $m\to\infty$ while $\tau=\frac{n}{m}\in I_{\epsilon_0}$, with a uniform 
implicit constant.
Next, we put 
\[
P_{m,n}^\star:=F_{m,n}^{\langle \kappa\rangle}\chi_{\tau, 0}-u_{0}
\]
which is then automatically a polynomial of degree $n$, since the function is
entire and has growth $|P^\star_{m,n}(z)|\asymp |z|^n$ near infinity. 
Moreover, in view of \eqref{eq-dbarsol1.1}, this polynomial is very close to
the function $F_{m,n}^{\langle\kappa\rangle}\chi_{\tau, 0}$ in the norm 
of $L^2(\C,\e^{-2mQ})$:
\begin{equation}
\int_\C|P_{m,n}^\star-F_{m,n}^{\langle\kappa\rangle}\chi_{\tau, 0}|^2
\e^{-2mQ}\dA=\int_\C|u_{0}|^2\e^{-2mQ}\dA=\Ordo(m^{-\frac12}\e^{-\alpha_2m}).
\label{eq-dbarsol1.2}
\end{equation}
It now follows from \eqref{eq:est00.018} and \eqref{eq-dbarsol1.2} that 
for all polynomials $p\in\operatorname{Pol}_n$ of degree at most $n-1$,
we have that
 \begin{equation}
\int_\C p\, \bar{P}_{m,n}^\star\,\e^{-2mQ}\dA=\Ordo(m^{-\kappa-\frac13}\|p\|_{2mQ}),
\label{eq-orthrel1}
\end{equation}
while
\begin{equation}
\int_\C |P_{m,n}^\star|^2\e^{-2mQ}\dA=1+\Ordo(m^{-\kappa-\frac13}).
\label{eq-orthrel2}
\end{equation}
We observe that by duality, \eqref{eq-orthrel1} asserts that
\begin{equation}
\|\Pop_{m,n} P_{m,n}^\star\|_{2mQ}=\Ordo(m^{-\kappa-\frac13}),
\label{eq-projest1}
\end{equation}
where $\Pop_{m,n}$ denotes the orthogonal projection in $L^2(\C,\e^{-2mQ})$
onto the subspace $\operatorname{Pol}_n$ of polynomials of degree at
most $n-1$. If we use this to correct the polynomial $P_{m,n}^\star$, and put
$\tilde P_{m,n}:=\Pop_{m,n}^\perp P_{m,n}^\star=P_{m,n}^\star-\Pop_{m,n} P_{m,n}^\star$, 
then automatically $\tilde P_{m,n}$ has degree $n$ and it is also orthogonal to
all the lower degree polynomials. As a consequence,  $\tilde P_{m,n}$ must be
a scalar multiple of $P_{m,n}$, the orthogonal polynomial we are looking for,
which we write as $\tilde P_{m,n}=c P_{m,n}$ for a constant $c$.  Putting
things together so far, we have obtained that 
\begin{equation}
\big\|\tilde P_{m,n}-F_{m,n}^{\langle\kappa\rangle}\chi_{\tau, 0}
\big\|_{2mQ}=\Ordo(m^{-\kappa-\frac13})
\label{eq:projest1.001}
\end{equation}
with a uniform implied constant. Moreover, by \eqref{eq-orthrel2} and
\eqref{eq-projest1}, the norm of $\tilde P_{m,n}$ equals 
\begin{equation}
|c|=\|cP_{m,n}\|_{2mQ}=\big\|\tilde P_{m,n}\big\|_{2mQ}=1+\Ordo(m^{-\kappa-\frac13}),
\label{eq-orthrel3}
\end{equation}
Next, by our version of the Bernstein-Walsh lemma
(Proposition~\ref{prop-gengrowth-ext}), 
it follows from \eqref{eq:projest1.001} that 
\[
\big|c P_{m,n}-F_{m,n}^{\langle\kappa\rangle}\big|=
\big|\tilde P_{m,n}-F_{m,n}^{\langle\kappa\rangle}
\big|=\Ordo(m^{-\kappa+\frac16}\e^{m\checkQ_\tau})
\]
holds in $\calS_\tau^c$, which after division by $F_{m,n}^{\langle \kappa\rangle}$ 
gives that
\begin{equation}\label{eq:im-c-coeff}
  \bigg| c \frac{P_{m,n}}{F_{m,n}^{\langle\kappa\rangle}}-1\bigg|
=\Ordo(m^{-\kappa-\frac{1}{12}}),
\end{equation}
since $f_{m,n}^{\langle \kappa\rangle}$ is uniformly bounded
away from zero. Next, we let $|z|\to+\infty$ and observe that both 
the functions $F_{m,n}^{\langle \kappa\rangle}$ and $P_{m,n}$ have positive
leading coefficients, whose quotient is denoted by $\gamma_{m,n}$.
Since $\gamma_{m,n}>0$ we obtain from \eqref{eq:im-c-coeff} that
\[
\frac{|\Im c|}{|c|} \le \big|c\gamma_{m,n}-1\big|=\Ordo(m^{-\kappa-\frac{1}{12}}),
\]
where the left-hand side inequality is elementary. Moreover, we can also
realize from the above that $\Re(c)>0$. But then it follows from
\eqref{eq-orthrel3} that
\[
c=1+\Ordo(m^{-\kappa-\frac{1}{12}}).
\]
It now follows from this observation 
combined with \eqref{eq:projest1.001} that
$$
\lVert P_{m,n}-\chi_{\tau, 0}F_{m,n}^{\langle \kappa\rangle}\rVert_{2mQ}=
\Ordo(m^{-\kappa-\frac1{12}}).
$$
This falls slightly short of allowing us to obtain Theorem~\ref{main}
right away. The problem is that our error term is larger than what is 
claimed. However, since the precision $\kappa$ is arbitrary, 
we might as well replace $\kappa$ by $\kappa+1$ and see what we get. 
This would give that
\begin{equation}
\label{eq:Fkappa+1}
\lVert P_{m,n}-\chi_{\tau, 0}F_{m,n}^{\langle \kappa+1\rangle}\rVert_{2mQ}=
\Ordo(m^{-\kappa-1-\frac1{12}}).
\end{equation}
By analyzing the last term in the asymptotic expansion, it is easy to 
verify that
$$
\lVert \chi_{\tau, 0}F_{m,n}^{\langle \kappa+1\rangle}-
\chi_{\tau, 0}F_{m,n}^{\langle \kappa\rangle}\rVert_{2mQ}=\Ordo(m^{-\kappa-1}),
$$
and hence the assertion of the theorem immediate from this estimate and  
\eqref{eq:Fkappa+1}.
\end{proof}

\subsection{Proof of the main theorem} 
We are now ready to obtain the pointwise asymptotic 
expansion of the orthogonal polynomials.
We still assume that Lemma~\ref{lem:main-flow} holds.

\begin{proof}[Proof of Theorem~\ref{thm:main-pw}]
The quasipolynomials $F_{m,n}^{\langle\kappa\rangle}$ obtained
in Theorem~\ref{main} may be written in the form
$$
F_{m,n}^{\langle\kappa\rangle}=
m^{\frac14}\sqrt{\phi_\tau'}[\phi_\tau]^n
\e^{m\calQ_\tau}\sum_{j=0}^\kappa m^{-j}\calB_{\tau, j}, 
$$
where $\calB_{\tau, j}=[\phi_\tau']^{\frac12}\,B_{\tau, j}\circ\phi_\tau$ 
are uniformly bounded, and holomorphic in the 
exterior domain $\calK_\tau^c$. To obtain the theorem,
we need to show that $F_{m,n}^{\langle\kappa\rangle}$ is 
close to $P_{m,n}$ pointwise in the complement of the set 
\begin{equation}\label{eq:def-KtAm}
  \calK_{\tau,A,m}=
  \big\{z\in \C: \mathrm{dist}_\C(z,\calS_\tau^c)
  \ge A(m^{-1}\log m)^{\frac12}\big\}.
\end{equation}
On the complement $\calK_{\tau, A, m}^c$ we have the estimate
$$
0\le m(\checkQ_\tau-\breve{Q}_\tau)(z)\le D\log m,
$$
and hence
$$
\e^{m(\checkQ_\tau-\breve{Q}_\tau)}\le \e^{D\log m}=m^D
$$
where $D$ is some positive constant, which is uniformly bounded while 
$\tau\in I_{\epsilon_0}$. To see this, a simple Taylor expansion of the 
difference $\checkQ_\tau-\breve{Q}_\tau$ in the interior direction suffices.
In view of Theorem~\ref{main}, and the pointwise estimate of 
Proposition~\ref{prop-gengrowth-ext} applied
to the intermediate set $\mathcal{X}_\tau$ between $\calK_\tau$ and 
$\calS_\tau^c$ where the cut-off function $\chi_{\tau, 0}$ assumes the value 
$1$, we find that
\[
\lvert P_{m,n}(z)-F_{m,n}^{\langle \kappa\rangle}(z)\rvert=
\Ordo\big(m^{-\kappa-\frac12}\e^{m\checkQ_\tau(z)}\big)=
\Ordo\big(m^{-\kappa-\frac12+D}\e^{m\breve{Q}_\tau(z)}\big),
\qquad z\in\calK_{\tau,A,m}^c,
\]
where the implicit constant again is uniform in the relevant parameter 
range.
We may rephrase this as saying that
\[
P_{m,n}(z)=F_{m,n}^{\langle \kappa\rangle}(z)+\Ordo\big(m^{-\kappa-
\frac12+D}\e^{m\breve{Q}_\tau(z)}\big)=m^{\frac14}
\sqrt{\phi_\tau'}[\phi_\tau]^n\e^{m\calQ_\tau}
\Big(\sum_{j=0}^{\kappa}\calB_{\tau, j}
+\Ordo\big(m^{-\kappa-\frac{3}{4}+ D}\big)\Big),
\]
for $z\in\calK_{\tau,A,m}^c$.
This essentially proves the theorem, except that 
the error term is now 
worse than claimed. However, we may fix this by
replacing $\kappa$ by $\kappa':=\kappa+\lceil D\rceil+1$ 
in the above argument, to obtain
on $\calK_{\tau,A,m}^c$ that
\begin{multline*}
P_{m,n}(z)=m^{\frac14}\sqrt{\phi_\tau'}[\phi_\tau]^n
\e^{m\calQ_\tau}\Big(\sum_{j=0}^{\kappa'} m^{-j}\calB_{\tau, j} 
+\Ordo(m^{-\kappa-\frac74})\Big)
\\
=m^{\frac14}\sqrt{\phi_\tau'}[\phi_\tau]^n
\e^{m\calQ_\tau}\Big(\sum_{j=0}^\kappa m^{-j}\calB_{\tau, j} 
+\Ordo(m^{-\kappa-1})\Big),
\end{multline*}
where the last step follows since the functions 
$m^{-j}\calB_{\tau, j}$ are all $\Ordo(m^{-\kappa-1})$ 
for $j$ in the range
$\kappa+1\le j\le \kappa'$.
The proof is complete.
\end{proof}

\section{Algorithmic determination of the coefficients 
\texorpdfstring{$\calB_{\tau, j}$}{B-j-tau}}\label{sect:alg}
\subsection{Implementation of the radial Laplace 
method}\label{ss:alg-explicit}

We turn to the algorithm of Theorem~\ref{thm:main-coeff}. To proceed, 
we need two families of differential operators.
We recall the differential operators $\Lop_k$ defined in \eqref{eq:Lop} 
appearing in the application of 
Laplace's method in Proposition~\ref{prop:steep}.
We need to apply these operators to functions defined in a neighborhood of
the unit circle,
and we apply them in the radial direction. So, for functions 
$f(r\e^{\imag\theta})$, we put
$$
\Lop_k[f](r\e^{\imag\theta})=\sum_{\nu=k}^{3k}\frac{(-1)^{\nu-k}2^{-\nu}}
{\nu!(\nu-k)![\partial_r^2R_{\tau}(r\e^{\imag\theta})]^\nu}
\partial^{2\nu}_r\left(\big[W_{\tau}(r\e^{\imag\theta})\big]^{\nu-k}
f(r\e^{\imag\theta})\right),
$$
where
$$
W_{\tau}(r\e^{\imag\theta})=R_\tau(r\e^{\imag\theta})-\frac{1}{2}(r-1)^2
\partial_x^2R_\tau(x\e^{\imag\theta})\Big\vert_{x=1}.
$$
The second family of operators is defined implicitly in the following lemma,
which turns explicit appearances of the parameter $l$ into differential
operators.

\begin{lem}\label{lem:pseudo}
Let $k$ be a nonnegative integer. Then there exist partial differential 
operators $\Mop_k$ of order $2k$ with real-analytic coefficients, such that 
for any integer $l\geq 0$ and any function smooth function $f$ defined 
in a neighborhood of $\T$, we have that
$$
\int_{\T}\e^{\imag l\theta} \big(\partial^2_r R_\tau(r\e^{\imag\theta})\big)^{-\frac12}
\Lop_{k}[r^{1-l}f(r \e^{\imag\theta})]
\bigg\vert_{r=1}\diff\theta=\int_{\T}\e^{\imag l\theta}
\Mop_k [f](\e^{\imag\theta})\diff\theta.
$$
\end{lem}

\begin{proof}
We first observe that by integration by parts, multiplication by $l$ 
corresponds to applying the differential operator $\imag\partial_\theta$ inside
the integral:
$$
l\int_{\T}f(\theta)\e^{\imag l\theta}\diff\theta=\int_{\T}\imag\partial_\theta 
f(\theta)\e^{\imag l\theta}\diff\theta.
$$
From this it is immediate that the formula
\begin{equation}
\label{eq:part-int}
p(l)\int_{\T}f(\theta)\e^{\imag l\theta}\diff\theta=\int_{\T}p(\imag\partial_\theta) 
f(\theta)\e^{\imag l\theta}\diff\theta
\end{equation}
holds for polynomials $p$.
Structurally, 
$\Lop_{k}[r^{1-l}f(r \e^{\imag\theta})]$
can be written as
\begin{equation}\label{eq:L-struct}
 \Lop_{k}[r^{1-l}f(r\e^{\imag\theta})]=\sum_{\nu=k}^{3k}b_\nu(r\e^{\imag\theta}) 
\partial^{2\nu}_r\big[[W_{\tau}(r\e^{\imag\theta})]^{\nu-k} r^{1-l}
f(r\e^{\imag\theta})\big],
\end{equation}
where $b_\nu$ is the real-analytic function given by
$$
b_\nu(r\e^{\imag\theta})=\frac{(-1)^{\nu-k}2^{-\nu}}
{\nu!(\nu-k)![\partial_r^2R_{\tau}(r\e^{\imag\theta})]^\nu}.
$$ 
We observe that by the Leibniz rule
\begin{multline}\label{eq:rad-deriv}
\partial_r^j(r^{1-l}f(r\e^{\imag\theta}))\Big\vert_{r=1}=\sum_{i=0}^{j}\binom{j}{i}
(-1)^{j-i}(l-1)_{j-i} r^{1-l-j+i}\partial_r^{i}f(r\e^{\imag\theta})\Big\vert_{r=1} 
\\
=\sum_{i=0}^{j}\binom{j}{i}(-1)^{j-i}(l-1)_{j-i}\partial_r^i f(r\e^{\imag\theta})
\Big\vert_{r=1},
\end{multline}
where $(x)_i=x(x+1)\cdots(x+i-1)$ denotes the standard Pochhammer symbol.
We return to the formula \eqref{eq:L-struct} for $\Lop_k$. Again by the Leibniz 
formula
we have that
\begin{multline*}
\partial^{2\nu}_r[W_\tau^{\nu-k}(\e^{\imag\theta}) r^{1-l}f(r\e^{\imag\theta})]
\Big\vert_{r=1}
=\sum_{j=0}^{2\nu} \binom{2\nu}{j}\partial_r^{2\nu-j}\Big([W_{\tau}
(r\e^{\imag\theta})]^{\nu-k}\Big) \partial_r^{j}\Big(r^{1-l}
f(r\e^{\imag\theta})\Big)\bigg\vert_{r=1}
\\
=\sum_{j=0}^{3k-\nu}\binom{2\nu}{j}\partial_r^{2\nu-j}\Big([W_{\tau}
(r\e^{\imag\theta})]^{\nu-k}\Big) \partial_r^{j}\Big(r^{1-l}
f(r\e^{\imag\theta})\Big)\bigg\vert_{r=1} 
\\
=\sum_{j=0}^{3k-\nu}\sum_{i=0}^{j}(-1)^{j-i}\binom{2\nu}{j}\binom{j}{i}(l-1)_{j-i}
\partial_r^{2\nu-j}\Big([W_{\tau}
(r\e^{\imag\theta})]^{\nu-k}\Big)
\partial_r^i f(r\e^{\imag\theta})
\Big\vert_{r=1},
\end{multline*}
where the truncation of the sum follows from an application of the flatness 
of $W_\tau$ near the unit circle $\T$, and the last equality is due to 
\eqref{eq:rad-deriv}.
We write the expression for $\Lop_k[r^{1-l}f(r\e^{\imag\theta})]$ as
$$
\Lop_k[r^{1-l}f(r\e^{\imag\theta})]\Big\vert_{r=1}
=\sum_{\nu=k}^{3k}\sum_{j=0}^{3k-\nu}\sum_{i=0}^{j} (l-1)_{j-1} c_{i,j,\nu}
(\e^{\imag\theta}) \partial^i_rf(r\e^{\imag\theta})\Big\vert_{r=1},
$$
where
$$
c_{i,j,\nu}(\e^{\imag\theta})=(-1)^{j-i}\binom{2\nu}{j}
\binom{j}{i}(l-1)_{j-i}
b_\nu(\e^{\imag\theta})\partial_r^{2\nu-j}\Big([W_{\tau}
(r\e^{\imag\theta})]^{\nu-k}\Big)\Big\vert_{r=1}.
$$
Changing the order of summation, we arrive at
\begin{multline*}
\Big(\partial^2_r R_\tau(r\e^{\imag\theta})\Big)^{-\frac12}\Lop_k[r^{1-l}
f(r\e^{\imag\theta})]
\Big\vert_{r=1} 
\\
=\sum_{i=0}^{2k}\sum_{j=i}^{2k}
(-1)^{j-i}\binom{j}{i}(l-1)_{i-j}
\Big(\partial^2_r R_\tau(r\e^{\imag\theta})\Big)^{-\frac12}
d_{j}(\e^{\imag\theta})\partial^i_rf(r\e^{\imag\theta})
\Big\vert_{r=1},
\end{multline*}
where 
$$
d_{j}(\e^{\imag\theta})=\sum_{\nu=k}^{3k-j}\binom{2\nu}{j}
b_\nu(\e^{\imag\theta})\partial_r^{2\nu-j}\Big([W_{\tau}
(r\e^{\imag\theta})]^{\nu-k}\Big)\Big\vert_{r=1}.
$$
It follows from \eqref{eq:part-int} that the asserted identity holds with 
$\Mop_k$ given by
$$
\Mop_k[f](\e^{\imag\theta})=\sum_{i=0}^{2k}\sum_{j=i}^{2k}
(-1)^{j-i}\binom{j}{i}(\imag\partial_\theta-1)_{i-j}\Big[(\partial^2_r 
R_\tau(r\e^{\imag\theta}))^{-\frac12}d_{j}(\e^{\imag\theta})
\partial^i_rf(r\e^{\imag\theta})\Big]
\bigg\vert_{r=1}.
$$
The proof of the lemma is complete.
\end{proof}

\subsection{Algorithmic computation of the coefficients
 in the asymptotic expansion}\label{sec:alg}

In this section we supply the proof of Theorem~\ref{thm:main-coeff}, and 
explain the underlying computational algorithm. The main point is 
that we show how to iteratively obtain the coefficients, given that 
an asymptotic expansion exists, as formulated in Theorem~\ref{main}.

\begin{proof}[Proof of Theorem~\ref{thm:main-coeff}]
Fix the precision $\kappa$ to be a positive integer. Let $F_{m,n}
^{\langle \kappa\rangle}$ be the 
approximate orthogonal quasipolynomials from Theorem~\ref{main} with the 
expansion
$$
F_{m,n}^{\langle \kappa\rangle}(z)=m^{\frac14}\sqrt{\phi_\tau'(z)}[\phi_\tau(z)]^n
\e^{m\calQ_\tau(z)}\sum_{j=0}^\kappa m^{-j}\calB_{\tau, j}(z),
$$
where the functions $\calB_{\tau, j}$ are bounded and holomorphic on 
$\calK_\tau^c$ for some compact subset $\calK_{\tau}$ of 
$\calS_\tau^\circ$, which we may assume to be the conformal image 
of the exterior disk $\D_\e(0,\rho_0)$ under the mapping 
$\phi_\tau^{-1}$. If we make the ansatz
$$
\calB_{\tau, j}(z)=\sqrt{\phi_\tau'(z)}(B_{\tau, j}\circ\phi_\tau)(z),
$$
we may express $F_{m,n}$ using the canonical positioning operator as 
$F_{m,n}^{\langle \kappa\rangle}=m^{\frac14}\Vop_{m,n}[f_{m,n}
^{\langle \kappa\rangle}]$, where
\begin{equation}\label{eq:fnm}
f_{m,n}^{\langle \kappa\rangle}(z)=\sum_{j=0}^\kappa m^{-j}B_{\tau, j}(z),
\qquad z\in\D_\e(0,\rho_0).
\end{equation} 
According to Theorem~\ref{main}, the functions $F_{m,n}^{\langle \kappa\rangle}$
have the approximate orthogonality property
\begin{equation}\label{eq:F-orth}
\int_{\C}\chi_{\tau, 0}F_{m,n}^{\langle \kappa\rangle}\bar{p}\,\e^{-2mQ}\diffA=
\Ordo(m^{-\kappa-1}
\lVert p\rVert_{2mQ}),\qquad p\in \mathrm{Pol}_n.
\end{equation}
The function $\chi_{\tau, 0}$ is a cut-off function with $0\le \chi\le 1$ 
throughout $\C$, such that $\chi_{\tau, 0}$ vanishes on $\calK_\tau$ and 
equals $1$ on 
$\mathcal{X}_{\tau}^c$, where $\calK_\tau$ lies 
at a fixed positive distance from $\partial\calS_\tau$, and 
$\mathcal{X}_\tau$ is an 
intermediate set between them (cf.\ Definition~\ref{def:intermediate}).
We consider the associated cut-off function 
$\chi_{\tau, 1}=\chi_{\tau, 0}\circ\phi_\tau^{-1}$, tacitly extended to vanish 
where it is undefined. 
Without loss of generality, we may assume that $\chi_{\tau, 1}$ is radial.
By Remark~\ref{rem:im-set}, we may assume that 
$\chi_{\tau, 1}$ vanishes on $\D(0,\rho_0')$ for some number $\rho_0'$ with
$\rho_0<\rho_0'<1$.
In order to compute the functions $B_{\tau, j}$, we would like to apply
equation~\eqref{eq:F-orth} to 
$$
q(z)=\Vop_{m,n}[z^{-l}]=\phi_\tau'(z)[\phi_\tau(z)]^{n-l}\e^{m\calQ_\tau(z)} 
$$
for a positive integer $l$, but this function 
is unfortunately not a polynomial. To fix this, we consider the
$L^2_{2mQ,n}$-minimal 
solution $v$ to the $\bar\partial$-problem
$$
\bar\partial v=\bar\partial(\chi_{\tau, 0} q)=q\,\bar\partial \chi_{\tau, 0}. 
$$
If $v$ is the solution, then the difference $\chi_{\tau, 0}q-v$ will be an entire 
function with the polynomial growth bound $\Ordo(\lvert z\rvert^{n-1})$ at 
infinity, 
and hence a polynomial of degree less than or equal to $n-1$.
By the estimate of Proposition~\ref{bh}, we have the norm control
$$
\int_{\C}\lvert v\rvert^2\e^{-2mQ}\diffA\le \frac{1}{2m}
\int_{\C}\lvert q\rvert^2\lvert 
\bar\partial\chi_{\tau, 0}\rvert^2 \frac{\e^{-2mQ}}{\hDelta Q} \diffA 
\leq \frac{A^2}
{2m\alpha_1}\int_{\mathcal{X}_\tau\setminus\calK_\tau}\lvert 
q\rvert^2 \e^{-2mQ}\diffA,
$$
where we have used that there exists a positive real $\alpha_1$ such that 
$\hDelta Q\ge \alpha_1$ holds on $\calS_\tau$, which contains the support of 
$\bar\partial\chi_{\tau, 0}$, and that we have the bound 
$\lvert \bar\partial\chi_{\tau, 0}\rvert\le A$. Since the support of 
$\bar\partial\chi_{\tau, 0}$
lies in $\calK_\tau^c$, we may use the structure of $q$ as 
$q=\Vop_{m,n}[z^{-l}]$ and Proposition~\ref{prop:Vop}
$$
\int_{\mathcal{X}_\tau\setminus\calK_\tau}\lvert 
q\rvert^2 \e^{-2mQ}\diffA=\int_{\rho_0\le \lvert z\rvert 
\le \rho_0''}\lvert z\rvert^{-2l}
\e^{-2mR_\tau(z)}\diffA(z),
$$
where $\rho_0''$ is associated with a natural choice of the intermediate set 
$\mathcal{X}_\tau$ as the image of an exterior disk under $\phi_\tau^{-1}$, and 
satisfies $\rho_0<\rho_0'<\rho_0''<1$. 
Due to Proposition~\ref{prop:R}, this immediately 
gives that for any fixed positive integer $l$
$$
\int_{\C}\lvert v\rvert^2\e^{-2mQ}\diffA=\Ordo(\e^{-\epsilon_1 m})
$$
as $m, n$ tend to infinity while $\tau=\frac{n}{m}\in I_{\epsilon_0}$, for some 
positive real $\epsilon_1$. This means that for a fixed positive integer 
$l$, we have
for $q=\Vop[z^{-l}]$ the approximate orthogonality
\begin{equation}
\int_{\C}\chi_{\tau, 0}^2 F_{m,n}^{\langle \kappa\rangle} \bar{q}\,\e^{-2mQ}
\diffA=\Ordo(m^{-\kappa-1}),
\label{eq:orth101}
\end{equation}
where we have used that $\chi_{\tau, 0} q-v$ is a polynomial of degree at most 
$n-1$, and the above smallness of $v$.
If we use the canonical positioning operator as in Proposition \ref{prop:Vop} 
in polarized form, \eqref{eq:orth101} reads in polar coordinates
\begin{equation}\label{eq:pol-f}
m^{\frac14}\int_{\T}\e^{\imag l\theta}\int_{\rho_0}^\infty r^{1-l}\chi_{\tau, 1}^2
(r)f_{m,n}^{\langle\kappa\rangle}(r\e^{\imag\theta})
\e^{-2mR_{\tau}(r\e^{\imag\theta})}\diff r\diffs(\e^{\imag \theta})=
\Ordo\left(m^{-\kappa-1}\right),
\end{equation}
for fixed $l$. 
We now apply Proposition~\ref{prop:steep2} to the radial integral, with 
$V(r)=2R_{\tau}(r\e^{\imag\theta})$. Note that $\partial^2_r R_\tau(r\e^{i\theta})
\vert_{r=1}=4\hDelta R_\tau(\e^{\imag\theta})$.
As a consequence, the inner integral in \eqref{eq:pol-f} has an expansion
\begin{multline*}
\int_{\rho_0}^\infty r^{1-l}\chi_{\tau, 1}^2(r)f_{m,n}^{\langle\kappa\rangle}
(r\e^{\imag\theta})
\e^{-2mR_{\tau}(r\e^{\imag\theta})}\diff r
=\bigg(\frac{\pi}{4m \hDelta R_{\tau}(\e^{\imag\theta})}\bigg)^{\frac12}\sum_{j=0}
^{\kappa}m^{-j}\Lop_j [r^{1-l}f_{m,n}^{\langle\kappa\rangle}(r\e^{\imag\theta})]
\Big\vert_{r=1} 
\\
+ \Ordo\bigg(m^{-\kappa-1}\big\lVert r^{1-l}\chi_{\tau, 1}^2f_{m,n,\theta}
^{\langle\kappa\rangle}
\big\rVert_{C^{2(\kappa+1)}([\rho_0,\rho_2])}+\big\lVert r^{1-l}\chi_{\tau, 1}^2 
f_{m,n,\theta}^{\langle\kappa\rangle}
\big\rVert_{L^\infty([\rho_1,\infty))}\rho_1^{-m\vartheta+1}\bigg),
\end{multline*}
where we to simplify the notation we use the subscript $\theta$ to denote 
the radial 
restriction $f_\theta(r)= f(r\e^{\imag\theta})$. Here, $\vartheta,\alpha$ and 
$\rho_1$ 
are some real numbers with 
$\vartheta>0,\,\alpha>0$ and $1<\rho_1<\rho_2$, which are independent of 
$\tau\in I_{\epsilon_0}$. 
By applying the standard Cauchy estimates to the functions 
$f_{m,n}^{\langle \kappa\rangle}$, and
by Remark~\ref{rem:im-set} (both part (a) and (b) are needed) we have uniform 
control on the norms
$$
\big\lVert r^{1-l}\chi_{\tau, 1}^2f_{m,n,\theta}^{\langle\kappa\rangle}
\big\rVert_{C^{2(\kappa+1)}([\rho_0,\rho_2])}\quad \text{and}\quad
\big\lVert r^{1-l}\chi_{\tau, 1}^2 f_{m,n,\theta}
^{\langle\kappa\rangle}\big\rVert_{L^\infty([\rho_1,\infty))}
$$ 
provided that $l$ is fixed, and that 
$f_{m,n}^{\langle \kappa\rangle}$ are uniformly bounded.
For fixed $l$, it follows that 
\begin{multline}\label{eq:asymp-clean}
\int_{\rho_0}^\infty r^{1-l}\chi_{\tau, 1}^2(r)f_{m,n}^{\langle\kappa\rangle}
(r\e^{\imag\theta})
\e^{-2mR_{\tau}(r\e^{\imag\theta})}\diff r
\\
=\bigg(\frac{\pi}{4m \hDelta R_{\tau}(\e^{\imag\theta})}\bigg)^{\frac12}\sum_{j=0}
^{\kappa}m^{-j}\Lop_j [r^{1-l}f_{m,n}^{\langle\kappa\rangle}(r\e^{\imag\theta})]
\Big\vert_{r=1} + \Ordo\big(m^{-\kappa-1}\big),
\end{multline}
where the implied constant is uniformly bounded as long as 
$f_{m,n}^{\langle\kappa\rangle}$
is uniformly bounded on $\D_\e(0,\rho_0)$. By expanding the expression
\eqref{eq:fnm} for $f_{m,n}^{\langle\kappa\rangle}$, it follows from 
\eqref{eq:asymp-clean} that
\begin{multline}\label{eq:comp-radial}
\int_{\rho_0}^\infty r^{1-l}\chi^2_{\tau, 1}(r)f_{m,n}^{\langle \kappa\rangle}
(r\e^{\imag\theta})
\e^{-2mR_{\tau}(r\e^{\imag\theta})}\diff r 
\\
=\bigg(\frac{\pi}{4m\hDelta R_{\tau}(\e^{\imag\theta})}\bigg)^{\frac12}
\sum_{k=0}^{\kappa}m^{-k}\Lop_k[r^{1-l}f_{m,n}^{\langle \kappa\rangle}
(r\e^{\imag\theta})]\bigg\vert_{r=1}
+ \Ordo(m^{-\kappa-1})
\\
=\bigg(\frac{\pi}{4m\hDelta R_{\tau}(\e^{\imag\theta})}
\bigg)^{\frac12}\sum_{j=0}^\kappa m^{-j}\sum_{k=0}^j
\Lop_k[r^{1-l}B_{\tau, j-k}(r\e^{\imag\theta})]\bigg\vert_{r=1}+ 
\Ordo(m^{-\kappa-1}),
\end{multline}
as $m\to\infty$. We multiply the expression \eqref{eq:comp-radial} by 
$\e^{\imag l\theta}$ and integrate with respect to $\theta$ to get
\begin{multline*}
m^{\frac14}\int_{\T}\e^{\imag l\theta}
\int_{\rho_0}^\infty r^{1-l}\chi_{\tau, 1}^2
(r)f_{m,n}^{\langle\kappa\rangle}(r\e^{\imag\theta})
\e^{-2mR_{\tau}(r\e^{\imag\theta})}\diff r\diffs(\e^{\imag \theta})
\\
=\sum_{j=0}^\kappa m^{-j-\frac14}\int_{\T}\e^{\imag l\theta}\bigg(\frac{\pi}
{4\hDelta R_{\tau}(\e^{\imag\theta})}
\bigg)^{\frac12}\sum_{k=0}^j
\Lop_k[r^{1-l}B_{\tau, j-k}(r\e^{\imag\theta})]
\bigg\vert_{r=1}\diffs(\e^{\imag \theta})+
\Ordo(m^{-\kappa-\frac34}),
\end{multline*}
as $m\to\infty$. This is an asymptotic series, and so is \eqref{eq:pol-f}, 
only that all the coefficients vanish in the latter, 
and only the error term remains. Since two 
asymptotic series coincide only if they coincide term 
by term, we find that for 
integers $j=0,\ldots, \kappa$,
$$
\int_{\T}\e^{\imag l\theta}\big(4\hDelta R_{\tau}(\e^{\imag\theta})\big)^{-
\frac12}\sum_{k=0}^{j}\Lop_k[r^{1-l} B_{\tau, j-k}(r\e^{\imag\theta})]
\bigg\vert_{r=1}\diffs(\e^{\imag \theta})=0,\qquad l=1,2,3,\ldots.
$$
This condition looks like the standard condition membership in the Hardy space 
$H^2$. The problem with this is that the functions unfortunately depend on the 
parameter $l$, so the criterion does not apply. To remedy this, we apply
Lemma~\ref{lem:pseudo}, which gives
\begin{equation}\label{eq:Hardy-Mop}
\int_{\T}\e^{\imag l\theta}\sum_{k=0}^{j}\Mop_k[B_{\tau, j-k}](\e^{\imag \theta})
\diffs(\e^{\imag\theta})=0,
\qquad l=1,2,3,\ldots,
\end{equation}
which is now of the desired form. So, by the standard Fourier analytic 
characterization of the Hardy space, the equation \eqref{eq:Hardy-Mop} is 
equivalent to having
\begin{equation}\label{eq:Hardy-Mop2}
\sum_{k=0}^{j}\Mop_k[B_{\tau, j-k}]\Big\vert_{\T}\in H^2,\qquad j=0,\ldots,\kappa.
\end{equation}
We look at the case $j=0$ first. Then \eqref{eq:Hardy-Mop2} says that 
$\Mop_0[B_{\tau, 0}]\big\vert_{\T}\in H^2$. The operator $\Mop_0$, 
with the defining property given by Lemma~\ref{lem:pseudo}, has the form
\begin{equation}\label{eq:def-M0}
\Mop_0[f](\e^{\imag\theta})= (4\hDelta R_\tau(\e^{\imag\theta}))^{-\frac12}
f(\e^{\imag\theta}).
\end{equation}
We recall that it is given that $B_{\tau, 0}$ is bounded and holomorphic 
in a neighborhood of the closed exterior disk $\bar{\D}_{\e}$, so that in 
particular 
$B_{\tau, 0}\big\vert_{\T}\in H^2_{-}$. If we combine this with the 
observation that 
$\Mop_0[B_{\tau, 0}]\big\vert_{\T}\in H^2$ together with the explicit 
expression \eqref{eq:def-M0} for $\Mop_0$, we arrive at
\begin{equation}\label{eq:crit-B0}
B_{\tau, 0}\big\vert_{\T}\in (4\hDelta R_\tau)^{\frac12}H^2\cap H^2_{-}.
\end{equation}
Let $H_{R_\tau}$ be the 
bounded holomorphic function in $\D_\e$ such that 
\begin{equation}\label{eq:HR}
\Re H_{R_\tau}=\frac12\log (4\hDelta R_\tau)^{\frac12}=\frac14
\log(4\hDelta R_\tau),\qquad \text{on}\;\;\T
\end{equation}
with $\Im H_{R_\tau}(\infty)=0$.
It follows from the given regularity of $R_\tau$ that $H_{R_\tau}$ is a 
bounded holomorphic function in the exterior disk, which extends 
holomorphically 
to a neighborhood of $\bar\D_\e$. We may rewrite \eqref{eq:crit-B0} in the form
$$
B_{\tau, 0}\Big\vert_{\T} \in \e^{2\Re H_{R_\tau}} H^2\cap H^2_{-}.
$$
By Proposition~\ref{prop:toeplitz-ker} applied with $u=v=-\bar{H}_{R_\tau}$ and 
$F=0$, 
it follows that $B_{\tau, 0}$ is of the form
\begin{equation}\label{eq:B0-struct}
B_{\tau, 0}=c_{\tau, 0}\e^{H_{R_\tau}}
\end{equation}
for some constant $c_{\tau, 0}$, which must be positive by 
our normalization.

We proceed to consider more generally $j=1,2,3,\ldots$. 
If we separate out the term corresponding to $k=0$ from 
equation \eqref{eq:Hardy-Mop2}, we 
find that
\begin{equation}\label{eq:orth-1}
\frac{B_{\tau, j}}{(4\hDelta R_{\tau})^{\frac12}}
+\sum_{k=1}^{j}\Mop_k[B_{\tau, j-k}]\bigg\vert_{\T}
\in H^2,\qquad j=1,\ldots, \kappa.
\end{equation}
This equation allows us to compute $B_{\tau, j}$, given that we have already 
obtained the functions $B_{\tau, 0},\ldots, B_{\tau,j-1}$.
Indeed, if we put
\[
F_{\tau, j}=\sum_{k=1}^{j}\Mop_k[B_{\tau, j-k}],
\]
which involves only the functions $B_{\tau, 0},\ldots, B_{\tau, j-1}$, we may write 
\eqref{eq:orth-1} in the form
\[
B_{\tau, j}\big\vert_{\T}\in H^2_{-}\cap 
(4\hDelta R_\tau)^{\frac12}(-F_{\tau, j}+H^2)=
H^2_{-}\cap \e^{2\Re H_{R_\tau}}(-F_{\tau, j}+H^2),
\]
which by Proposition~\ref{prop:toeplitz-ker} has the solution
\begin{equation}\label{eq:Bj-struct}
B_{\tau, j}=c_{\tau, j}\e^{H_{R_\tau}}-\e^{H_{R_\tau}}\Pop_{H^2_{-,0}}
[\e^{\bar{H}_{R_\tau}}F_{\tau, j}],
\end{equation}
for some constant $c_{\tau, j}$, which have to be real
in view of our normalization $f_{m,n}^{\langle\kappa\rangle}(\infty)>0$. 
Since $B_{\tau, 0}$ is known up to a constant 
multiple, this allows us to iteratively derive $B_{\tau, j}$ for 
$j=1,\ldots,\kappa$.
The only remaining freedom is the choice of the constants $c_{\tau, j}$ for 
$j=0,\ldots,\kappa$. We proceed to determine them. 
Since the orthogonal polynomials $P_{m,n}$ are normalized, it follows from 
Theorem~\ref{main} together with the triangle inequality that
$$
\big\lVert \chi_{\tau, 0}F_{m,n}^{\langle\kappa\rangle}
\big\rVert_{2mQ}=1+\Ordo(m^{-\kappa-1})
$$
as $m\to\infty$. Since $\chi_{\tau, 0}F_{m,n}^{\langle\kappa\rangle}
=m^{\frac14}\Vop_{m,n}
[\chi_{\tau, 1}f_{m,n}^{\langle\kappa\rangle}]$, 
it follows from the isometric property 
described in Proposition~\ref{prop:Vop} that
\begin{equation}\label{eq:norm-approx}
m^{\frac12}\int_{\C}\chi_{\tau, 1}^2\lvert f_{m,n}^{\langle\kappa\rangle}
\rvert^2\e^{-2mR_\tau}\diffA=\int_{\C}\chi_{\tau, 0}^2\lvert F_{m,n}
^{\langle\kappa\rangle}\rvert^2\e^{-2mQ}
\diffA=1
+\Ordo(m^{-\kappa-1}).
\end{equation}
Here, the integrals are over the whole plane, although the isometry is 
only over the 
the complements of certain compact subsets. However, since we interpret the 
products with the cut-off functions as vanishing where the cut-off function 
vanishes itself, this
is of no concern to us.
We now expand $f_{m,n}^{\langle\kappa\rangle}$ according to \eqref{eq:fnm}, 
so that by equation \eqref{eq:norm-approx},
\begin{multline}\label{eq:norm-approx2}
2m^{\frac12}\sum_{j, k=0}^{\kappa}m^{-(j+k)}\int_{\T}\int_{\rho_0}^\infty
\chi_{\tau, 1}^2(r)B_{\tau, j}(r\e^{\imag\theta})
\bar{B}_{\tau, k}(r\e^{\imag\theta})\e^{-2mR_\tau(r\e^{\imag\theta})}r\diff
r\diffs(\e^{\imag\theta})
\\
=1+\Ordo(m^{-\kappa-1}),
\end{multline}
where the factor $2$ appears as a result of our normalizations.
This equation is what will give us the values of the constants $c_{\tau, j}$. 
We turn 
first to the case $j=0$.
By a trivial version of Proposition~\ref{prop:steep2}, for any integers $j,k$
with $0\le j,k\le \kappa$ we have the rough estimate
$$
\int_{\rho_0}^\infty \chi_{\tau, 1}^2(r)B_{\tau, j}(r\e^{\imag\theta})
\bar{B}_{\tau, k}(r\e^{\imag\theta})\e^{-2mR_\tau(r\e^{\imag\theta})}r\diff 
r\diffs(r\e^{\imag\theta})=\Ordo(m^{-\frac12}),
$$
where the implicit constant is uniform for $\tau\in I_{\epsilon_0}$. If we 
disregard all the contributions in \eqref{eq:norm-approx2} which are of
order $\Ordo(m^{-\frac12})$, we see that only $j=k=0$ gives a nontrivial 
contribution. The term corresponding to $j=k=0$ in \eqref{eq:norm-approx2}
can be expanded using the Laplace method of Proposition~\ref{prop:steep2}
(recall the formula \eqref{eq:B0-struct} for $B_{\tau, 0}$), to give
\begin{multline*}
2m^{\frac{1}{2}}\int_{\T}\int_{\rho_0}^\infty \chi_{\tau, 1}^2(r)\lvert B_{\tau, 0}
(r\e^{\imag\theta})\rvert^2\e^{-2mR_\tau(r\e^{\imag\theta})}
r\diff r\diffs(\e^{\imag\theta})
\\
=2m^{\frac{1}{2}}\lvert c_{\tau, 0}\rvert^2
\int_{\T}\Big(\frac{\pi}{4m\hDelta R_{\tau}
(\e^{\imag\theta})}\Big)^{\frac12}
\Lop_0[r\e^{2\Re H_{R_\tau}(r\e^{\imag\theta})}]\Big\vert_{r=1}\diffs
+\Ordo(m^{-\frac12}).
\end{multline*}
Since in general, for a smooth function $f$ we have that 
$\Lop_0[f(r)]\big\vert_{r=1}=f(1)$, the leading contribution simplifies to 
(recall the definition \eqref{eq:HR} of $H_{R_\tau}$),
\begin{multline*}
2m^{\frac{1}{2}}\lvert c_{\tau, 0}\rvert^2
\int_{\T}\Big(\frac{\pi}{4m\hDelta R_{\tau}
(\e^{\imag\theta})}\Big)^{\frac12}
\Lop_0[r\e^{2\Re H_{R_\tau}(r\e^{\imag\theta})}]\Big\vert_{r=1}\diffs 
\\
=2\pi^{\frac12}\lvert c_{\tau, 0}\rvert^2\int_{\T}\big(4\hDelta R_{\tau}
(\e^{\imag\theta})\big)^{-\frac12}
\e^{2\Re H_{R_\tau}(\e^{\imag\theta})}\diffs(\e^{\imag\theta})
\\
=2\pi^{\frac12}\lvert c_{\tau, 0}\rvert^2\int_{\T}
\diffs(\e^{\imag\theta})=2\pi^{\frac12}\lvert c_{\tau, 0}\rvert^2.
\end{multline*}
As this is the leading contribution to \eqref{eq:norm-approx2}, we must have 
$2\pi^{\frac12}\lvert c_{\tau, 0}\rvert^2=1$. This determines the constant 
$c_{\tau, 0}$
up to a unimodular factor, and by positivity we find that
$c_{\tau, 0}=(4\pi)^{-\frac14}$. 

We turn to the remaining coefficients $c_{\tau, j}$, for $j=1,\ldots,\kappa$. By 
applying the Laplace method of Proposition~\ref{prop:steep} to the radial 
integral in 
the formula \eqref{eq:norm-approx2}, we arrive at
$$
2\pi^{\frac12}\sum_{j= 0}^{\kappa}m^{-j}\sum_{(i,k,l)\in\indset^\star_j}
\int_{\T}(4 \hDelta R_{\tau}(\e^{\imag\theta}))^{-\frac12}
\Lop_{k}[rB_{\tau, i}(r\e^{\imag\theta})\bar{B}_{\tau, l}(r\e^{\imag\theta})]
\Big\vert_{r=1}
\diffs(\e^{\imag\theta})=1+\Ordo(m^{-\kappa-\frac12}),
$$
where the index set is $\indset_j^\star:=
\big\{(i,k,l)\in \mathbb{N}^3\;:\; i+k+l=j\big\}$. Here,  
$\N=\{0,1,2,\ldots\}$ as usual.
As this represents an equality of asymptotic series, we may identify 
term by term. 
The term with $j=0$ was already analyzed, and it follows that for 
$j=1,\ldots,\kappa$ we have
\begin{multline}\label{eq:norm:approx3}
\sum_{(i,k,l)\in\indset^\star_j}
\int_{\T}(4 \hDelta R_{\tau}(\e^{\imag\theta}))^{-\frac12}
\Lop_{k}[rB_{\tau, i}(r\e^{\imag\theta})\bar{B}_{\tau, l}(r\e^{\imag\theta})]
\Big\vert_{r=1}\diffs (\e^{\imag\theta}) 
\\
=2\Re\int_{\T}(4 \hDelta R_{\tau}(\e^{\imag\theta}))^{-\frac12}
\Lop_{0}[rB_{\tau, j}(r\e^{\imag\theta})\bar{B}_{\tau, 0}(r\e^{\imag\theta})]
\Big\vert_{r=1}\diffs(\e^{\imag\theta})
\\
+\sum_{(i,k,l)\in\indset_j}
\int_{\T}(4 \hDelta R_{\tau}(\e^{\imag\theta}))^{-\frac12}
\Lop_{k}[rB_{\tau, i}(r\e^{\imag\theta})\bar{B}_{\tau, l}(r\e^{\imag\theta})]
\Big\vert_{r=1}\diffs(\e^{\imag\theta})=0,
\end{multline}
where $\indset_j$ denotes the restricted index set 
$\indset_j:=\big\{(i,k,l)\in\indset_j^\star\,:\, 
i,l<j\big\}$,
and where we separate out the terms involving the leading term $B_{\tau, j}$.
We successfully resolve the first term on the right-hand side of 
\eqref{eq:norm:approx3}, while the second term is much more complicated.
However, we may observe that it only depends on the functions $B_{\tau,\nu}$ with 
$\nu=0,\ldots,j-1$, and hence only on the constants $c_{\tau,\nu}$ with 
$\nu=0,\ldots,j-1$. This allows us to algorithmically determine these 
constants, 
albeit with increasing degree of complexity.
As for the first term on the right-hand side, we observe that the operator 
$\Lop_{0}\big\vert_{r=1}$ only evaluates at $r=1$. Using
the structure of $B_{\tau, j}$ as given by \eqref{eq:Bj-struct}, we find that
\begin{multline*}
\int_{\T}(4 \hDelta R_{\tau}(\e^{\imag\theta}))^{-\frac12}
\Lop_{0}[rB_{\tau, j}(r\e^{\imag\theta})\bar{B}_{\tau, 0}(r\e^{\imag\theta})]
\Big\vert_{r=1}\diffs(\e^{\imag\theta})
\\
=\int_{\T}\big(4 \hDelta R_\tau(\e^{\imag\theta})\big)^{-\frac12}
B_{\tau, j}(\e^{\imag\theta})\bar{B}_{\tau, 0}(\e^{\imag\theta})\diffs(\e^{\imag\theta})
\\
=c_{\tau, 0}\int_{\T}\big(4\hDelta R_\tau(\e^{\imag\theta})\big)^{-\frac12}
\e^{2\Re H_{R_\tau}(r\e^{\imag\theta})}\big(c_{\tau, j}-
\Pop_{H^2_{-,0}}[\e^{\bar{H}_{R_\tau}} F_{\tau, j}](\e^{\imag\theta})\big)
\diffs(\e^{\imag\theta})
\\
=c_{\tau, 0}\int_{\T}\big(c_{\tau, j}-
\Pop_{H^2_{-,0}}[\e^{\bar{H}_{R_\tau}} F_{\tau, j}](\e^{\imag\theta})\big)
\diffs(\e^{\imag\theta})=c_{\tau, 0}c_{\tau, j}.
\end{multline*}
Here we use the definition \eqref{eq:HR} of $H_{R_\tau}$ and the 
fact that the projection $\Pop_{H^2_{-,0}}$ maps into a subspace of functions
with mean $0$. Assume now that
$j$ is given, and that we have determined $c_{\tau, k}$ for $k=0,\ldots,j-1$. The 
above equality together with 
\eqref{eq:norm:approx3} then gives that 
\[
2\Re c_{\tau, j}c_{\tau, 0}=-\sum_{(i,k,l)\in\indset_j}
\int_{\T}\big(4\hDelta R_{\tau}(\e^{\imag\theta})\big)^{-\frac12}
\Lop_{k}[rB_{\tau, i}(r\e^{\imag\theta})\bar{B}_{\tau, l}(r\e^{\imag\theta})]
\Big\vert_{r=1}\diffs(\e^{\imag\theta}).
\]
Since $c_{\tau, 0}=(4\pi)^{-\frac14}$ and moreover since the constants
$c_{\tau,j}$ must be real by our normalization, we obtain that
\begin{equation*}
c_{\tau, j}=-\frac12(4\pi)^{\frac14}\sum_{(i,k,l)\in\indset_j}
\int_{\T}\big(4\hDelta R_{\tau}(\e^{\imag\theta})\big)^{-\frac12}
\Lop_{k}[rB_{\tau, i}(r\e^{\imag\theta})\bar{B}_{\tau, l}(r\e^{\imag\theta})]
\Big\vert_{r=1}\diffs(\e^{\imag\theta}),
\end{equation*}
where the integral may be expressed in terms of the operator $\Mop_k$ by
\[
\int_{\T}\big(4\hDelta R_{\tau}(\e^{\imag\theta})\big)^{-\frac12}
\Lop_{k}[rB_{\tau, i}(r\e^{\imag\theta})\bar{B}_{\tau, l}(r\e^{\imag\theta})]
\Big\vert_{r=1}\diffs(\e^{\imag\theta})
=\int_\T\Mop_k\big[B_{\tau, i}(r\e^{\imag\theta})\bar{B}_{\tau, l}(r\e^{\imag\theta})\big]
\diffs(\e^{\imag\theta}).
\]
This completes the proof.
\end{proof}

\section{Applications to random matrix theory}
\label{sect:erfc}
\subsection{The random normal matrix model} \label{ss:rnm}
For extensive treatments of the random normal matrix ensembles, see 
e.g.\ \cite{HM, ahm2, ahm3, akm, AKMW, WZ-conj}.
Here we only briefly discuss the topic, in order to fix the notation and 
recall some basic concepts.

Let $M$ be a matrix, picked with respect to the probability measure
(``tr'' stands for trace)
$$
\diff\mu_m(M)=\frac{1}{Z_{m,Q}}\,\e^{-2m\operatorname{tr}(Q(M))}\diff M,
$$
where $\diff M$ denotes the measure induced by the flat Euclidean metric of
$\C^{m^2}$ on the submanifold of normal $m\times m$ matrices, where
$Z_{m,Q}$ is a normalizing constant. 
Such a matrix $M$ has a set of $m$ random eigenvalues, which we 
denote by $\Phi_m=\{z_{1,m},\ldots, z_{m,m}\}$. 
It is known that the eigenvalues follow the law
\begin{equation}\label{eq:law}
\diff \mathbb{P}_m(z_1,\ldots, z_m)=\frac{1}{\mathcal{Z}_{m,Q}} 
\bigg[\prod_{j<k}\lvert z_j-z_k\rvert^2\bigg] \e^{-2m\sum_{j=1}^{m}Q(z_j)}
\diffA^{\otimes n}(z_1,\ldots,z_m),
\end{equation}
where $\mathcal{Z}_{m,Q}$ is a related normalizing constant, known as the 
{\em partition function} of the ensemble. Here, $\diffA^{\otimes n}$ stands 
for Euclidean volume measure in $\C^n$ normalized by the factor $\pi^{-n}$. 
We recognize this as the law for the Coulomb gas with $m$ particles at the 
inverse temperature $\beta=2$ in the external field $Q$. 
Courtesy of the fact that the product expression in \eqref{eq:law} may
be written as the square modulus of a Vandermondian determinant, 
these ensembles are determinantal. That is, if the $k$-point intensities 
$R_{k,m}(z_1,\ldots z_k)$ are defined as the intensities associated to finding
points simultaneously at the locations $z_1,\ldots,z_k$, then we may 
compute $R_{k,m}$ by 
\begin{equation}\label{eq:det-int}
R_{k,m}(z_1,\ldots, z_k)= 
\det\left({\Kcorrker}_m(z_j,z_l)\right)_{1\leq j,l\leq k}.
\end{equation}
Here ${\Kcorrker}_m$ is the {\em correlation kernel}
$$
{\Kcorrker}_m(z,w)=K_m(z,w)\,\e^{-m(Q(z)+Q(w))},\qquad z,w\in\C
$$
where $K_m$ is the reproducing kernel for the space 
$\operatorname{Pol}_m$, supplied with the inner product of the space
$L^2_{2mQ}(\C)$. We remark that the correlation kernel ${\Kcorrker}_m$  
is not uniquely determined by the above-mentioned intensities, 
since any kernel modified by a cocycle
$$
{\Kcorrker}_m^{c}(z,w)=c(z)\bar{c}(w){\Kcorrker}_m(z,w),
$$
will generate the same point process by the determinantal formula 
\eqref{eq:det-int}. Here, the cocycle is associated with 
a continuous unimodular function $c:\C\to\T$. 
This means that in terms of convergence of point processes, we need only
correlation kernel convergence modulo cocycles.
It is known (see \cite{HM, WZ-conj}) that the process $\Phi_m$ 
condensates to the droplet
$\calS_1$ as $m\to+\infty$. Indeed, if $\nu_m$ denotes the 
empirical measure
$$
\nu_m=\frac{1}{m}\sum_{z\in\Phi_m}\delta_{z},
$$
then almost surely, $\nu_m$ converges weakly to the equilibrium measure 
$\mu_{\tau}$ with $\tau=1$, the support of which equals $\calS_1$.
We rescale the point process near a 
boundary point $z_0$, in the outer normal direction ${\rm n}$,
in order to understand the microscopic behavior of $\Phi_m$.
To rescale we use the linear transformation
$$
z_m(\zeta):=z_0+{{\rm n}}
\frac{\zeta}{\sqrt{2m\hDelta Q(z_0)}}.
$$
Writing $\Phi_m=\{z_{j,m}\}_j$,
we introduce the rescaled local process by
$\Psi_m=\{\zeta_{j,m}\}_j$, where  
$$
z_{j,m}=z_m(\zeta_{j,m}),\qquad j=1,\ldots, m.
$$
Similarly, we denote by $\kcorrker_m$ the rescaled correlation kernel
$$
\kcorrker_m(\xi,\eta)=\frac{1}{2m\hDelta Q(z_0)}
{\Kcorrker}_m(z_m(\xi), z_m(\eta)).
$$
We recall the familiar notion that a function $F(\xi,\eta)$ 
is {\em Hermitian entire} if it is an entire function of the two variables 
$(\xi,\bar\eta)$ with the symmetry property $F(\xi,\eta)=\bar F(\eta,\xi)$. 
The following is from \cite{akm}. 

\begin{thm}\label{thm:conv-cocycle}
There exists a sequence of continuous unimodular functions $c_m: \C\to\T$,
such that for any given infinite sequence of positive integers $\calN$, there exist
an infinite subsequence $\calN^*\subset\calN$ and an Hermitian entire function 
$F(\xi,\eta)$ such that
$$
\lim_{\calN^*\ni m\to\infty}c_m(\xi)\bar{c}_m(\eta)\,\kcorrker_{m}(z_m(\xi), 
z_m(\eta))= \e^{\xi\bar\eta-\frac{1}{2}(|\xi|^2+|\eta|^2)}F(\xi,\eta).
$$
\end{thm}

\subsection{Uniform asymptotics near
\texorpdfstring{$\tau=1$}{tau=1}}
We take as our starting point the first term of the asymptotic expansion of 
Theorem~\ref{thm:main-pw}.
Recall the definition of the compact set $\calK_{\tau,A,m}$ in
\eqref{eq:def-KtAm}.

\begin{cor}\label{cor:asymp-simple}
Let $\mathcal{H}_{Q,\tau}$ be the bounded holomorphic function in the set
$\calK_{\tau}^c$ with real part $\Re \mathcal{H}_{Q,\tau}=
\frac{1}{4}\log (2\hDelta Q)$ 
on the boundary $\partial\calS_\tau$, which is real-valued at infinity.
Then, in the limit as $m,n\to\infty$ while 
$\tau=\frac{n}{m}\in I_{\epsilon_0}$, we have the 
asymptotics
$$
\lvert P_{m,n}(z)\rvert^2\e^{-2mQ(z)}=
m^{\frac12}\lvert \phi_\tau'(z)\rvert \,\e^{-2m(Q-\breve{Q}_\tau)(z)}
\Big(\pi^{-\frac12}\e^{2\Re\calH_{Q,\tau}(z)}+\Ordo\big(m^{-1}\big)\Big),
$$
where the implied constant is uniform for $z\in\calK_{\tau,A,m}^c$.
\end{cor}

\begin{proof}
We recall that 
$$
\breve{Q}_\tau=\Re\calQ_\tau + \tau\log\lvert \phi_\tau\rvert=\Re\calQ_\tau
+\tfrac{n}{m}\log\lvert\phi_\tau\rvert,
$$
and in view of Theorems~\ref{thm:main-pw} and 
\ref{thm:main-coeff}, we may write
\begin{multline*}
\lvert P_{m,n}\rvert^2=m^{\frac12}\lvert \phi_\tau'(z)\rvert
\lvert \phi_{\tau}\rvert^{2n}\e^{2m\Re\calQ_\tau}
\big\lvert \calB_{\tau, 0}+\Ordo(m^{-1})\big\rvert^2
\\
=m^{\frac12}\lvert \phi_\tau'(z)\rvert
\,\e^{2m\breve{Q}_\tau}
\Big(\pi^{-\frac12}\e^{2\Re\calH_{Q,\tau}(z)}+\Ordo(m^{-1})\Big),
\end{multline*}
and the assertion follows.
\end{proof}

\subsection{Error function asymptotics}
In view of Corollary~\ref{cor:asymp-simple}, we observe that the 
probability density $\lvert P_{m,n}\rvert^2\e^{-2mQ}$ resembles a Gaussian 
wave which crests around the boundary $\partial\calS_\tau$ of the droplet, 
where $\tau=\frac{n}{m}$. As a consequence, we expect the density to be 
obtained as the sum of such Gaussians. Near the droplet boundary, this 
effect is the strongest, and adding a large but finite number of such 
Gaussian waves crested along boundary curves $\partial\calS_{\tau}$ which 
move with the degree parameter $n$ results in error function asymptotics.

\medskip

\begin{prop}\label{prop:erf}
 If $Q$ is $1$-admissible and $z_0\in\partial\calS_1$ is a boundary point, 
then if $\rho_m$ is the blow-up density given by \eqref{eq:resc1} 
and \eqref{eq:resc2}, we have the convergence
$$
\lim_{m\to\infty}\rho_m(\zeta)=\mathrm{erf}\,({2}\zeta),
$$
locally uniformly on $\C$.
\end{prop}

\begin{proof}
We recall the rescaled variable from the introduction
$$
z_m(\xi)=z_0+{\rm n}\frac{\xi}{\sqrt{2m\hDelta Q(z_0)}},
$$
where $z_0\in\partial\calS_\tau$ and ${\rm n}$ is the outward 
unit normal to $\calS_\tau$ at $z_0$, and the rescaled density $\rho_m(\xi)$ 
given by \eqref{eq:resc2}.
In terms of orthogonal polynomials, the object of study is the function
$$
\rho_m(\xi)=\frac{1}{2m\hDelta Q(z_0)}\sum_{n=0}^{m-1}
\lvert P_{m,n}(z_m(\xi))\rvert^2 \e^{-2mQ(z_m(\xi))}.
$$
We begin by noting that $z_m(\xi)$ is in the set
$\calK_{\tau,A,m}^c$ (see Theorem~\ref{thm:main-pw}), 
provided that $\xi$ is confined to the disk $\D(0,r_m)$, where
$r_m=A\sqrt{\hDelta Q(z_0) \,\log m}$, and that $m$ is large enough.
We shall assume throughout that $\xi\in\D(0,r_m)$.

Next, we write
$$
\rho_{m_1,m}(\xi)=\frac{1}{2m\hDelta Q(z_0)}\sum_{n=0}^{m_1-1}
\lvert P_{m,n}(z_m(\xi))\rvert^2 \e^{-2mQ(z_m(\xi))}
$$
and split accordingly for $m_1< m$
\begin{equation}\label{eq:1212}
\rho_m(\xi)=\frac{1}{2m\hDelta Q(z_0)}\sum_{n=m_1}^{m-1}
\lvert P_{m,n}(z_m(\xi))\rvert^2 \e^{-2mQ(z_m(\xi))}+\rho_{m_1,m}(\xi).
\end{equation}
We choose $m_1$ to be the integer part of $m-m^{\frac12}\log m$. 

By Proposition~\ref{prop-gengrowth-ext} it follows that for $n\le m_1$,
\begin{equation}
\label{eq:0101}
\rvert P_{m,n}(z)\rvert^2\e^{-2mQ(z)}\le Cm\,\e^{-2m(Q-\checkQ_{\tau_1})(z)},
\end{equation}
where $\tau_1=\frac{m_1}{m}\in I_{\epsilon_0}$ for $m$ large enough.
By Taylor's formula applied to $Q-\check{Q}_{\tau_1}=
R_{\tau_1}\circ\phi_{\tau_1}$ in $\calS_{\tau_1}^c$ (Proposition~\ref{prop:R}), 
it follows that
\begin{equation}\label{eq:10001}
(Q-\checkQ_{\tau_1})(z)\ge \beta_0\mathrm{dist}_\C
(z,\partial \calS_{\tau_1})^2
\end{equation}
for some constant $\beta_0>0$, provided that $z\in\calS_{\tau_1}^c$ is 
close enough to $\partial\calS_{\tau_1}$. For instance, this estimate holds for 
$z\in\calS_1\setminus\calS_{\tau_1}$.
Moreover, as $\tau_1=\frac{m_1}{m}$ eventually is in $I_{\epsilon_0}$, the 
function $Q-\checkQ_{\tau_1}$ does not vanish on $\calS_{\tau_1}^c$, and tends 
to infinity at infinity. The latter observation shows that further away 
from the boundary $\partial\calS_{\tau_1}$, the right-hand side of 
\eqref{eq:0101} decays exponentially.

If $n\le m_1$ and $\tau=\frac{n}{m}$, then 
$1-\tau\ge m^{-\frac12}\log m=\delta_m$.
As a consequence of Lemma~\ref{lem:boundary-movement} we obtain that 
the boundary $\partial\calS_\tau$ moves at a positive speed in $\tau$.
In particular, for $\tau=\frac{n}{m}$ where $n\le m_1$ we have that the 
distance $\mathrm{dist}_\C(\partial\calS_{\tau},\partial\calS_1)$ is at least 
$2\alpha_0\delta_m$, for some fixed positive $\alpha_0$.
Since $\mathrm{dist}_\C(\partial\calS_{\tau_1},\partial\calS_1)$ is at least 
$2\alpha_0\delta_m$, we have that
\begin{equation}\label{eq:1012}
\mathrm{dist}_\C(z, \partial\calS_{\tau_1})\ge \alpha_0\delta_m,\qquad 
z\in\D(z_0,\alpha_0 \delta_m).
\end{equation}
Next, we note that if $\zeta\in\D(0,r_m)$, then for large enough $m$ we have
$z_m(\zeta)\in\D(z_0,\alpha_0 \delta_m)$. This follows from the obvious fact
that $(\log m)^{\frac12}=\ordo(\log m)$.
By a combination of \eqref{eq:10001} and \eqref{eq:1012} it follows that
$$
(Q-\checkQ_{\tau_1})(z_m(\zeta))\ge \beta_0\alpha_0^2\delta_m^2.
$$
Now, it follows from the above estimates \eqref{eq:0101} and 
\eqref{eq:10001} that for $n\le m_1$
$$
\rvert P_{m,n}(z_m(\xi))\rvert^2\e^{-2mQ(z_m(\xi))}
=\Ordo(m\,\e^{-2\beta_0\alpha_0^2(\log m)^2}),
$$
where the implicit constant is uniform in $\xi\in\D(0,r_m)$.
It follows that
$$
\rho_{m_1,m}(\xi)=\Ordo\big(m^2\e^{-\beta_0\alpha_0^2(\log m)^2}\big),\qquad 
\xi\in\D(0,r_m)
$$
which shows in particular $\rho_{m_1,m}(\xi)=\Ordo(m^{-M})$ for arbitrarily 
large $M$.

As a result of the above considerations, it follows that we may focus on 
the remaining sum in \eqref{eq:1212} over the degrees $n$ with 
$m_1\le n\le m-1$, that is, 
$\tau=\frac{n}{m}$ with $\tau_1 \le \tau\le 1$. In particular, the 
asymptotics of Corollary~\ref{cor:asymp-simple} applies in the whole range. 
Set $\tau(j)=\tau_m(j)=1-\tfrac{j}{m}$, where $j$ ranges from $1$ to 
$m-m_1$, which is approximately $m^{\frac12}\log m$.
We obtain
\begin{multline}\label{eq:density-1}
\rho_m(\xi)=\frac{(\pi m)^{-\frac12}}{2\hDelta Q(z_0)}\sum_{j=1}^{m-m_1}
\big\lvert \phi'_{\tau(j)}(z_m(\xi))\big\rvert 
\,\e^{-2m(Q-\breve{Q}_{\tau(j)})(z_m(\xi))+2\Re\calH_{Q,\tau(j)}(z_m(\xi))}
\\+\Ordo(m^{-M}).
\end{multline}
By Taylor's formula, it follows that 
$$
\lvert \phi'_{\tau(j)}(z_m(\xi))\rvert=\lvert 
\phi_1'(z_0)\rvert+\Ordo((m^{-1}\log m)^{\frac12}),
$$
and by the same token that
$$
2\Re\calH_{Q,\tau(j)}(z_m(\xi))=
\frac{1}{2}\log \hDelta Q(z_0)+\Ordo((m^{-1}\log m)^{\frac12})
$$
as $m\to\infty$ for all $j\leq m-m_1$.
The next thing to consider is the movement of $\partial\calS_\tau$, where 
$\tau=\tau(j)$ and $j$ increases. 
As ${\rm n}$ denotes the outward pointing unit normal to 
$\partial\calS_1$
at the point $z_0$, Lemma~\ref{lem:boundary-movement} tells us that the line
$z_0+{\rm n}\mathbb{R}$ intersects $\partial\calS_{\tau(j)}$ at the nearest point
$$
z_{j}=z_0- {\rm n}\frac{j}{m}
\frac{\lvert \phi_1'(z_0)\rvert}{4\hDelta Q(z_0)}
+\Ordo\Big(\Big(\frac{j}{m}\Big)^2\Big),
$$
and the outer unit normal ${\rm n}_j$ to 
$\partial\calS_{\tau(j)}$ at the point 
$z_j$ will satisfy 
$$
{\rm n}_j={\rm n}+\Ordo\Big(\frac{j}{m}\Big)
={\rm n}+\Ordo(m^{-\frac12}\log m).
$$
We may hence write
\[
(Q-\breve{Q}_{\tau(j)})(z_m(\xi))=(Q-\breve{Q}_{\tau(j)})
\Big(z_j+{\rm n}_{j}
\frac{\xi+\frac{j}{2}\frac{\lvert\phi'_1(z_0)\rvert}{\sqrt{2m\hDelta Q(z_0)}}
+\Ordo\big(m^{-\frac12}
(\log m)^2\big)}
{\sqrt{2m\hDelta Q(z_0)}}\Big).
\]
A simple Taylor series expansion in normal and tangential coordinates at 
the point $z_j$ gives that 
$$
(Q-\breve{Q}_{\tau_j})(z_j+{\rm n}_j\lambda)=
2\hDelta Q(z_j)\,(\Re\lambda)^2+\Ordo(\lvert \lambda\rvert^3)=
2\hDelta Q(z_0)(\Re\,\lambda)^2+
\Ordo\Big(\lvert \lambda\rvert^2\frac{j}{m}+\lvert \lambda\rvert^3\Big),
$$
for $\lambda$ close to $0$. From this we deduce that for $\eta$ with 
$\lvert \eta\rvert=\Ordo(\log m)$ we have
$$
2m(Q-\breve{Q}_{\tau(j)})\Big(z_{j}+
{\rm n}_j\frac{\eta}{\sqrt{2m\hDelta Q(z_0)}}\Big)
=\frac12(2\Re \eta)^2+\Ordo(m^{-1/2}(\log m)^{3}), \quad m\to \infty.
$$
We apply this with $\eta$ given by
$$
\eta=\xi+\frac{j}{2}\frac{\lvert\phi'_1(z_0)\rvert}{\sqrt{2m\hDelta Q(z_0)}}+
\Ordo\big(m^{-\frac12}(\log m)^2\big),
$$
which then gives that
$$
(2\Re \eta)^2=\Big(2\Re \xi+j\frac{\lvert\phi'_1(z_0)\rvert}
{\sqrt{2m\hDelta Q(z_0)}}\Big)^2+\Ordo\big(m^{-\frac12}(\log m)^3\big).
$$
Putting these asymptotic relations together, we find that
\begin{multline}\label{eq:density-2}
\rho_m(\xi)= \frac{1}{\sqrt{2\pi}}
\Big(1+\Ordo\big(m^{-\frac12}(\log m)^3\big)\Big)
\\
\times \sum_{j=1}^{m-m_1}
\frac{\lvert \phi_1'(z_0)\rvert}{\sqrt{2m\hDelta Q(z_0)}}
\;\exp\Big\{-\frac12\Big(2\Re \xi+j\frac{\lvert\phi'(z_0)\rvert}
{\sqrt{2m\hDelta Q(z_0)}}\Big)^2\Big\}
+\Ordo(m^{-M}).
\end{multline}
We recognize immediately \eqref{eq:density-2} as an approximate 
Riemann sum for 
$$
\mathrm{erf}\,({2}\Re \xi)=
\frac{1}{\sqrt{2\pi}}\int_{0}^{\infty}\e^{-\frac12(2\Re\xi +t)^2} \diff t
$$
with respect to a partition of the interval
$[0,\gamma_0\log m]$, with step length $m^{-\frac12}\gamma_0$,
where
$$
\gamma_0=\frac{\lvert\phi'(z_0)\rvert}{\sqrt{2\hDelta Q(z_0)}}.
$$
Since such Riemann sums converge to the corresponding integral 
with small error, this implies that
$$
\lim_{m\to\infty}\rho_m(\xi)=\mathrm{erf}\,\big({2}\Re\xi\big),
$$
which completes the proof.
\end{proof}

\subsection{Convergence of correlation kernels to the
Faddeeva plasma kernel}
\label{ss:kernel}
Finally, we turn to the convergence of the 
rescaled kernels $\kcorrker_m(z_m(\xi),z_m(\eta))$ as $m\to\infty$.
In principle, this should follow from our expansion of the orthogonal 
polynomials,
but to do this directly seems a bit tricky. 
However, given the work of Ameur, Kang, and Makarov \cite{akm}, 
it turns out to be enough to obtain the more straightforward diagonal 
convergence of the correlation kernel. 

\begin{proof}[Proof of Corollary \ref{cor:erf}]
We denote by $G(\xi,\eta)$ the Ginibre-$\infty$ kernel
$$
G(\xi,\eta)=\e^{\xi\bar\eta-\frac12(\lvert \xi\rvert^2+\lvert \eta\rvert^2)},
$$
which is the correlation kernel of a translation invariant planar
point process.
We now present some material from \cite{akm}. An important concept is that of
cocycles. By Theorem~\ref{thm:conv-cocycle}, 
there exists a sequence of continuous functions $c_{m} :\C\to\T$ such that, 
for any subsequence $\calN$ of the natural numbers $\N$, there exists a 
Hermitian entire function $F(\xi,\eta)$ 
and a further subsequence $\calN^\star\subset\calN$ such that
\begin{equation}\label{eq:999}
c_m(\xi)\bar c_m(\eta)\;\kcorrker_{m}(z_m(\xi), z_m(\eta))\to 
G(\xi,\eta)F(\xi,\eta),\quad m\in \calN^\star, m\to\infty,
\end{equation}
where the convergence is uniform on compact subsets of $\C^2$.
For Hermitian entire functions, the diagonal restriction
$F(\xi,\xi)$ determines the function uniquely. Indeed, the 
polarization of the diagonal
restriction gives back our function $F(\xi,\eta)$.
We denote by $\rho(\xi)$ the limiting density
$$
\rho(\xi)=\lim_{m\to\infty, m\in\calN^\star}\;\kcorrker_{m}(z_m(\xi), z_m(\xi))=
G(\xi,\xi)F(\xi,\xi),
$$ 
and since $G(\xi,\xi)\equiv 1$, it follows that $F(\xi,\xi)=\rho(\xi)$. 
By Proposition~\ref{prop:erf} we have that
$$
\rho(\xi)=\mathrm{erf}({2}\Re\xi).
$$
Moreover, by the uniqueness property of diagonal restriction, 
the only possibility 
for the Hermitian entire function is 
$$
F(\xi,\eta)=\mathrm{erf}\,({\xi+\bar\eta}).
$$
This shows that the limit along some subsequence of any given 
sequence of positive integers is always the same. 
We claim that this means that
the whole sequence converges. 
Indeed, in case the convergence \eqref{eq:999} were to fail along the positive 
integers, by a normal families argument, we could distill a sequence $\calN_0$ 
such that the left-hand side of \eqref{eq:999} would converge to something 
else along the subsequence $\calN_0$. This would contradict what we have 
already established, which is that the we have diagonal convergence to
the error function.
The assertion of the corollary follows. 
\end{proof}

\section{The existence of the orthogonal foliation flow}
\label{sec:flow}
\subsection{Smoothness classes and polarization of functions}
\label{ss:flow0}
In order to proceed with less obscuring notation, we consider a smooth 
family of bounded holomorphic functions $f_s(z)$, and
a smooth family of orthostatic conformal 
mappings $\psi_{s,t}$. 
Here, $f_s$ corresponds to $f_{m,n}^{\langle \kappa\rangle}$ where $s=m^{-1}$,
and $\psi_{s,t}$ corresponds to the mappings $\psi_{m,n,t}$ appearing in
Lemma~\ref{lem:main-flow}.
We suppress $\tau$ and $\kappa$ in the notation, because $\kappa$
is thought of as fixed, and we work with uniformity in the parameter $\tau$.
Moreover, we denote by $R$ a weight whose properties 
are analogous to those of $R_\tau$, 
as captured in Definition~\ref{def:classR} below. 

We denote by $\mathbb{A}(\hrho_1,\hrho_2)$ the annulus
$$
\mathbb{A}(\hrho_1,\hrho_2):=\D(0,\hrho_2)\setminus\bar{\D}(0,\hrho_1),
$$
for positive real numbers $\hrho_1$ and $\hrho_2$ with $\hrho_1<\hrho_2$ (notice that
we distinguish between the symbols $\rho$ and $\varrho$).
In addition, for parameters $\hrho_0$ and $\sigma_0$, we
denote by $\hat{\mathbb{A}}(\hrho_0,\sigma_0)$ the $2\sigma_0$-fattened
diagonal annulus in $\C^2$:
\[
\hat{\mathbb{A}}(\hrho_0,\sigma_0):=\big\{(z,w)\in\mathbb{A}(\hrho_0,\hrho_0^{-1})
\times\mathbb{A}(\hrho_0,\hrho^{-1}_0)\;:\,\; \lvert z-w\rvert\le2\sigma_0\big\},
\]

For a real-analytic function $R$ there exists a {\em polarization} $R(z,w)$, 
which is 
holomorphic in $(z,\bar w)$ and has $R(z,z)=R(z)$. This is easy to see using 
convergent local Taylor series expansions of $R(z)$ in the coordinates 
which are the real and imaginary parts, $\Re z$ and $\Im z$. By replacing 
$\Re z$ by $\frac{1}{2}(z+\bar{w})$ and $\Im z$ by
$\frac{1}{2\imag}(z-\bar{w})$ in 
this expansion, we obtain the polarization $R(z,w)$.
We observe that if $R(z,w)$ is such a polarization of a function
$R(z)$ which is real-analytically smooth near the circle $\T$ and in addition
is quadratically flat there, then $R(z)=(1-|z|^2)^2R_0(z)$, where $R_0(z)$
is real-analytic near the circle $\T$. In polarized form, $R(z,w)$ factors as 
\begin{equation}\label{eq:def-R-0}
R(z,w)=(1-z\bar w)^2R_0(z,w),
\end{equation}
where $R_0(z,w)$ is holomorphic in $(z,\bar w)$
in a neighborhood of the part of the diagonal where both variables are
near $\T$. 

\begin{defn}\label{def:classR}
For positive real numbers $\hrho_0,\sigma_0$ where $\hrho_0<1$, we denote by 
$\mathcal{W}(\hrho_0,\sigma_0)$ 
the class of $C^2$-smooth non-negative functions $R$ on $\D_\e(0,\hrho_0)$
such that the following holds:
\begin{enumerate}[(i)]
\item The functions $R$ and $\nabla R$ 
both vanish on $\T$, while $\hDelta R>0$ holds on 
$\T$. 
\item $R$ is real-analytic on $\mathbb{A}(\hrho_0,(\hrho_0)^{-1})$ and 
both $R(z,w)$ and $R_0(z,w)$ given by \eqref{eq:def-R-0} 
polarizes to bounded holomorphic functions 
in $(z,\bar w)$ on the diagonal annulus $\hat{\mathbb{A}}(\hrho_0,\sigma_0)$,
such that $R_0(z,w)$ remains bounded away from $0$ there.
\item In addition,
\[
R_0(z,z)\ge 
\alpha(R)>0,\qquad z\in\mathbb{A}(\hrho_0,\hrho_0^{-1}),
\]
and further away,
$$
\inf_{z\in\D_\e(0,\hrho_0^{-1})}\frac{R(z)}{\log\lvert z\rvert}=\theta(R)>0.
$$
\end{enumerate}
We say that a subset $\frakS\subset\mathcal{W}(\hrho_0,\sigma_0)$ is a 
\emph{uniform family}, provided that for each $R\in\frakS$, the corresponding
$R_0(z,w)$ is uniformly bounded and bounded away from $0$ on 
$\hat{\mathbb{A}}(\hrho_0,\sigma_0)$
while the controlling constants such as $\alpha(R)$ and $\theta(R)$ are 
uniformly bounded away from $0$.
\end{defn}

If a function $f(z,w)$ is holomorphic in $(z,\bar w)$, we may consider 
the associated function 
\begin{equation}\label{eq:f-T-hol}
f_{\T}(z)=f\Big(z,\frac{1}{\bar z}\Big)
\end{equation}
which is then holomorphic in $z$, wherever it is well-defined. We note that 
$f_{\T}(z)=f(z,z)$ on the circle $\T$. We recall the notation of 
Definition~\ref{def:classR}.

\begin{prop}
\label{prop:f-T-hol}
Suppose that $f(z,w)$ is holomorphic in $(z,\bar w)$ on the domain 
$\hat{\mathbb{A}}(\hrho,\sigma)$, where $0<\hrho<1$ and $\sigma>0$. 
Then the function $f_{\T}(z)$, which extends the restriction of the 
diagonal function $f(z,z)$ to $\T$, 
has a holomorphic extension to the annulus
\[
\hrho'<\lvert z\rvert<\frac{1}{\hrho'}
\]
where
\[
\hrho'=\max\Big\{\hrho, (\sqrt{1+\sigma^2}+\sigma)^{-1}\Big\}.
\]
\end{prop}
\begin{proof}
The function $f_{\T}(z)=f(z,\bar{z}^{-1})$ is automatically holomorphic in 
the variable $z$ in the domain 
\[
\Big\lvert z-\frac{1}{\bar{z}}\Big\rvert\le 2\sigma,
\]
provided that 
$z\in\mathbb{A}(\hrho,\hrho^{-1})$. By checking these requirements
carefully, the assertion follows.
\end{proof}

\begin{rem}
We note that if $\hrho$ is close enough to $1$ to guarantee that 
$\hrho\ge (\sqrt{1+\sigma^2}+\sigma)^{-1}$, then $\hrho'=\hrho$.
\label{rem:rho'}
\end{rem}

\begin{rem}
\label{rem:Hop-hol}
Suppose a real-analytic function $F(z)$ admits a polarization $F(z,w)$
which is holomorphic in $(z,\bar w)$ for 
$(z,w)\in\hat{\mathbb{A}}(\hrho,\sigma)$,
and let $f$ be given in terms of the Herglotz kernel by
$f=\Hop_{\D_\e}[F\vert_{\T}]$ (cf.\ Subsection~\ref{ss:herglotz}).
We note that by the properties of the Herglotz kernel, 
$f$ may be obtained by the formula
$f=2\Pop_{H^2_{-}}[F\vert_{\T}]-\langle F\rangle_{\T}$, 
where $\langle F\rangle_{\T}$ denotes the average of $F$ on the unit circle. 
Let $F_{\T}$ be as in \eqref{eq:f-T-hol},
and express it in terms of its 
Laurent series, which by Proposition~\ref{prop:f-T-hol} converges in the 
annulus 
$\mathbb{A}(\hrho',(\hrho')^{-1})$:
\[
F_{\T}(z)=\sum_{n\in\mathbb{Z}} a_nz^n.
\]
In terms of the Laurent series, $\Pop_{H^2_{-}}[F\vert_{\T}]$ equals 
$\sum_{n\le 0}a_nz^n$ and $\langle F\rangle_{\T}=a_0$.
As a consequence, $\Pop_{H^2_{-}}[F\vert_{\T}]$ defines a holomorphic function
on the exterior disk $\D_{\e}(0,\hrho')$ and hence, $f$ is holomorphic on 
$\D_{\e}(0,\hrho')$ as well.
\end{rem}

The setting which will prove useful to us is when we may control certain 
related quantities and their polarizations, which is possible on thinner 
$\C^2$-complexified annuli. The polarization of $\log\hDelta R$ appears later
in the induction algorithm, while $\log(z\partial \hat R)$ is important for
the control associated with the implicit function theorem.

\begin{prop}
If $R$ belongs to a uniform family $\frakS\subset\mathcal{W}(\hrho_0,\sigma_0)$
for some positive reals $\hrho_0,\sigma_0$ with $\hrho_0<1$, and if $\hat R=
\sqrt{R}$ is chosen so that $\hat R(z)$ is positive for $|z|>1$ and negative
for $|z|<1$, then there exist positive $\hrho_1,\sigma_1$ with 
$\hrho_0\le\hrho_1<1$, $\sigma_1\le\sigma_0$ and 
$\hrho_1\ge (\sqrt{1+\sigma_1^2}+\sigma_1)^{-1}$, 
such that the polarizations of the functions $\log\hDelta R$, $\hat R$, 
$\log(z\partial \hat R)$ are all holomorphic in $(z,\bar w)$ and uniformly 
bounded on the $2\sigma_1$-fattened diagonal annulus 
$\hat{\mathbb{A}}(\hrho_1,\sigma_1)$.
\label{prop-rho1sigma1}
\end{prop}

\begin{proof}
This follows from the assumptions on the uniform family, if we use the standard
Cauchy estimates plus the fact that $\log\hDelta R=\log (2R_0^2)$ and 
$\log(z\partial \hat R)=\frac12\log R_0$ hold on the unit circle $\T$. 
The condition $ \hrho_1\ge (\sqrt{1+\sigma_1^2}+\sigma_1)^{-1}$ is achieved by 
choosing $\hrho_1$ large enough, but still in the range $\hrho_0\le \hrho_1<1$.
\end{proof}

\subsection{The master equation for the orthogonal foliation 
flow}\label{ss:flow1}
For an integer $n$, we denote by $\indsett_n$ 
the triangular index set
\begin{equation}
\indsett_{n}=\big\{(j,l)\in\N^2:\; 2j+l\le n\big\},
\label{eq:triangindset}
\end{equation}
and supply it with the inherited lexicographic ordering $\prec_{\mathrm{L}}$:
\[
(i,k)\prec_{\mathrm{L}}(j,l) \;\;\text{if}\;\; i<j \;\;\text{or}\;\; i=j\;\;
\text{and}\;\; k<l.
\]

We recall the notation of the pair $(\hrho_1,\sigma_1)$ from Proposition 
\ref{prop-rho1sigma1}. 

The following is an analogue of Lemma~\ref{lem:main-flow}. We introduce a
parameter $s$, which is supposed to be close to $0$, and plays the role of
the Planck constant $\hbar$. Later on, we will put $s=1/m$. 

\begin{prop}
\label{prop:flow-general}
Let $\kappa$ be a given positive integer and let 
$R\in\mathcal{W}(\hrho_0,\sigma_0)$,
for some $\hrho_0,\sigma_0$ with $0<\hrho_0<1$ and $\sigma_0>0$. Then there 
exist a radius $\hrho_2$ with $\hrho_1<\hrho_2<1$, bounded holomorphic 
functions $b_j$ on the exterior disk $\D_\e(0,\hrho_1)$
for $j=0,\ldots,\kappa$, and orthostatic conformal mappings
$$
\psi_{s,t}=\psi_{0,t}+\sum_{\substack{(j,l)\in\indsett_{2\kappa+1}\\ j\ge 1}}
s^jt^l\hat{\psi}_{j,l}
$$
defined on $\D_\e(0,\hrho_2)$ with
$\psi_{s,t}(\D_\e(0,\hrho_2))
\subset\D_\e(0,\hrho_1)$  for $s$ and $t$ close to $0$, such that the 
following holds. For fixed $s$, 
the domains $\psi_{s,t}(\D_\e)$ increase with $t$: 
$\psi_{s,t}(\D_\e)\subset \psi_{s,t'}(\D_\e)$ for $t<t'$, and if we put 
$h_s=\sum_{j=0}^{\kappa}s^jb_j$ and $f_s=\exp(h_s)$, the functions $f_s$ and 
$\psi_{s,t}$ have the property that for $\zeta\in\T$
\begin{multline}\label{eq:flow-alg-0}
\lvert f_s\circ\psi_{s,t}(\zeta)\rvert^2\e^{-2s^{-1}R\circ\psi_{s,t}(\zeta)}
\Re\big(-\bar\zeta\partial_t\psi_{s,t}(\zeta)
\overline{\psi_{s,t}'(\zeta)}\big)\\
=\e^{-t^2/s}\Big\{(4\pi)^{-\frac12}+\Ordo\big(\lvert s\rvert^{\kappa+\frac12}+
\lvert t\rvert^{2\kappa+1}\big)\Big\}.
\end{multline}
Here, the implicit constant remains uniformly bounded as long as $R$ is 
confined to a 
uniform family in $\mathcal{W}(\hrho_0,\sigma_0)$, for fixed $\hrho_0,\sigma_0$.
\end{prop}

\begin{rem}\label{rem:flow-remark} 
{\rm{(a)}}
Strictly speaking, the functions $\psi_{s,t}$ and $h_s$ we write down depend 
on the precision parameter $\kappa$, while the coefficient functions $b_j$ 
and $\hat{\psi}_{j,l}$ do not. We observe that the orthostaticity of 
$\psi_{s,t}$ gives that $\hat{\psi}_{0,0}'(\infty)>0$, and moreover that
$\Im \hat{\psi}_{j,l}'(\infty)=0$ for all $j,l\ge 0$.

\noindent {\rm (b)} The function $f_s$ will also admit an asymptotic expansion 
of the form
\[
f_s(\zeta)=\sum_{j=0}^{\kappa}s^j B_j(\zeta)+\Ordo(s^{\kappa+1}),\qquad 
\zeta\in\D_{\e}(0,\hrho_1)),
\]
where the coefficient functions $B_j$ may be obtained algorithmically 
as multivariate polynomials in the functions $b_0,\ldots,b_j$.
\end{rem}

The first step towards finding the conformal mappings $\psi_{s,t}$ is 
to note the 
following: we find by taking 
logarithms that
\begin{equation}\label{eq:flow-alg-1}
2\Re h_s\circ\psi_{s,t}(\zeta)-2s^{-1}(R\circ\psi_{s,t})(\zeta) + 
\log\Re\Big(-\bar\zeta\partial_t\psi_{s,t}
\overline{\psi_{s,t}'(\zeta)}\Big)=
-s^{-1}t^2+\Ordo(1),
\end{equation}
as $s,t\to0$. Next, we multiply both sides by $s$, to obtain
\begin{equation}\label{eq:flow-alg-1s}
2s\Re h_s\circ\psi_{s,t}(\zeta)-2\,R\circ\psi_{s,t}(\zeta) + 
s\log\Re\Big(-\bar\zeta\partial_t\psi_{s,t}
\overline{\psi_{s,t}'(\zeta)}\Big)=
-t^2+\Ordo(s).
\end{equation}
Finally, we take the limit as $s\to0$ in \eqref{eq:flow-alg-1s}, expecting 
that $\Re h_s\circ\psi_{s,t}$ 
and $\log\Re(-\bar\zeta\partial_t\psi_{s,t}\bar{\psi}_{s,t}')$ 
remain bounded, and arrive at the equation
$$
R\circ\psi_{0,t}(\zeta)=\frac{t^2}{2}.
$$
As a consequence, $\psi_{0,t}$ should be a 
conformal mapping of $\D_\e$ onto 
the exterior of the appropriate 
level curve of the weight $R$.

\begin{prop}
\label{prop:conformal}
Let $R$ be as in Proposition~\ref{prop:flow-general}.
There exists a positive number $t_0$, and a real-analytically 
smooth family $\{\psi_{0,t}\}_{t\in(-t_0,t_0)}$ 
of orthostatic conformal mappings 
$\D_\e\to\Omega_t$, where $\Omega_t$ 
is the unbounded component of 
$\C\setminus\Gamma_t$, and where $\Gamma_t$ 
are real-analytically smooth, 
simple closed level curves of $R$:
\[
R\vert_{\Gamma_t}=\frac{t^2}{2}.
\]
Moreover, $\Omega_0=\D_\e$ and $\Omega_t$ 
increases with $t$.
\end{prop}

\begin{proof}
The assumed strict subharmonicity of $R$ gives that there exists a 
neighborhood $U$ of $\T$ such that 
$\nabla R\big\vert_{U\setminus\T}\neq 0$. This shows that 
the level sets must be simple and closed curves, for $\lvert t\rvert$ 
sufficiently small.
Indeed, if a curve would possess a loop, then $R$ would have to have a local 
extremal point inside the loop, which is impossible. 
Since $\nabla R$ vanishes on $\T$, we cannot apply the implicit function 
theorem directly to $R$ to obtain the result. However, the function
\[
\tilde{R}(r\e^{\imag\theta}):=\frac{R(r\e^{\imag\theta})}{(r-1)^2}
\]
is, in view of Proposition~\ref{prop:R}, strictly positive and 
real-analytic in a neighborhood of the unit circle $\T$. We form the square 
root $\hat{R}=\sqrt{R}$ by
\[
\hat{R}(r\e^{\imag\theta})=(r-1)\sqrt{\tilde{R}(r\e^{\imag\theta})},
\]
where the square root on the right-hand side is the standard square root 
of a positive number. We may now apply the implicit function theorem to 
the function $\hat{R}$.
The result follows immediately by applying the Riemann mapping theorem 
to the exterior of the resulting analytic level curves of $\hat{R}$.
\end{proof}

\begin{rem}
Proposition~\ref{prop:conformal} tells us that the conformal mappings 
$\psi_{0,t}$
extend to some domain containing $\bar{\D}_\e$, but supplies little 
information on
how much bigger such a domain is allowed to be. We will discuss this issue in
Subsection~\ref{ss:psi-ext} below. Along the way, we also obtain an 
alternative proof
of Proposition~\ref{prop:conformal}, which may be viewed as a 
quantitative version
of the implicit function theorem in the given context.
\end{rem}

The Taylor coefficients $\hat{\psi}_{0,l}$ (in the flow variable $t$) of the 
conformal mappings $\psi_{0,t}$ may be explicitly computed in terms of
the weight $R$, using a higher order version of Nehari's formula 
for conformal mappings to nearly circular domains. We will 
return to this in Subsection~\ref{ss:flow2}. Before we carry on, we formulate
the following lemma, which allows us to draw the conclusion 
that the mappings $\psi_{s,t}$ of Proposition~\ref{prop:flow-general} 
are actually conformal.

\begin{lem}\label{lem:becker}
Assume that $\psi$ is a holomorphic 
function on $\D_\e(0,\hrho)$ of the form
\[
\psi(z)=z + F(z),
\]
such that $|F'|\le \frac12$ and 
\[
2|zF''(z)|\le\frac{\hrho^2}{|z|^2-\hrho^2},\qquad z\in\D_{\e}(0,\hrho).
\]
Then $\psi$ is univalent on $\D_\e(0,\hrho)$.
\end{lem}

\begin{proof}
This is immediate from the Becker-Pommerenke 
univalence criterion \cite{becker}.
\end{proof}

It is clear that the mappings $\psi_{s,t}$
meet this criterion for some $\hrho<1$, for small enough $s$ and $t$
for a fixed accuracy parameter $\kappa$. 

\subsection{The smoothness of level curves and the implicit 
function theorem}
\label{ss:psi-ext} 
In this subsection, we analyze the extension properties of conformal
mappings from $\D_\e$ onto the exterior of the level curves of $R$ near 
the unit circle. In a sense, this may be viewed as a quantitative version
of the implicit function theorem.

The function $R$ is assumed to belong to the class
$\mathcal{W}(\hrho_0,\sigma_0)$ of Definition~\ref{def:classR}, which is a 
quantitative
way to say that $R$ is real-analytic near the unit circle $\T$, and 
vanishes along with 
its normal derivative on $\T$, while $\hDelta R$ is positive on $\T$. 
We recall the definition of the choice of square root $\hat{R}$
of $R$ from the proof of Proposition~\ref{prop:conformal}. 
This function is also real-analytic near the circle, 
vanishes on $\T$ but its gradient is nonzero and points in the 
direction of the outward normal.
To make this more quantitative, we let $\hrho_1$ 
and $\sigma_1$ be the parameters of 
Proposition~\ref{prop-rho1sigma1}. 
Then, in the $2\sigma_1$-fattened diagonal annulus 
$\hat{\mathbb{A}}(\hrho_1,\sigma_1)$, we have the control
\begin{equation}
\sup_{(z,w)\in \hat{\mathbb{A}}(\hrho_1,\sigma_1)}
\lvert\log(z\partial_z\hat{R}(z,w))\rvert<+\infty.
\label{eq:Rcontrol01}
\end{equation}
We recall that the mappings $\psi_{0,t}$ are defined 
by the requirement of orthostaticity and
\begin{equation}
\hat{R}\circ\psi_{0,t}(\zeta)=-\frac{t}{\sqrt{2}},\qquad \zeta\in\T.
\label{eq:psi0t-1.01}
\end{equation}
By differentiating the relation \eqref{eq:psi0t-1.01} with respect to $t$, 
we obtain from the chain rule
$$
[(\partial_r\hat{R})\circ\psi_{0,t}]\,\partial_t\lvert \psi_{0,t}\rvert 
+ [(\partial_\theta\hat{R})\circ\psi_{0,t}]\,\partial_t\arg\psi_{0,t}=
-\frac{1}{\sqrt{2}},
$$
which we may rewrite as
\begin{multline*}
[(r\partial_r\hat{R})\circ\psi_{0,t}]\,\partial\log\lvert\psi_{0,t}\rvert 
+[(\partial_\theta\hat{R})\circ\psi_{0,t}]\;\partial_t\arg\psi_{0,t}\\
=\Re\Big\{\big[\big(r\partial_r\hat{R}
-\imag\partial_\theta\hat{R}\big)\circ\psi_{0,t}\big]\,
\partial_t\log\frac{\psi_{0,t}}{\zeta}\Big\}
=-\frac{1}{\sqrt{2}}.
\end{multline*}
Here, we divided by the coordinate function $\zeta$ in order to avoid 
issues with branch cuts of the logarithm. The differential operator acting on 
$\hat{R}$ may be written as $r\partial_r-\imag\partial_\theta=2z\partial_z$, 
so the above expression simplifies further to
\begin{equation}
\Re\Big\{\big[\big(2z\partial_z\hat{R}\big)\circ\psi_{0,t}]\,
\partial_t\log\frac{\psi_{0,t}}{\zeta}\Big\}=-\frac{1}{\sqrt{2}}\quad\text{on}
\,\,\,\T.
\label{eq:psi0t-1.02}
\end{equation}
It is on the basis of the relation \eqref{eq:psi0t-1.02} 
that we will try to recover 
information on the mappings $\psi_{0,t}$.
We introduce the notation 
\begin{equation}\label{eq:mu-001}
\mu(\zeta):=\log\big(2z\partial_z\hat{R}\big),\quad 
\mu_t=\mu\circ\psi_{0,t},\quad
F_t(\zeta)=\partial_t\log\frac{\psi_{0,t}(\zeta)}{\zeta},
\end{equation}
and observe that \eqref{eq:psi0t-1.02} may be written in the form
\begin{equation}
  \label{eq:mu-1212}
\e^{\mu_t}F_t+\e^{\bar\mu_t}\bar{F_t}=-\sqrt{2}\quad\text{on}
\,\,\,\T.
\end{equation}
We note that along the unit circle $\T$, the function 
$\e^{\mu}=2z\partial_z\hat{R}$ equals the positive function 
$\sqrt{2\hDelta R}$, so there are no problems with taking the logarithm 
in the definition of $\mu$ in a neighborhood of $\T$.
In particular, if $\psi_{0,t}$ is a perturbation 
of the identity, the function $\mu_t$ 
is well-defined and smooth.
Next, we decompose $\mu_t=\mu_t^{+}+\mu_t^{-}$, where $\mu_t^+\in H^2$ and
$\mu_t^{-}\in H^2_{-,0}$ are both smooth, and write $G_t=\e^{\mu_t^{-}}F_t$.
Given that $F_t\in H^2_-$, it is clear that $G_t\in H^2_{-}$. 
If we multiply the above equation \eqref{eq:mu-1212} by $\e^{-2\Re \mu_t^+}$, 
we arrive at
\[
\e^{-\bar{\mu_t}^+}G_t + \e^{-\mu_t^+}\bar{G}_t=
2\Re\big\{\e^{-\bar{\mu}_t^+}G_t\big\}=-\sqrt{2}\,\e^{-2\Re\mu_t^+},
\]
where we point out that $\e^{-\bar{\mu}_t^+}G_t\in H^2_-$ while
$\e^{-\mu_t^+}\bar{G}_t\in H^2$.
This is equation is solved by applying the 
Herglotz kernel, and yields the solution
\[
G_t=-\frac{1}{\sqrt{2}}\,\e^{\bar\mu_t^+}\Hop_{\D_\e}
\big[\e^{-2\Re\mu_t^+}\big],
\]
where we use the fact that $F_t$ and $\mu_t^+$ 
are real-valued at infinity 
(cf. Remark~\ref{rem:flow-remark} {\rm (a)}). 
That is,
\begin{equation}\label{eq:ft-sol}
F_t=-\frac{1}{\sqrt{2}}\,\e^{\bar\mu_t^+-\mu_t^-}
\Hop_{\D_\e}[\e^{-2\Re\mu_t^+}\big].
\end{equation}
Let us write
\begin{equation}\label{eq:g-mu}
g_t(\zeta)=\log\frac{\psi_{0,t}(\zeta)}{\zeta},
\end{equation}
so that $\partial_t g_t=G_t$ and $g_0=0$. 
Here, the logarithm is understood as the principal 
branch. In terms of these functions, the equation 
\eqref{eq:ft-sol} becomes the following nonlinear 
differential equation in $t$:
\begin{equation}\label{eq:psi-nl}
\partial_t g_t=-\frac{1}{\sqrt{2}}\exp
\Big\{\overline{\Pop_{H^2}[\mu\circ\psi_{0,t}]}-
\Pop_{H^2_{-,0}}
[\mu\circ\psi_{0,t}]\Big\}
\Hop_{\D_\e}\big[\exp\big\{-2\Re \Pop_{H^2}[\mu\circ\psi_{0,t}]\big\}\big].
\end{equation}
It is not difficult to see that the equation 
\eqref{eq:psi-nl} may be solved 
by an 
iterative procedure, if we rewrite 
it in integral form
\begin{equation}\label{eq:psi-nl2}
g_t=-\frac{1}{\sqrt{2}}\int_{0}^t\exp\Big\{\overline{\Pop_{H^2}
[\mu\circ\psi_{0,\theta}]}-
\Pop_{H^2_{-,0}}[\mu\circ\psi_{0,\theta}]\Big\}
\Hop_{\D_\e}\big[\exp\big\{-2\Re \Pop_{H^2}[\mu\circ\psi_{0,\theta}]\big\}\big]
\diff\theta.
\end{equation}
As a first order approximation, we start with 
$\psi_{0,t}^{[0]}(\zeta)=\zeta$, 
and use the formula \eqref{eq:psi-nl2} to define 
$g_{t}^{[j+1]}$ in terms of $\psi_{0,t}^{[j]}$, 
for $j=0,1,2,\ldots$ by 
integration. 
The process is interlaced with computing $\psi_{0,t}^{[j+1]}:=
\zeta\exp(g_t^{[j+1]})$, and results in 
convergent sequences $g_t^{[j]}$ 
and $\psi_{0,t}^{[j]}$.

We are interested in analyzing where the function $\psi_{0,t}$ 
extends as a holomorphic mapping. To this end, we recall that the 
function $\mu$ given by \eqref{eq:mu-001} has a well-defined polarization to 
$\hat{\mathbb{A}}(\hrho_1,\sigma_1)$.
It is clear that if for some $\hat{\hrho}_t<1$, $\psi_{0,t}$ maps the annulus 
$\mathbb{A}(\hat{\hrho}_t,(\hat{\hrho}_t)^{-1})$ into 
$\mathbb{A}(\hrho_1, \hrho_1^{-1})$, we obtain the estimate
\[
\lVert\partial_t g_t
\rVert_{H^\infty(\mathbb{A}(\hat{\hrho}_t,(\hat{\hrho}_t)^{-1}))}\le 
\frac{\sqrt{2}}{1-\hat{\hrho}_t^2}\exp\Big\{5\frac{\lVert 
\mu\rVert_{H^\infty(\mathbb{A}(\hrho_1, \hrho_1^{-1}))}}{1-\hat{\hrho}_t^2}\Big\},
\]
where we use the estimate
\[
\lVert \Pop_{H^2}[f]\rVert_{H^\infty(\D(0,(\hat{\hrho}_t)^{-1}))}\le 
\frac{\lVert f\rVert_{H^\infty(\mathbb{A}(\hat{\hrho}_t,(\hat{\hrho}_t)^{-1}))}}
{1-\hat{\hrho}_t^2},
\]
and the analogous estimate for $\Pop_{H^2_{-,0}}[f]$.
Assume for the moment that $\hat{\hrho}_t<1$ is monotonically increasing in 
$\lvert t\rvert$, and recall that $\psi_{0,t}(\zeta)=\zeta\exp(g_t)$. 
In light of the above estimate of $\partial_t g_t$, we obtain
$$
\lVert g_t\rVert_{H^\infty(\mathbb{A}(\hat{\hrho}_t,(\hat{\hrho}_t)^{-1}))}\le 
\frac{\sqrt{2}\lvert t\rvert}{1-\hrho_t^2}\exp\Big\{5\frac{\lVert 
\mu\rVert_{H^\infty(\mathbb{A}(\hrho_1,\hrho_1^{-1}))}}
{1-\hat{\hrho}_t^2}\Big\}=:C_t\lvert t\rvert,
$$
where $C_t$ is defined implicitly by the last relation.
This leads to the control  
$$
\e^{-C_t\lvert t\rvert}\hat{\hrho}_t\le \lvert \psi_{0,t}(\zeta)\rvert \le 
\e^{C_t\lvert t\rvert}(\hat{\hrho}_t)^{-1},\qquad 
\zeta\in\mathbb{A}(\hat{\hrho}_t,(\hat{\hrho}_t)^{-1}),
$$
which means that $\psi_{0,t}$ maps the annulus 
$\mathbb{A}(\hat{\hrho}_t,(\hat{\hrho}_t)^{-1})$ into
$\mathbb{A}(\hrho_1,\hrho_1^{-1})$, provided that
$$
\e^{-C_t\lvert t\rvert}\hat{\hrho}_t\ge \hrho_1.
$$
Let us make the ansatz $\hat{\hrho}_t=\hrho_1\e^{M\lvert t\rvert}$, for 
some positive constant $M$.
The above requirement is then satisfied provided that $M\ge C_t$.
If we restrict $t$ to the interval 
\begin{equation}
\label{eq:t-bound}
\lvert t\rvert\le \frac{\log\frac{1}{\hrho_1}}{2M},
\end{equation}
it is immediate that
$$
\frac{1}{1-\hat{\hrho}_t^2}\le 
\frac{1}{1-\hrho_1}.
$$
This then gives the estimate for $C_t$
$$
C_t\le \frac{\sqrt{2}}{1-\hrho_1}
\exp\Big\{5\frac{\lVert 
\mu\rVert_{H^\infty(\mathbb{A}(\hrho_1,\hrho_1^{-1})}}{1-\hrho_1}\Big\},
$$
where the right-hand side does not depend on $t$. We may finally choose 
$M$ to be this constant,
\begin{equation}
M=\frac{\sqrt{2}}{1-\hrho_1}
\exp\Big\{5\frac{\lVert 
\mu\rVert_{H^\infty(\mathbb{A}(\hrho_1,\hrho_1^{-1})}}{1-\hrho_1}\Big\}
\label{eq:Aformula}
\end{equation}
and obtain that $\psi_{0,t}$ is holomorphic in the exterior disk 
$\D_\e(0,\hat{\hrho}_t)$, where
\[
\hat{\hrho}_t=\hrho_1\e^{M\lvert t\rvert}, 
\]
provided that $t$ satisfies \eqref{eq:t-bound}. For $t$ close to $0$, 
$\hat{\hrho}_t$ is then close to $\hrho_1$ in a quantifiable fashion. 
We gather these observations in a proposition.

\begin{prop}\label{prop:psi0t-ext}
Suppose $R$ is in the class $\mathcal{W}(\rho_0,\sigma_0)$ and that  
\eqref{eq:Rcontrol01} holds. Then the conformal maps $\psi_{0,t}$, initially
defined on $\bar\D_\mathrm{e}$, extend to holomorphic functions on the
exterior disk $\D_\mathrm{e}(0,\hat{\rho}_t)$, where 
$\hat{\rho}_t=\rho_1\e^{M|t|}\le\sqrt{\rho_1}$
and $M$ is given by \eqref{eq:Aformula}, provided that $t$ is in the interval
\eqref{eq:t-bound}.
\end{prop}

\subsection{An outline of the orthogonal foliation flow algorithm}
\label{ss:alg-flow-over}
We now proceed to describe an outline of the algorithm. With the notation
\begin{equation}
\label{eq:omega-def}
\logdens_{s,t}(\zeta)=2\Re h_s\circ\psi_{s,t}(\zeta)
-\tfrac{2}{s}\big((R\circ\psi_{s,t})(\zeta)-\tfrac{1}{2}t^2\big)
+\log\Big(\Re\big\{-\bar\zeta\partial_t\psi_{s,t}(\zeta)
\overline{\psi_{s,t}'(\zeta)}\big\}\Big),
\end{equation}
we may rewrite the master equation \eqref{eq:flow-alg-0} for the
orthogonal foliation flow as
\begin{equation}
\label{eq:omega-id}
\hat{\logdens}_{j,l}(\zeta)
=\frac{\partial^j_s\partial^l_k\logdens_{s,t}(\zeta)}{j!l!}\big\vert_{s=t=0}=
\begin{cases}
-\frac12\log(4\pi)&\quad\text{for}\quad \zeta\in\T\;\;
\text{and}\;\; (j,l)=(0,0),
\\
0&\quad\text{for}\quad \zeta\in\T\;\;\text{and}\;\; 
(j,l)\in\indsett_{2\kappa}\setminus
\{(0,0)\}.
\end{cases}
\end{equation}
provided that the functions  $h_s$ and $\psi_{s,t}$ obtained by solving these 
equations do not degenerate, as long as $R$ remains in a bounded set of 
$\mathcal{W}(\hrho_0,\sigma_0)$ for some $\hrho_0$ with $0<\hrho_0<1$ and 
some $\sigma_0>0$.
We recall that $h_s$ is defined by the finite expansion
\begin{equation}\label{eq:def-h}
h_s=\sum_{j=0}^{\kappa}s^{j}b_j.
\end{equation}
As it turns out later on in Proposition~\ref{prop:exp-3}, we have
for $j,l\ge 1$,
\begin{equation}\label{eq:hatphi-det}
\hat{\logdens}_{j-1,l}(\zeta)
=-2(4\hDelta R(\zeta))^{\frac12}\Re(\bar\zeta\hat{\psi}_{j,l-1}(\zeta))
+\frakT_{j-1,l}(\zeta),
\qquad \zeta\in\T,
\end{equation}
where $\frakT_{j-1,l}$ is real-valued and real-analytic, and 
depends only on $b_0,\ldots,b_{j-1}$ and $\hat{\psi}_{p,q}$ where 
$(p,q)\prec_{\mathrm{L}} (j,l-1)$, where we recall that $\prec_{\mathrm{L}}$ 
denotes the standard lexicographic ordering. Moreover, when $l=0$ we get
\begin{equation}
\label{eq:Bj-det}
\hat{\logdens}_{j,0}(\zeta)=
2\Re b_j(\zeta)+\frakT_{j,0}(\zeta),\qquad \zeta\in\T,
\end{equation}
where $\frakT_{j,0}$ depends only on $b_0,\ldots, b_{j-1}$ and 
$\hat{\psi}_{p,q}$ for $(p,q)\prec_{\mathrm{L}} (j+1,0)$. Such dependencies
will be encoded in terms of complexity classes introduced in 
Subsection~\ref{ss:pol-complexity}.
\\
\\
\noindent{\sc Step 1.} We let $\psi_{0,t}$ be the orthostatic conformal 
mappings to 
the exterior of level curves of $R$, as given by 
Proposition~\ref{prop:conformal}. In 
particular, this determines uniquely the coefficient functions 
$\hat{\psi}_{0,l}$, for $l=0,\ldots, 2\kappa+1$ 
(for the details, see Proposition~\ref{prop:conf-expl} below).
For instance, we find that $\hat{\psi}_{0,0}(\zeta)=\zeta$, while
$\hat{\psi}_{0,1}(\zeta)=-\zeta\Hop_{\D_\e}[(4\Delta R)^{-\frac12}]$.
\\
\\
\noindent{\sc Step 2.} By evaluating 
$\hat{\logdens}_{0,0}(\zeta)=\logdens_{s,t}(\zeta)\big\vert_{s=t=0}$, 
we obtain from 
\eqref{eq:omega-id} that
\[
2\Re b_0(\zeta)+\log\Re(-\bar{\zeta}\hat{\psi}_{0,1}(\zeta)) 
=-\frac12\log(4\pi),\qquad \zeta\in\T.
\]
As $\hat{\psi}_{0,1}$ is already known and the above real part is 
strictly positive on
$\T$ (see Proposition~\ref{prop:conf-expl} below), this gives the value of
$2\Re b_0$ on the unit circle $\T$, which gives that
\[
b_0=-\frac14\log(4\pi)+\frac14\Hop_{\D_\e}\big[\log (4\hDelta R)\big].
\]
\\
\noindent {\em We proceed from Step 2 to Step 3 with $j=1$.}\\
\\
\noindent{\sc Step 3.} We have determined $b_0,\ldots,b_{j_0-1}$ and
$\hat{\psi}_{j,l}$ for all $(j,l)\prec_{\mathrm{L}}(j_0,0)$, and in this step we 
intend to determine all the coefficient functions $\hat{\psi}_{j,l}$ 
for $(j,l)\prec_{\mathrm{L}}(j_0+1,0)$.
In view of Proposition~\ref{prop:exp-3} below, we may obtain explicitly 
$\frakT_{j_0-1,1}$ in terms of this known data set, 
which by the equations~\eqref{eq:omega-id} and 
\eqref{eq:hatphi-det} gives 
an equation for $\hat{\psi}_{j_0,0}$. More generally, the equation which 
gives $\hat{\psi}_{j_0,l_0}$ takes the form 
\[
\Re (\bar\zeta\hat{\psi}_{j_0,l_0})=\tfrac12(4\hDelta R)^{-\frac12}\frakT_{j_0-1,l_0+1}
\quad \text{on }\;\T,
\] 
and we solve it with the formula 
\[
\hat{\psi}_{j_0,l_0}(\zeta)=\tfrac12\zeta\,\Hop_{\D_\e}
\big[(4\hDelta R)^{-\frac12}\frakT_{j_0-1,l_0+1}\big](\zeta).
\] 
If we apply this solution formula with $l=0$, the background data gets 
extended to all $\hat{\psi}_{j,l}$ with $(j,l)\prec_{\mathrm{L}}(j_0,1)$. 
Continuing in the same fashion, Proposition~\ref{prop:exp-3} shows that 
$\frakT_{j_0-1,2}$ may be expressed in terms of this extended data set. 
Consequently, the above solution formula also determines $\hat{\psi}_{j_0,1}$. 
More generally,  as we proceed iteratively in the same manner, we obtain 
all the coefficient functions $\hat{\psi}_{j,l}$ with $j=j_0$ and 
$(j,l)\prec_{\mathrm{L}}(j_0+1,0)$.\\
\\
\noindent{\sc Step 4.} At this stage, using Step 3, we have at our 
disposal the 
functions $b_0,\ldots,b_{j_0-1}$, and $\hat{\psi}_{j,l}$ for all 
$(j,l)\prec_{\mathrm{L}} (j_0+1,0)$.
Proposition~\ref{prop:exp-3} now allows us to compute 
$\frakT_{j_0,0}$ in terms of this data, 
and from~\eqref{eq:omega-id} and \eqref{eq:Bj-det},
we derive an equation for $b_{j_0}$:
\[
2\Re b_{j_0}=-\frakT_{j_0,0},\qquad \text{on }\T.
\]
We solve this equation explicitly by
$$
b_{j_0}(\zeta)=-\frac12\Hop_{\D_\e}
\big[\frakT_{j_0,0}\big](\zeta),\qquad 
\zeta\in\D_\e.
$$
After completing this step in the algorithm, we have extended the
data set to contain $b_0,\ldots,b_{j_0}$ and all coefficient functions 
$\hat{\psi}_{j,l}$ with $(j,l)\prec_{\mathrm{L}}(j_0+1,0)$.

\medskip

\noindent{\sc Step 5.} Finally, we iterate Steps 3 and 4 with 
$j_0$ replaced by $j_0+1$, until all coefficient functions $b_k$ and 
$\hat{\psi}_{j,l}$ have been determined, for $0\le k\le \kappa$ and
$(j,l)\in\indsett_{2\kappa+1}$. This also means that the flow equation
\eqref{eq:omega-id} is met with the given choices of coefficient functions.

\begin{rem}
If we apply the above algorithm to the function $R=R_\tau$, 
the coefficient functions 
$B_j$ in the expansion of $f_s=\exp(h_s)$ obtained here are 
the same as those appearing in 
Theorem~\ref{thm:main-coeff}. There, the algorithm was based on Laplace's 
method and inhomogeneous Toeplitz kernel equations.
The algorithm presented here is in principle an 
alternative route towards 
finding the coefficient functions. 
However, a drawback is that the algorithm requires us to compute the 
additional functions $\hat{\psi}_{j,l}$, which adds further complexity.
\end{rem}

\subsection{The general multivariate Fa{\`a} di Bruno formula}
\label{ss:bruno}
We recall Fa{\`a} di Bruno's formula in several variables, and study 
some of its properties. To prepare for the formulation, we introduce 
the well-ordering used in \cite{bruno}, which we call the order-lexicographical
ordering (OL for short). Given two multi-indices
\[
\bm{\alpha}=(\alpha_1,\ldots,\alpha_d)\quad\text{and}\quad 
\bm{\beta}=(\beta_1,\ldots,\beta_d),
\] 
we write that $\bm{\alpha} \prec_{\mathrm{OL}}\bm{\beta}$ if:

\noindent{\rm(i)} $|\bm{\alpha}|<|\bm{\beta}|$, or if

\noindent{\rm(ii)} $|\bm{\alpha}|=|\bm{\beta}|$ and 
$\bm{\alpha}\prec_{\mathrm{L}}\bm{\beta}$ (lexicographically).

\noindent Here, we recall that in the lexicographical ordering 
$\bm{\alpha}\prec_{\mathrm{L}}\bm{\beta}$ holds if either $\alpha_1<\beta_1$ or 
$\alpha_1=\beta_1,\ldots,\alpha_{k-1}=\beta_{k-1}$ while 
$\alpha_{k+1}<\beta_{k+1}$ holds for some $1\le k\le d$. 
As a matter of notation, $\bm{\alpha}\preceq\bm{\beta}$ means that 
either $\bm{\alpha}\prec\bm{\beta}$ or $\bm{\alpha}=\bm{\beta}$; this applies 
to both the lexicographical and order-lexicographical orderings.  
We use some elements of standard multi-index notation. For instance, if 
$\bm{\alpha}=(\alpha_1,\ldots,\alpha_d)$ is a $d$-dimensional multi-index, that
is, a $d$-vector of integers in $\N:=\{0,1,2,\ldots\}$, we write
\begin{align*}
\lvert \bm{\alpha}\rvert= & \sum_i \alpha_i,
\\
\bm{\alpha}!=& \prod_i (\alpha_i!),\\
\bm{\xi}^{\bm{\alpha}}=& \prod_i \xi_i^{\alpha_i},\qquad 
\bm{\xi}=(\xi_1,\ldots,\xi_d)\in\C^d, 
\\
\partial^{\bm{\alpha}} f(\bm{x})=&\partial^{\alpha_1}_{x_1}\cdots 
\partial^{\alpha_d}_{x_d} f(\bm{x}),\qquad \bm{x}=(x_1,\ldots, x_d)\in\R^d.
\end{align*}
We will need the index set
\begin{multline}
\indsetT^{\mathrm{OL}}_{k;d',d}
=\big\{(\bm{\alpha}_1,\ldots, \bm{\alpha}_k; \bm{\beta}_1,\ldots, 
\bm{\beta}_k)\in (\N^{d'})^k\times (\N^d)^k:
\\
0\prec_{\mathrm{OL}}\bm{\alpha}_1\prec_{\mathrm{OL}}
\cdots\prec_{\mathrm{OL}}\bm{\alpha}_k\;\;\text{and}\;\;
\forall i=1,\ldots,k:\; \lvert \bm{\beta}_i\rvert>0 \big\}.
\label{eq:indsetTOL}
\end{multline}
We now formulate the multivariate Fa{\`a} di Bruno's formula as it appears in
\cite{bruno}.

\begin{prop}
\label{prop:bruno}
Let $\Omega\subset\R^{d}$ and $\Omega'\subset\R^{d'}$ be domains in the 
respective Euclidean space. Let $\mathbf{g}=(g_1,\ldots,g_d):
\Omega'\to\Omega$ and $f:\Omega\to\R$ be $C^n$-smooth, so that the composition 
$f\circ\mathbf{g}:\Omega'\to\R$ is $C^n$-smooth as well. 
Then, for any $d'$-dimensional multi-index $\bm{\nu}$ with 
$\lvert\bm{\nu}\rvert=n$, we have on $\Omega'$
$$
\partial^{\bm{\nu}}(f\circ \mathbf{g})=
\sum_{1\le \lvert \bm{\mu}\rvert \le n}[(\partial^{\bm{\mu}} f)
\circ \mathbf{g}]\;\mathcal{G}_{\bm{\mu},{\bm \nu}}(\mathbf{g}),
$$
where $\bm{\mu}$ runs over the $d'$-dimensional multi-indices, and the function 
$\mathcal{G}_{\bm{\mu},{\bm \nu}}(\mathbf{g})$ is given by
\[
\mathcal{G}_{\bm{\mu},{\bm \nu}}(\mathbf{g})=
\bm{\nu}!\sum_{k=1}^{n}\sum_{(\bm{\alpha};
\bm{\beta})
\in\indsetS_k^{\mathrm{OL}}(\bm{\mu},\bm{\nu})} 
\prod_{i=1}^{k} \frac{[\partial ^{\bm{\alpha}_i}
\mathbf{g}]^{\bm{\beta}_i}}{\bm{\beta}_i![\bm{\alpha}_i!]^{\lvert\bm{\beta}_i\rvert}}.
\]
Here, the indicated index set is given by
\[
\indsetS^{\mathrm{OL}}_k(\bm{\mu},\bm{\nu}):
=\Big\{(\bm{\alpha};\bm{\beta})=(\bm{\alpha}_1,\ldots,
\bm{\alpha}_k; \bm{\beta}_1,\ldots,\bm{\beta}_k)\in 
\indsetT^{\mathrm{OL}}_{k;d',d}\,:\,
\sum_i \bm{\beta}_i=\bm{\mu},\;
\sum_i \lvert \bm{\beta}_i\rvert \bm{\alpha}_i=\bm{\nu}\Big\}.
\]
\end{prop}

Note that since $\mathbf{g}$ is assumed vector-valued, the multi-index 
partial derivative $\partial^{\bm{\alpha}_j}\mathbf{g}$ is vector-valued as well,
and the multi-index power $[\partial^{\bm{\alpha}_j}\mathbf{g}]^{\bm{\beta}_j}$ 
produces a real-valued function. 

\begin{rem}
\label{rem:reshuff}
Both the order-lexicographical and the lexicographical ordering are 
well-orderings of the multi-indices. If in \eqref{eq:indsetTOL} we replace
$\prec_{\mathrm{OL}}$ by the lexicographic ordering $\prec_{\mathrm{L}}$ to obtain 
the analogous index set $\indsetT^{\mathrm{L}}_{k;d',d}$, this amounts to a 
reshuffling of the multi-indices $\bm{\alpha}_1,\ldots, \bm{\alpha}_k$ to get
them ordered with respect to $\prec_{\mathrm{L}}$ instead. This allows us to 
define the index set $\indsetS^{\mathrm{L}}_k(\bm{\mu},\bm{\nu})$ as well, 
based on $\indsetT^{\mathrm{L}}_{k;d',d}$ instead. It is important to note that the
assertion of Proposition \ref{prop:bruno} holds with the index set
$\indsetS^{\mathrm{OL}}_{k}(\bm{\mu},\bm{\nu})$ replaced by 
$\indsetS^{\mathrm{L}}_{k}(\bm{\mu},\bm{\nu})$. The reason why this is so is that
if we reshuffle both the $\bm{\alpha}_j$ and the $\bm{\beta}_j$, then nothing
really happened and the involved sum remains the same. 
\end{rem}

\subsection{The multivariate Fa\`a di Bruno formula
adapted to our setting}
\label{ss:bruno-oursetting}

We specialize Proposition~\ref{prop:bruno} to the situation that we need to 
analyze. We will consider only the case of $d=d'=2$. We work in terms of
polar coordinates $(r,\theta)$, and put
$\Rfun(r,\theta)= R(r\e^{\imag\theta})$. Although still not specified completely,
we assume the function $\psi_{s,t}$ is sufficiently smooth in both $(s,t)$, 
and introduce the function $\bm{\Psi}_{s,t}$, 
\begin{equation}
\label{eq:bPsist}
\bm{\Psi}_{s,t}=(\lvert \psi_{s,t}\rvert,
\arg\psi_{s,t}),
\end{equation}
which maps to polar coordinates, so that 
\begin{equation}
\label{eq:Rfun-1}
R\circ\psi_{s,t}=\Rfun\circ\bm{\Psi}_{s,t}.
\end{equation}
Accordingly, we denote by $D_{r,\theta}^{\bm{\mu}}$ the differential operator
\[
D_{r,\theta}^{\bm{\mu}}=\partial_r^{\mu_1}\partial_\theta^{\mu_2}, \quad \bm{\mu}
=(\mu_1,\mu_2),
\]
and obtain by applying Proposition~\ref{prop:bruno} to 
$\Rfun\circ\bm{\Psi}_{s,t}$ with $\bm{\nu}=(j,l)$ that along the circle $\T$,
\begin{multline}
\label{eq:bruno-R}
\partial_s^j\partial_t^l (R\circ\psi_{s,t})\big\vert_{s=t=0}
=\partial_s^j\partial_t^l (\Rfun\circ\bm{\Psi_{s,t}})\big\vert_{s=t=0}
\\
=\sum_{2\le\lvert\bm{\mu}\rvert\le j+l}[(D_{r,\theta}^{\bm{\mu}} 
\Rfun)\circ\bm{\Psi_{s,t}}]\,
\mathcal{G}_{\bm{\mu},(j,l)}(\bm{\Psi}_{s,t})\Big\vert_{s=t=0}
\\
=\sum_{2\le\lvert\bm{\mu}\rvert\le j+l}[(D_{r,\theta}^{\bm{\mu}} 
R)(\psi_{s,t})]\,
\mathcal{G}_{\bm{\mu},(j,l)}(\bm{\Psi}_{s,t})\Big\vert_{s=t=0}
\\
=\sum_{2\le\lvert\bm{\mu}\rvert\le j+l}(D_{r,\theta}^{\bm{\mu}} 
R)\,
\mathcal{G}_{\bm{\mu},(j,l)}(\bm{\Psi}_{s,t})\Big\vert_{s=t=0},
\end{multline}
where the terms corresponding to indices $\bm{\mu}$ with 
$\lvert \bm{\mu}\rvert\le 1$ vanish and hence get dropped. The reason for this
is that $\psi_{0,0}(\zeta)=\zeta$ preserves $\T$ and that the function $R$ 
together with its gradient vanish along the unit circle $\T$.
More generally, we find that
\begin{equation}
\label{eq:flatR-1.01}
D_{r,\theta}^{\bm{\mu}}R\vert_{\T}=\partial_r^{\mu_1}\partial_\theta^{\mu_2}R|_\T=0, 
\qquad 
\bm{\mu}=(\mu_1,\mu_2)\in\{0,1\}\times\N.
\end{equation}
In the context of \eqref{eq:bruno-R}, it is important to point out that the 
multi-index derivatives that appear in the 
expression $\mathcal{G}_{\bm{\mu},(j,l)}(\bm{\Psi}_{s,t})$ (as defined in 
Proposition~\ref{prop:bruno}) are taken with respect to the variables $(s,t)$. 
Moreover, in the equality \eqref{eq:bruno-R} we have suppressed the 
variable $\zeta\in\T$, and consider it to be fixed.

We will be interested in identifying the {\em maximal} index $(p,q)$ 
with respect to the lexicographical ordering, such that the partial derivative 
$\partial^p_s\partial^q_t\bm{\Psi}_{s,t}$ appears nontrivially in 
the right-hand side expression of \eqref{eq:bruno-R}. 

\begin{prop}\label{prop:index}
Let $\bm{\nu},\bm{\mu}\in\N^2$ be double-indices with 
$2\le\lvert\bm{\mu}\rvert\le \lvert \bm{\nu}\rvert$ and 
$\bm{\mu}\notin \{(1,1), (0,2)\}$, and let 
$(\bm{\alpha};\bm{\beta})\in \indsetS_k^{\mathrm{L}}(\bm{\mu},\bm{\nu})$. Then
\begin{enumerate}[\rm(i)]

\item If $\bm{\nu}=(j,l)$, where $j,l\ge 1$, then for all 
$i=1,\ldots,k$, we have that
$\bm{\alpha}_i\preceq_{\mathrm{L}} (j,l-1)$. Moreover, the equality 
$\bm{\alpha}_i=(j,l-1)$ holds if and only if $i=k=2$, $\bm{\mu}=(2,0)$, and 
\[
(\bm{\alpha};\bm{\beta})=((0,1), (j,l-1); (1,0),(1,0)).
\]

\item If $\bm{\nu}=(j,0)$ with $j\ge 3$, then each $\bm{\alpha}_i$ is 
of the form $(a,0)$ with  $a\le j-1$. Moreover, the equality $a=j-1$ 
holds if and only if $i=k=2$, $\bm{\mu}=(2,0)$, and 
\[
(\bm{\alpha};\bm{\beta})=((1,0),(j-1,0); (1,0),(1,0)).
\]

\item If $\bm{\nu}=(0,l)$ with $l\ge 3$, then $\bm{\alpha}_i$ is of 
the form $(0,b)$ with  $b\le l-1$. Moreover, the equality $b=l-1$ holds if 
and only if $\bm{\mu}=(2,0)$ and 
\[
(\bm{\alpha};\bm{\beta})=((0,1),(0,l-1); (1,0),(1,0)).
\]

\item If $\bm{\nu}=(2,0)$, then necessarily $\bm{\mu}=(2,0)$ and the 
only nontrivial index 
$(\bm{\alpha};\bm{\beta})$ is 
\[
(\bm{\alpha};\bm{\beta})=((1,0);(2,0)).
\]

\item  If $\bm{\nu}=(0,2)$, then necessarily $\bm{\mu}=(2,0)$ and the 
only nontrivial index $(\bm{\alpha};\bm{\beta})$ is 
\[
(\bm{\alpha};\bm{\beta})=((0,1);(2,0)).
\]
\end{enumerate}
\end{prop}

Note that since $\lvert \bm{\nu}\rvert\ge 2$, the above list covers 
all the possibilities. We will denote by $(\bm{\alpha}^\circledast;
\bm{\beta}^\circledast)$ the indicated extremal index $(\bm{\alpha};\bm{\beta})$
in each of the cases (i)--(v). 

\begin{proof}[Proof of Proposition~\ref{prop:index}]
We show how to obtain $\mathrm{(i)}$, $\mathrm{(ii)}$ and $\mathrm{(iv)}$. 
The remaining cases $\mathrm{(iii)}$ and $\mathrm{(v)}$ are analogous 
and omitted. 
We recall the compatibility conditions on the index set 
$\indsetS_k^{\mathrm{L}}(\bm{\mu},\bm{\nu})$. After all, the assertion 
$(\bm{\alpha};\bm{\beta})\in \indsetS_k^{\mathrm{L}}(\bm{\mu},\bm{\nu})$ 
means that $(\bm{\alpha};\bm{\beta})\in\indsetT_{k;2,2}^{\mathrm{L}}$ plus the
\begin{equation}\label{eq:index-bruno}
\sum_{i=1}^k\lvert \bm{\beta}_i\rvert\bm{\alpha}_i=\bm{\nu},
\qquad \sum_{i=1}^k\bm{\beta}_i=\bm{\mu},
\end{equation}
where each $\bm{\beta}_i$ has $\lvert \bm{\beta}_i\rvert\ge 1$, and the 
multi-indices $\rm{\alpha}_i$ are strictly increasing with $i$ in the 
lexicographical ordering. From these assumptions it is immediate that each 
$\bm{\alpha}_i$ satisfies $\bm{\alpha}_i\preceq_{\mathrm{L}}\bm{\nu}$. 

As for assertion $\mathrm{(i)}$, we see that equality $\bm{\alpha}_i=(j,l)$ 
could hold only if $k=1$, with $\bm{\alpha}_1=(j,l)$ and $|\bm{\beta}_1|=1$. 
But then $\lvert\bm{\mu}\rvert=1$, which would contradict our assumption that 
$\lvert \bm{\mu}\rvert\ge 2$. Hence, given the structure of the lexicographic
ordering, for any index $i$, we have 
$\bm{\alpha}_i\preceq_{\mathrm{L}} (j,l-1)$. However, if equality holds here, 
that is, if for some $i_0$ we have $\bm{\alpha}_{i_0}=(j,l-1)$, we find from 
\eqref{eq:index-bruno} that $\lvert \bm{\beta}_{i_0}\rvert=1$, whereas 
the sum on the left-hand side, taken over all other indices $i\neq i_0$, 
must equal $(0,1)$. 
As a consequence, only $k=2$ is possible, and then 
$\bm{\alpha}=((0,1),(j,l-1))$. In addition, we get that $|\bm{\beta}_1|=
|\bm{\beta_2}|=1$, so that by the second relation in \eqref{eq:index-bruno},
$\lvert \bm{\mu}\rvert=2$ holds. Given the assumptions on $\bm\mu$,
the only remaining possibility is $\bm{\mu}=(2,0)$, and then 
$\bm{\beta}_1=\bm{\beta}_2=(1,0)$. 

We turn our attention to the assertion $\mathrm{(ii)}$. 
In a similar manner as above, since the weighted sum of the multi-indices 
$\bm{\alpha}_i$ equals $(j,0)$, we see that for each index $i=1,\ldots,k$, 
$\bm{\alpha}_i=(a_i,0)$ for some $a_i\in\N$ with $0<a_i\le j$. It is
clear that $a_{i_0}=j$ could occur for some $i_0$ only if $i_0=k=1$, 
$\lvert \bm{\beta}_1\rvert=1$ and $\lvert\bm{\mu}\rvert=1$, 
which again would contradict our assumption $\lvert \bm{\mu}\rvert\ge 2$. 
It follows that $a_i\le j-1$ for each $i$. 
Next, the only way we could have $\bm{\alpha}_{i_0}=(j-1,0)$ for some 
$i_0$ is if $i_0=k=2$ and correspondingly $\bm{\alpha}=((1,0),(j-1,0))$. 
The remaining properties are immediate.

Finally, to see why $\mathrm{(iv)}$ holds, we analogously find that each 
$\bm{\alpha}_i$ is of the form $\bm{\alpha}_i=(a_i,0)$, where $0<a_i\le 2$. 
In view of \eqref{eq:index-bruno}, 
\[
a_1|\bm{\beta}_1|+\cdots +a_k|\bm{\beta}_k|=2,
\]
with $|\bm{\beta}_i|\ge1$ for each $i$. This is possible only if $1\le k\le2$.
If $k=2$, we get that $a_1=a_2=1$ and $|\bm{\beta}_1|=|\bm{\beta}_2|=1$,
which leads to $\bm{\alpha}_1=\bm{\alpha}_2=(1,0)$. This gets excluded 
on the basis of the monotonicity requirement $\bm{\alpha}_1\prec_{\mathrm{L}}
\bm{\alpha}_2$, so $k=1$ is the only possibility. So the requirement 
\eqref{eq:index-bruno} now reads $a_1|\bm{\beta}_1|=2$ and 
$\bm{\beta}_1=\bm{\mu}$. If $a_1=2$, then $\lvert \bm{\beta}_1\rvert=1$, and
consequently $\lvert\bm{\mu}\rvert=1$, which is contrary to our assumption that 
$\lvert \bm{\mu}\rvert\ge 2$. The only remaining alternative is that
$\bm{\alpha}_1=(1,0)$ and $\lvert \bm{\beta}_1\rvert=2$. Since 
$\bm{\beta}_1=\bm{\mu}$, and the only admissible $\bm{\mu}$ of length 
$2$ is $\bm{\mu}=(2,0)$, it follows that $\bm{\beta}_1=(2,0)$, and the 
claim follows.
\end{proof}

We observe that in each of the cases $\mathrm{(i)}$-$\mathrm{(v)}$, the 
lexicographically maximal $\bm{\alpha}_i$ occurs as the index $i=k$,
where $k\in\{1,2\}$ and 
$(\bm{\alpha};\bm{\beta})\in \indsetS_k^{\mathrm{L}}(\bm{\mu},\bm{\nu})$ and 
$\bm{\mu}=\bm{\mu}_0:=(2,0)$ while $|\bm{\nu}|\ge2$. If we put  
\[
A(\bm{\nu})=\max_k\max_{(\bm{\alpha};\bm{\beta})\in 
\indsetS_k^{\mathrm{L}}(\bm{\mu}_0,\bm{\nu})}\bm{\alpha}_k
\]
where the maximum is taken lexicographically over the entire range 
$k=1,\ldots,\lvert \bm{\nu}\rvert$, then the maximum occurs for $k=2$
unless if $\bm{\nu}=(2,0)$ or $\bm{\nu}=(0,2)$. Moreover, if $\bm{\nu}=(2,0)$ 
or $\bm{\nu}=(0,2)$, the maximum occurs for $k=1$. Let $k_{\bm{\nu}}\in\{1,2\}$
be the parameter value for which the maximum is attained, depending on 
$\bm\nu$, as just explained.
In any of the instances $\mathrm{(i)}$-$\mathrm{(v)}$, there exists  
a unique extremal pair $(\bm{\alpha}^\circledast;\bm{\beta}^\circledast)\in 
\indsetS_k^{\mathrm{L}}(\bm{\mu}_0,\bm{\nu})$ provided that $k=k_{\bm\nu}$. 
Next, let $\indsetS_{k,\circledast}^{\mathrm{L}}(\bm{\mu}_0,\bm{\nu})$ denote 
the depleted index set
\[
\indsetS_{k,\circledast}^{\mathrm{L}}(\bm{\mu}_0,\bm{\nu})=
\begin{cases}
\indsetS_{k}^{\mathrm{L}}(\bm{\mu}_0,\bm{\nu}),& \text{if}\;\;
k\ne k_{\bm\nu},
\\
\indsetS_k^{\mathrm{L}}(\bm{\mu}_0,\bm{\nu})\setminus
\{(\bm{\alpha}^\circledast;\bm{\beta}^\circledast)\},& 
\text{if}\;\;
k=k_{\bm\nu},
\end{cases} 
\]
and consider the associated expression in the context of the multivariate
Fa\`a di Bruno formula:
\begin{equation}
\label{eq:G-star}
\mathcal{G}^\circledast_{\bm{\mu}_0,\bm{\nu}}(\bm{\Psi}_{s,t}):=\bm{\nu}!
\sum_{k=1}^{\lvert\bm{\nu}\rvert}\sum_{(\bm{\alpha};
\bm{\beta})\in \indsetS^{\mathrm{L}}_{k,\circledast}(\bm{\mu}_0,\bm{\nu})} 
\prod_{j=1}^{k}
\frac{[\partial ^{\bm{\alpha}_j}\bm{\Psi}_{s,t}]^{\bm{\beta}_j}}{\bm{\beta}_j!
[\bm{\alpha}_j!]^{\lvert\bm{\beta}_j\rvert}}.
\end{equation}
Then $\mathcal{G}_{\bm{\mu}_0,\bm{\nu}}(\bm{\Psi}_{s,t})$ splits as follows
(where $(\bm{\alpha}^\circledast;\bm{\beta}^\circledast)
=(\bm{\alpha}^\circledast_1,\ldots,
\bm{\alpha}^\circledast_{k_{\bm{\nu}}};\bm{\beta}^\circledast_1,\ldots,
\bm{\beta}_{k_{\bm{\nu}}}^\circledast)$):
\begin{equation}
\label{eq:G-star2}
\mathcal{G}_{\bm{\mu}_0,\bm{\nu}}(\bm{\Psi}_{s,t})=
\mathcal{G}^\circledast_{\bm{\mu}_0,\bm{\nu}}(\bm{\Psi}_{s,t})+
\mathcal{H}_{\bm{\mu}_0,\bm{\nu}}(\bm{\Psi}_{s,t}),\quad
\mathcal{H}_{\bm{\mu}_0,\bm{\nu}}(\bm{\Psi}_{s,t})
:=\bm{\nu}!\prod_{j=1}^{k_{\bm{\nu}}}
\frac{[\partial ^{\bm{\alpha}^\circledast_j}
\bm{\Psi}_{s,t}]^{\bm{\beta}^\circledast_j}}{\bm{\beta}^\circledast_j!
[\bm{\alpha}^\circledast_j!]^{\lvert\bm{\beta}^\circledast_j\rvert}}.
\end{equation}
If $\bm{\nu}=(2,0)$, the depleted index set 
$\indsetS_{k,\circledast}^{\mathrm{L}}(\bm{\mu}_0,\bm{\nu})$ is empty for
$k\in\{1,2\}$, which gives that
\begin{equation}
\label{eq:G-star3}
\mathcal{G}^\circledast_{\bm{\mu}_0,\bm{\nu}}(\bm{\Psi}_{s,t})=0\quad \text{if}\,\,\,
\bm{\nu}=(2,0).
\end{equation}

\subsection{Polynomial complexity classes}\label{ss:pol-complexity}
In order to make sure that the algorithm outlined above in 
Subsection~\ref{ss:alg-flow-over} does not break down, we need to carefully 
keep track of the 
dependency structure of the coefficient functions involved. In particular, 
when solving for the coefficient function $\hat{\psi}_{j,l}$
in terms of a Herglotz operator applied to a function $g_{j,l}$, we need 
to know that $g_{j,l}$ may be computed in terms of functions already 
determined in previous steps of the algorithm.
To help with this, we introduce for a nonnegative integer $j$ and a 
subset $\Sigma\subset\N^2$
the {\em polynomial complexity class} $\POL(j,\Sigma)$, defined as the 
following function class on the unit circle $\T$:
\begin{multline*}
\POL(j,\Sigma)=\R\Big[\Re\zeta,\Im\zeta,
\partial_{r,\theta}^\alpha R(\zeta),\Re b_r^{(k)},\Im b_r^{(k)}, 
\Re \hat{\psi}_{p,q},\Im\hat{\psi}_{p,q}, \Re \hat{\psi}'_{p,q},
\Im\hat{\psi}'_{p,q}\,:
\\
\,k\ge 0,\; \,0\le r\le j,\; (p,q)\in \Sigma\Big].
\end{multline*}
Here, $\R[X:Y]$ denotes the class of multivariate polynomials with 
real coefficients in the variables $X$,
restricted by the condition $Y$. In other words, $\POL(j,\Sigma)$ is 
the collection of multivariate 
polynomials in the expressions $\Re\zeta,\Im\zeta,
\partial_{r,\theta}^\alpha R(\zeta),\Re b_r^{(k)},\Im b_r^{(k)}, 
\Re \hat{\psi}_{p,q},\Im\hat{\psi}_{p,q}, \Re \hat{\psi}'_{p,q}$, and 
$\Im\hat{\psi}'_{p,q}$, under
the conditions $k\ge 0$, $0\le r \le j$ and  $(p,q)\in \Sigma$.
If there is no dependence on any of the functions $b_j$, we simplify the 
notation and write
$\POL(\Sigma)$ for the polynomial complexity class.
In connection with these classes, we will find it useful to introduce 
for nonnegative integers $p$ and $q$ the rectangular index sets
\[
\Sigma_{p,q}=\{ (a,b)\in\N^2\,:\,a\le p\;\text{ and }\;b\le q\}.
\]

\subsection{The semiclassical case of the orthogonal foliation
flow}
\label{ss:flow2}
We first explore {\sc Step 1} of the algorithmic procedure outlined in 
Subsection~\ref{ss:alg-flow-over}. We recall the notation
${\bf \Psi}_{0,t}=(\lvert \psi_{0,t}\rvert, \arg\psi_{0,t})$ from 
\eqref{eq:bPsist}. 
Moreover, we recall that $\hrho_1$ is as in Propositions \ref{prop:f-T-hol}
and \ref{prop-rho1sigma1}.
We have already established the regularity of $\psi_{0,t}$ in the
implicit function theorem of Subsection~\ref{ss:psi-ext}. 
We proceed to compute the Taylor coefficients in $t$, 
and highlight the algorithmic aspects.

\begin{prop}\label{prop:conf-expl}
The Taylor coefficients $\hat{\psi}_{0,l}$ in the variable $t$ near $t=0$ 
of the conformal mapping $\psi_{0,t}$ with
\[
\psi_{0,t}(\zeta)=\sum_{l=0}^{2\kappa+1} t^l\hat{\psi}_{0,l}(\zeta)
+\Ordo(t^{2\kappa+2}),
\]
are uniquely determined by the level-curve requirement
$$
R\circ\psi_{0,t}(\zeta)=\frac{t^2}{2},\qquad \zeta\in\T,
$$
together with the monotonicity condition that the images $\psi_{0,t}(\D_\e)$ 
grow with $t$ and the normalization $\psi'_{0,t}(\infty)>0$. 
Moreover, as such, they are given by
\begin{align*}
\hat{\psi}_{0,0}(\zeta)&=\zeta,
\\
\hat{\psi}_{0,1}(\zeta)&=-\zeta\Hop_{\D_\e}
\big[(4\hDelta R)^{-\frac{1}{2}}\big](\zeta),
\end{align*}
and, more generally, by
\begin{align*}
\hat{\psi}_{0,l}(\zeta)&=\zeta\Hop_{\D_\e}\big[(4\hDelta R)^{-\frac{1}{2}}
\frakG_l\big]
(\zeta),\qquad l=2,\ldots,2\kappa+3,
\end{align*}
where $\frakG_l(\zeta)\in\POL(\Sigma_{0,l-1})$ is given by the formula
($\bm{\mu}_0=(2,0)$)
\begin{multline*}
\frakG_l(\zeta):
=\frac{1}{(l+1)!}\bigg\{4(\hDelta R)\,\calG_{\bm{\mu}_0,(0,l+1)}
^\circledast(\bm{\Psi}_{0,t})\big\vert_{t=0}
\\
+\sum_{3\le\lvert \bm{\mu}\rvert\le l+1}
(\partial_r^{\mu_1}\partial_\theta^{\mu_2}
R)\,\calG_{\mu,(0,l+1)}(\bm{\Psi}_{0,t})\Big\vert_{t=0}-2(l+1)(\hDelta R)^{\frac12}
\frakg_l\bigg\},
\end{multline*}
where
\[
\frakg_l:=\partial_t^l\lvert \psi_{0,t}\rvert\big\vert_{t=0}-l!
\Re(\bar{\zeta}\hat{\psi}_{0,l}).
\]
The coefficient functions $\hat{\psi}_{0,l}$ all extend holomorphically to 
the domain $\D_\e(0,\hrho_1)$.
\end{prop}

\begin{proof}
By Proposition \ref{prop:conformal}, the conformal mappings $\psi_{0,t}$ 
are uniquely defined by the given requirements, and 
$\psi_{0,0}(\zeta)=\zeta$ holds. 
Moreover, since $t\mapsto\psi_{0,t}$ is smooth, the validity of the indicated
expansion follows from Taylor's formula, and the first coefficient then equals
$\hat{\psi}_{0,0}(\zeta)=\psi_{0,0}(\zeta)=\zeta$. 
In view of Taylor's formula applied to the function 
$t\mapsto R\circ\psi_{0,t}$, we have that
\begin{equation}\label{eq:phi-level}
(R\circ\psi_{0,t})(\zeta)=\sum_{l=0}^{2\kappa+1}\frac{t^l}{l!}
\partial_t^l(R\circ\psi_{0,t}(\zeta))\big\vert_{t=0}+
\Ordo(\lvert t\rvert^{2\kappa+2}).
\end{equation}
Since by assumption $R\circ\psi_{0,t}(\zeta)=\tfrac{t^2}{2}$ holds on $\T$, 
we find that for $\zeta\in\T$, 
\begin{equation}
\label{eq:eq-R-phi}
\partial^l_t(R\circ\psi_{0,t})(\zeta)\big\vert_{t=0}=\begin{cases}1,& 
\text{for}\;l=2,\\ 
0,&\text{for}\;l\ne2.\end{cases}
\end{equation}
It is automatic that \eqref{eq:eq-R-phi} holds for $0\le l\le1$, since
$R$ is quadratically flat on $\T$. We now consider $l=2$. 
By the multivariate Fa\`a di Bruno formula \eqref{eq:bruno-R} with $s=0$
treated as constant, together with the quadratic flatness of $R$ near the 
unit circle $\T$ a calculation shows that
\[
\partial^2_t(R\circ\psi_{0,t})\big\vert_{t=0}=(\partial_r^2 R)
(\partial_t|\psi_{0,t}|)^2\big|_{t=0}=4\hDelta R\,
[\Re(\bar\zeta\hat{\psi}_{0,1})]^2\quad\text{on}\,\,\,\T.
\]
Since the left-hand side equals $1$ by \eqref{eq:eq-R-phi}, we may solve for 
$\Re(\bar\zeta\hat{\psi}_{0,1})$ using either the positive or the negative root.
We choose the negative square root, which gives that
\begin{equation}
\label{eq:phi01-0}
\partial_t|\psi_{0,t}|\big|_{t=0}=
\Re(\bar\zeta\hat{\psi}_{0,1})=-(4\hDelta R)^{-\frac12}\quad\text{on}\;\;\T.
\end{equation}
This choice is the one which is compatible with the growth of the domains 
$\psi_{0,t}(\D_\e)$ as $t$ increases (so that the loops $\psi_{0,t}(\T)$ move 
inward). 
Finally, we solve this equation by means of the formula
\begin{equation}
\label{eq:phi01}
\hat{\psi}_{0,1}(\zeta)=-\zeta\Hop_{\D_\e}\big[(4\hDelta R)^{-\frac12}
\big](\zeta),
\end{equation}
as in {\sc Step 3} of the algorithmic procedure in 
Subsection~\ref{ss:alg-flow-over}. Here, the uniqueness of the solution
follows from Remark~\ref{rem:flow-remark} {\rm (a)}. 
Since $(4\hDelta R)^{-\frac12}$ has 
a polarization which is holomorphic in $(z,\bar w)$ for 
$(z,w)\in\hat{\mathbb{A}}(\hrho_1,\sigma_1)$, 
the function $\hat{\psi}_{0,1}$ extends holomorphically to
$\D_\e(0,\hrho_1)$, by 
Proposition~\ref{prop:f-T-hol} and Remark~\ref{rem:Hop-hol}.

As for the higher order Taylor coefficients, we again apply the multivariate
Fa\`a di Bruno formula \eqref{eq:bruno-R}.
As a result, on the circle $\zeta\in\T$ we have for $l\ge3$ that
(apply \eqref{eq:phi01-0} in the last step)
\begin{multline}
\label{eq:deriv-R-phi}
\partial_{t}^{l}(R\circ\psi_{0,t})\big\vert_{t=0}=\sum_{2\le\lvert \bm{\mu}
\rvert\le l}(\partial^{\mu_1}_r\partial^{\mu_2}_\theta R)\,\mathcal{G}
_{\bm{\mu},(0,l+1)}(\bm{\Psi}_{0,t})\Big\vert_{t=0}
\\
=4l\,(\hDelta R)(\partial^{l-1}_t\lvert 
\psi_{0,t}\rvert)\,(\partial_t\vert \psi_{0,t}\vert) 
+\mathcal{G}^\circledast_{\bm{\mu}_0,l}(\bm{\Psi}_{0,t})
\Big\vert_{t=0}
\\
+\sum_{3\le\lvert \bm{\mu}
\rvert\le l}(\partial^{\mu_1}_r\partial^{\mu_2}_\theta R)\,\mathcal{G}
_{\bm{\mu},(0,l)}(\bm{\Psi}_{0,t})\Big\vert_{t=0}
\\
=-l!(4\hDelta R)^{\frac12}\,\Re(\bar\zeta\hat{\psi}_{0,l-1})
-l(4\hDelta R)^{\frac12}\frakg_{l-1}+\mathcal{G}^\circledast_{\bm{\mu}_0,(0,l)}
(\bm{\Psi}_{0,t})\Big\vert_{t=0}
\\
+\sum_{3\le\lvert \bm{\mu}
\rvert\le l}(\partial^{\mu_1}_r\partial^{\mu_2}_\theta R)\,\mathcal{G}
_{\bm{\mu},(0,l)}(\bm{\Psi}_{0,t})\Big\vert_{t=0},
\end{multline}
where $\bm{\mu}_0=(2,0)$ and we recall that
\[
\frakg_{l-1}=\partial_t^{l-1}\lvert \psi_{0,t}\rvert\big\vert_{t=0}-(l-1)!
\Re(\bar{\zeta}\hat{\psi}_{0,l-1}).
\]
An elementary computation shows that the highest order derivatives cancel out, 
and it follows that $\frakg_{l-1}\in\POL(\Sigma_{0,l-2})$. 

We recall that the expression $\mathcal{G}^\circledast_{\bm{\mu}_0,(0,l+1)}
(\bm{\Psi}_{0,t})$ appearing in the above formula is as in \eqref{eq:G-star}.
We write
\[
\frakG_{l-1}=\frac{1}{l!}\bigg\{
-l(4\hDelta R)^{\frac12}\frakg_{l-1}+
\mathcal{G}^\circledast_{\bm{\mu}_0,(0,l)}(\bm{\Psi}_{0,t})\big\vert_{t=0}
+\sum_{3\le\lvert \bm{\mu}
\rvert\le l}(\partial^{\mu_1}_r\partial^{\mu_2}_\theta R)\,
\mathcal{G}_{\bm{\mu},(0,l)}(\bm{\Psi}_t)\Big\vert_{t=0}\bigg\},
\]
and claim that $\frakG_{l-1}\in\POL(\Sigma_{0,l-2})$. 
We already saw that $\frakg_l$ has this property, and hence 
$(\hDelta R)^{\frac12}\frakg_{l-1}=(\hDelta R)\Re(-\bar\zeta\hat{\psi}_{0,1})$ does as well. 
That the same holds for the remaining two terms of the above formula 
can be seen from Proposition~\ref{prop:index}, and hence it follows that 
$\frakG_{l-1}\in\POL(\Sigma_{0,l-2})$.
It is a consequence of \eqref{eq:deriv-R-phi} that the 
condition \eqref{eq:eq-R-phi} for $l\ge3$ may be expressed as
\[
-(4\hDelta R)^{\frac12}\Re(\bar\zeta\hat{\psi}_{0,l-1})+\frakG_{l-1}=0,
\qquad l=3,4,5,\ldots.
\]
This is an equation of a kind we have met before, and we know that 
a solution $\hat{\psi}_{0,l}$ is supplied by the formula 
(change $l$ by $l+1$ in the previous relation)
\begin{equation}
\label{eq:psi0l-formula}
\hat{\psi}_{0,l}(\zeta)=\zeta\Hop_{\D_\e}
\bigg[\frac{\frakG_{l}}{(4\hDelta R)^{\frac12}}\bigg](\zeta),\qquad 
l=2,3,4,\ldots.
\end{equation}
Let us assume for the moment that the lower order terms $\hat\psi_{0,b}$ 
with $0\le b\le l-1$ all extend holomorphically to an exterior disk 
$\D_\mathrm{e}(0,\hrho_1)$. 
Then the entire expression inside brackets in 
\eqref{eq:psi0l-formula} polarizes to extend to a $2\sigma_1$-fattened 
diagonal annulus $\hat{\mathbb{A}}(\hrho_1,\sigma_1)$ given that
various partial derivatives of $R$ do, as well as
$(\hDelta R)^{-\frac12}$, which follows from Proposition
\ref{prop-rho1sigma1}. Moreover, since $\hrho_1$ is big
enough to guarantee that $\hrho_1\ge (\sqrt{1+\sigma_1^2}+\sigma_1)^{-1}$,
then in view of Proposition \ref{prop:f-T-hol}, the expression on the
right-hand side of \eqref{eq:psi0l-formula} will be holomorphic in the
same exterior disk $\D_\mathrm{e}(0,\hrho_1)$ as well, by Remark \ref{rem:rho'}. 
But then we have enough to keep the iteration going, and obtain that all 
the terms $\hat\psi_{0,l}$ extend holomorphically to a single exterior disk 
$\D_\mathrm{e}(0,\hrho_1)$. 
\end{proof}

\subsection{Taylor expansion of the weight term in the
master equation}
\label{ss:flow2-2}
We continue with the Taylor expansion of the composition $R\circ\psi_{s,t}$ 
in terms of powers of $s$ and $t$, where the starting point is the 
application of the Fa\`a di Bruno formula in \eqref{eq:bruno-R}. 
We recall the definition \eqref{eq:triangindset} of the triangular index 
set $\indsett_n$. 
We work under the assumption that $\psi_{s,t}$ depends sufficiently smoothly 
on both $(s,t)$ near $(0,0)$. This assumption gets justified in the stepwise
proof which we outlined in Subsection \ref{ss:flow2}, which retrieves the 
Taylor coefficients of $\psi_{s,t}$ in $(s,t)$ iteratively.  
We use the notion of polynomial complexity classes $\POL(j,\Sigma)$ 
and the index sets $\Sigma_{p,q}$ from Subsection~\ref{ss:pol-complexity}.

\begin{prop}\label{prop:exp-R}
On the unit circle $\T$, the function $R\circ\psi_{s,t}$ enjoys the expansion
\begin{equation*}
\label{eq:exp-R}
2R\circ\psi_{s,t}=2R\circ\psi_{0,t}+\sum_{(j,l)\in\indsett_{2\kappa}}s^{j+1}t^l
\frakR_{j,l}+\Ordo\big(\lvert s\rvert\big(\lvert s\rvert^{\kappa+\frac12}
+\lvert t\rvert^{2\kappa+1}\big)\big),
\end{equation*}
where $\frakR_{0,0}=0$, while for the remaining indices $(j,l)\neq (0,0)$, 
we have
\[
\frakR_{j,l}=\frac{2}{(j+1)!l!}\,\Big\{
(4\hDelta R)\,
[\mathcal{H}_{\bm{\mu}_0,(j+1,l)}(\bm{\Psi}_{s,t})]\Big\vert_{s=t=0}
+\frakr_{j,l}^*\Big\},
\]
where $\bm{\mu}_0=(2,0)$. Here, the main term is given in terms of 
$\mathcal{H}_{\bm{\mu}_0,(j+1,l)}(\bm{\Psi}_{s,t})$, defined by 
\[
\mathcal{H}_{\bm{\mu}_0,(j+1,l)}(\bm{\Psi}_{s,t})=
\begin{cases}
l\,(\partial_t|\psi_{s,t}|)(\partial_s^{j+1}\partial_t^{l-1}|\psi_{s,t}|),
\quad\text{if}\,\,\,j\ge0,\,l\ge1, 
\\
(j+1)(\partial_s|\psi_{s,t}|)(\partial_s^j|\psi_{s,t}|), \quad\,\,\text{if}\,\,\,
j\ge2,\,l=0, 
\\
(\partial_s|\psi_{s,t}|)^2,\qquad\qquad\qquad\quad\;\;\text{if}\,\,\,j=1,\,l=0,
\end{cases}
\]
while the term $\frakr_{j,l}^*$, considered as a remainder, is given by
\begin{equation*}
\frakr_{j,l}^*=(4\hDelta R)\,\mathcal{G}^\circledast_{\bm{\mu}_0,(j+1,l)}
(\bm{\Psi}_{s,t})+\sum_{3\le\lvert\bm{\mu}\rvert\le j+l+1}
(D_{r,\theta}^{\bm{\mu}} R)\,\mathcal{G}_{\bm{\mu},(j+1,l)}(\bm{\Psi}_{s,t})
\Big\vert_{s=t=0},
\end{equation*}
where we recall that $\mathcal{G}^\circledast_{\bm{\mu}_0,(j+1,l)}(\bm{\Psi}_{s,t})$ 
is given by \eqref{eq:G-star}. 
For $j\ge0$ and $l\ge1$, we have 
\[
\frakr_{j,l}^*\in\POL(\Sigma),\quad \text{with}\quad 
\Sigma=\{(p,q)\in\Sigma_{j+1,l}:\,(p,q)\prec_{\mathrm{L}}(j+1,l-1)\}).
\]
In a similar fashion, for $j\ge1$, we have that the Taylor coefficient  
$\frakR_{j,0}\in \POL(\Sigma_{j,0})$.
Moreover, the implied constant in the above expansion of $R\circ\psi_{s,t}$ 
remains bounded if the weight $R$ is confined to a uniform family in 
$\mathcal{W}(\hrho_0,\sigma_0)$ for some fixed  $0<\hrho_0<1$ and $\sigma_0>0$, 
while the functions $\psi_{s,t}$ are smooth with bounded norms 
in $C^{2\kappa+4}$ with respect to $(s,t)$ in a neighborhood of $(0,0)$, 
uniformly on the circle $\T$.
\end{prop}

\begin{proof}
The fact that $R\circ\psi_{s,t}$ enjoys an expansion of the indicated form 
for some coefficients $\frakR_{j,l}$ with the given error term is an 
immediate consequence of the multivariate Taylor's formula. The coefficients 
$\frakR_{j,l}$ are then obtained from the successive partial derivatives 
\eqref{eq:bruno-R}. 
It just remains to calculate them:
\begin{multline}
\frakR_{j,l}=\frac{2}{(j+1)!l!}\,\partial_s^{j+1}\partial_t^l\big(
R\circ\psi_{s,t}-R\circ\psi_{0,t}\big)\Big|_{s=t=0}=
\frac{2}{(j+1)!l!}\,\partial_s^{j+1}\partial_t^l\big(
R\circ\psi_{s,t}\big)\Big|_{s=t=0}
\\
=\frac{2}{(j+1)!l!}\,
\sum_{2\le\lvert\bm{\mu}\rvert\le j+l+1}(D_{r,\theta}^{\bm{\mu}} 
R)\,
\mathcal{G}_{\bm{\mu},(j+1,l)}(\bm{\Psi}_{s,t})\Big\vert_{s=t=0}.
\label{eq:Aformula-01}
\end{multline}
In particular, $\frakR_{0,0}=0$, as the sum is over the empty set. 
In the right-hand side, the sum over $|\bm{\mu}|=2$ is special as the 
only nontrivial contribution comes from $\bm{\mu}=\bm{\mu}_0=(2,0)$, 
by \eqref{eq:flatR-1.01}:
\begin{equation}
\sum_{\lvert\bm{\mu}\rvert=2}(D_{r,\theta}^{\bm{\mu}} 
R)\,
\mathcal{G}_{\bm{\mu},(j+1,l)}(\bm{\Psi}_{s,t})\Big\vert_{s=t=0}=
(4\hDelta R)\,
\mathcal{G}_{\bm{\mu}_0,(j+1,l)}(\bm{\Psi}_{s,t})\Big\vert_{s=t=0}\quad
\text{on}\,\,\,\,\T.
\label{eq:Aformula-02}
\end{equation}
Here, we use the fact that $\partial_r^2R=4\hDelta R$ on $\T$.
It follows that for $(j,l)\ne(0,0)$, we have on $\T$ that
\begin{multline*}
\frakR_{j,l}=\frac{2}{(j+1)!l!}\,\Big\{
(4\hDelta R)\,
\mathcal{G}_{\bm{\mu}_0,(j+1,l)}(\bm{\Psi}_{s,t})
+\sum_{3\le\lvert\bm{\mu}\rvert\le j+l+1}(D_{r,\theta}^{\bm{\mu}} 
R)\,
\mathcal{G}_{\bm{\mu},(j+1,l)}(\bm{\Psi}_{s,t})\Big\}\Big\vert_{s=t=0}.
\end{multline*}
We write $\bm{\nu}:=(j+1,l)$, and split the expression 
$\mathcal{G}_{\bm{\mu}_0,\bm{\nu}}(\bm{\Psi}_{s,t})$ further according to 
formula \eqref{eq:G-star2}:
\[
\mathcal{G}_{\bm{\mu}_0,\bm{\nu}}(\bm{\Psi}_{s,t})=
\mathcal{G}^\circledast_{\bm{\mu}_0,\bm{\nu}}(\bm{\Psi}_{s,t})+
\mathcal{H}_{\bm{\mu}_0,\bm{\nu}}(\bm{\Psi}_{s,t}).
\]
We turn to the task of expressing 
\[
\mathcal{H}_{\bm{\mu}_0,\bm{\nu}}(\bm{\Psi}_{s,t})=\bm{\nu}!\prod_{j=1}^{k_{\bm{\nu}}}
\frac{[\partial ^{\bm{\alpha}^\circledast_j}
\bm{\Psi}_{s,t}]^{\bm{\beta}^\circledast_j}}{\bm{\beta}^\circledast_j!
[\bm{\alpha}^\circledast_j!]^{\lvert\bm{\beta}^\circledast_j\rvert}}.
\] 
in explicit form in the various cases as outlined in Proposition 
\ref{prop:index}. First, if $j\ge0$ and $l\ge1$, then $k_{\bm{\nu}}=2$ and
\[
\mathcal{H}_{\bm{\mu}_0,\bm{\nu}}(\bm{\Psi}_{s,t})=l(\partial_t|\psi_{s,t}|)
(\partial_s^{j+1}\partial_t^{l-1}|\psi_{s,t}|).
\]
It remains to consider $j\ge1$ and $l=0$. If $j=1$ and $l=0$, then
\[
\mathcal{H}_{\bm{\mu}_0,\bm{\nu}}(\bm{\Psi}_{s,t})=(\partial_s|\psi_{s,t}|)^2,
\]
while if instead $j\ge2$ and $l=0$, then
\[
\mathcal{H}_{\bm{\mu}_0,\bm{\nu}}(\bm{\Psi}_{s,t})=(j+1)(\partial_s|\psi_{s,t}|)
(\partial_s^{j}|\psi_{s,t}|).
\]

It remains to discuss the algebraic properties of $\frakr^*_{j,l}$ for
$j\ge0,l\ge1$ and those of $\frakR_{j,0}$ for $j\ge1$. In view of Proposition
\ref{prop:index}, for $j\ge0,l\ge1$ all the indices $\bm{\alpha}_i$ have
\[
0\prec_{\mathrm{L}}\bm{\alpha}_0\prec_{\mathrm{L}}
\cdots\prec_{\mathrm{L}}\bm{\alpha}_k\prec_{\mathrm{L}}(j+1,l-1),
\] 
provided that $(\bm{\alpha};\bm{\beta})\in
\indsetS_k^{\mathrm{L}}(\bm{\mu},\bm{\nu})$ holds for a  
$k=1,\ldots,|\bm{\nu}|$, given that $|\bm{\mu}|\ge3$. In addition,
the same conclusion remains valid if $\bm{\mu}=\bm{\mu}_0=(2,0)$ provided 
that it is assumed that $(\bm{\alpha};\bm{\beta})\in
\indsetS_{k,\circledast}^{\mathrm{L}}(\bm{\mu}_0,\bm{\nu})$, which excludes the
extremal multi-index. After some additional
simplifications, this shows that $\frakr^*_{j,l}$ has the claimed form.
For $j\ge1$ and $l=0$, the assertion about the algebraic properties of 
$\frakR_{j,0}$ follows from the observation that if 
$(\bm{\alpha};\bm{\beta})\in \indsetS_k^{\mathrm{L}}(\bm{\mu},\bm{\nu})$ with
$\bm{\nu}=(j+1,0)$, then for $i=1,\ldots,k$, we have
$\bm{\alpha}_i=(a_i,0)$ with $0<a_i\le j$, by Proposition \ref{prop:index}.
The computational aspects are analogous to the case already discussed.
This completes the proof.
\end{proof}

\subsection{Taylor expansion of the remaining terms in the
master equation}

We recall that $h_s$ and $\psi_{s,t}$ stand for the functions
\[
h_s(\zeta)=\sum_{j=0}^\kappa s^jb_j(\zeta)\quad \text{and}\quad
\psi_{s,t}(\zeta)=\psi_{0,t}(\zeta)+
\sum_{\substack{(j,l)\in\indsett_{2\kappa+1}\\ j\ge 1}}
s^jt^l\hat{\psi}_{j,l}(\zeta),
\]
where $\kappa$ is a (big) positive integer. 
Moreover, the $b_j$ are bounded holomorphic functions in 
the exterior disk $\D_\e(0,\hrho_1)$, and 
$\psi_{0,t}$ is a conformal mapping of the exterior 
disk onto the exterior of the level curves 
$\Gamma_{t}$ of $R$ as above, and where 
$\hat{\psi}_{j,l}$ are holomorphic functions on $\D_\e(0,\hrho_1)$
with bounded derivative.
Let us denote by $\frakH_{j,l}$, $\frakR_{j,l}$ and 
$\frakJ_{j,l}$ the corresponding coefficient functions in the following 
three 
expansions (for $\zeta\in\T$):
\begin{align}
\label{eq:exp-align-1}2\Re (h_s\circ\psi_{s,t})(\zeta)&=
\sum_{(j,l)\in\indsett_{2\kappa}}s^jt^l
\frakH_{j,l}(\zeta)+\Ordo\big(\lvert s\rvert^{\kappa+\frac12}+
\lvert t\rvert^{2\kappa+1}\big),
\\
\label{eq:exp-align-2}2s^{-1}\big(R\circ\psi_{s,t}(\zeta)-\tfrac12{t^2}\big)&=
\sum_{(j,l)\in\indsett_{2\kappa}}
s^jt^l\frakR_{j,l}(\zeta)+\Ordo\big(\lvert s\rvert^{\kappa+\frac12}+
\lvert t\rvert^{2\kappa+1}\big),
\\
\label{eq:exp-align-3}\log\, \Re\big\{-\bar\zeta\partial_t\psi_{s,t}(\zeta)
\overline{\psi_{s,t}'(\zeta)}\big\}
&=\sum_{(j,l)\in\indsett_{2\kappa}}s^jt^l\frakJ_{j,l}(\zeta)
+\Ordo\big(\lvert s\rvert^{\kappa+\frac12}+\lvert t\rvert^{2\kappa+1}\big).
\end{align}
where \eqref{eq:exp-align-2} holds since 
$R\circ\psi_{0,t}(\zeta)=\tfrac{t^2}{2}$ on $\T$ 
as a matter of definition.
Moreover, we recall that by Proposition~\ref{prop:conf-expl}
we have that
\[
\exp(\frakJ_{0,0})=\Re(-\bar\zeta\partial_t\psi_{s,t}(\zeta)
\overline{\psi_{s,t}'(\zeta)})
\big\vert_{s=t=0}=
(4\hDelta R(\zeta))^{-\frac12}\quad \text{on }\,\T.
\]
We have already analyzed the coefficient functions $\frakR_{j,l}$ for
$(j,l)\in \indsett_{2\kappa}$ back in Proposition~\ref{prop:exp-R}. Here,
we refine the analysis and obtain a more convenient splitting of $\frakR_{j,l}$
into a main term plus a remainder, and express the coefficients 
$\frakH_{j,l}$ and $\frakJ_{j,l}$ in 
terms of the successive partial derivatives 
of the functions $b_j$ and $\psi_{s,t}$.

\begin{prop}\label{prop:exp-2}
In the above context, the Taylor coefficients 
$\frakH_{j,l}$ in $(s,t)$ of the function 
$2\Re h_s\circ\psi_{s,t}$ in $(s,t)$ according to 
\eqref{eq:exp-align-1} have the following properties.
For $j\ge 0$ we have
\[
\frakH_{j,0}=2\Re b_j+\frakh_{j,0}
\]
where $\frakh_{j,0}\in \POL(j-1,\Sigma_{j,0})$. On the other hand, for 
$(j,l)\in\indsett_{2\kappa}$ with $l\ge 1$, we see that 
$\frakH_{j,l}\in \POL(j,\Sigma_{j,l})$.

As for the Taylor coefficients $\frakR_{j,l}$ associated with $R\circ\psi_{s,t}$
according to \eqref{eq:exp-align-2}, we have
for $(j,l)\in\indsett_{2\kappa}$ with $j\ge 0$ and $l\ge 1$ that
\[
\frakR_{j,l}=
2(4\hDelta R)^{\frac{1}{2}}\Re(\bar\zeta\hat{\psi}_{j+1,l-1})+\frakr_{j,l}
\]
where 
\[
\frakr_{j,l}\in\POL\big(\big\{(p,q)\in \Sigma_{j+1,l}:
\,(p,q)\prec_{\mathrm{L}}(j+1,l-1)\big\}\big).
\]
As for the coefficients $\frakJ_{j,l}$ appearing in \eqref{eq:exp-align-3}, 
$\frakJ_{0,0}$ is given by
\[
\frakJ_{0,0}=\log\,\Re(-\bar\zeta\hat{\psi}_{0,1})=-\frac12\log(4\hDelta R), 
\]
while for $(j,l)\in\indsett_{2\kappa}\setminus\{(0,0)\}$ we see that 
$\frakJ_{j,l}\in \POL(\Sigma_{j,l+1})$.
\end{prop}

\begin{proof}
This follows from an application of the multivariate Taylor's formula, 
together with the multivariate Fa{\`a} di Bruno formula 
(Proposition~\ref{prop:bruno}), and the equation~\eqref{eq:exp-align-2} above.
Let us indicate the necessary computations, starting with the 
coefficients $\frakH_{j,l}$.
For $(j,l)\in\indsett_{2\kappa}\setminus\{(r,0):r\ge 0\}$ we have
\begin{equation}\label{eq:Hjl}
\frakH_{j,l}=
2\Re \sum_{i=0}^j\sum_{1\le \mu \le j+l-i} \partial^{\mu}b_i(\zeta)
\sum_{k=1}^{j+l-i}
\sum_{(\bm{\alpha},\beta)\in \indsetS_{k}^{\mathrm{L}}(\mu,(j-i,l))}
\prod_{r=1}^{k}
\frac{[\partial_{s,t}^{\boldsymbol{\alpha_r}}\psi_{s,t}(\zeta)]^{\beta_r}}
{(\boldsymbol{\alpha_r}!)^{\beta_r}\beta_r!},
\end{equation}
and consequently $\frakH_{j,l}\in \POL(j,\Sigma_{j,l})$, 
while for indices $(j,0)$ with $j\ge 0$ we have
\[
\frakH_{j,0}=2\Re b_j+\frakh_{j,0}
\]
where $\frakh_{j,0}\in \POL(j-1,\Sigma_{j,0})$ is given by
\begin{equation}\label{eq:hj0}
\frakh_{j,0}=2\Re\sum_{i=0}^{j-1}\sum_{1\le \mu \le j-i} \partial^{\mu}b_i(\zeta)
\sum_{k=1}^{j-i}\sum_{(\alpha,\beta)\in 
\indsetS_{k}^{\mathrm{L}}(\mu,(j-i,0))}\prod_{r=1}^{k}
\frac{[\partial_{s,t}^{\alpha_r}\psi_{s,t}]^{\beta_r}}
{(\alpha_r!)^{\beta_r}\beta_r!}.
\end{equation}
Turning to the claim about the coefficient $\frakR_{j,l}$, we note that 
\begin{equation}\label{eq:rjl}
\frakr_{j,l}=\frac{2}{(j+1)!l!}\frakr^*_{j,l}+
2(4\hDelta R)^{\frac12}\Big(
\frac{\partial^{j+1}_s\partial^{l-1}_t|\psi_{s,t}|}{(j+1)!(l-1)!}
-\Re(\bar\zeta\hat{\psi}_{j+1,l-1})\Big)\Big\vert_{s=t=0}.
\end{equation}
The claim in the proposition follows
from Proposition~\ref{prop:exp-R} together with the observation that
\begin{multline*}
\frac{\partial_{s}^{j+1}\partial_t^{l-1}|\psi_{s,t}(\zeta)|}{(j+1)!(l-1)!}
\Big\vert_{s=t=0}-\Re(\bar\zeta\hat{\psi}_{j+1,l-1}(\zeta))\\
\in\POL\big(\{(p,q)\in\Sigma_{j+1,l}:\,(p,q)\prec_{\mathrm{L}}(j+1,l-1)\}\big).
\end{multline*}
In order to see why this claim holds, we simply observe that the first 
term on the left-hand side is the Taylor coefficient in $(s,t)$ 
corresponding to the multi-index $(j+1,l-1)$ of the function 
$|\psi_{s,t}|$.
The Taylor expansion of this function may be found as follows. 
We notice that $\psi_{0,0}=\zeta$, so that
if we apply the generalized binomial theorem with exponent $\frac12$, 
we obtain
\begin{multline}
\label{eq:mod-psi-exp}
|\psi_{s,t}(\zeta)|=
\bigg|\Big(1+\sum_{(p,q)\ne (0,0)}s^pt^q\bar{\zeta}\hat{\psi}_{p,q}
\Big)^{\frac12}\bigg|^{2}
=\bigg|1+\sum_{k\ge 1}\binom{\frac12}{k}
\Big(\sum_{(p,q)\ne (0,0)}s^pt^q\bar\zeta\hat{\psi}_{p,q}\Big)^k\bigg|^2
\\
=1+\sum_{k,k'\ge 1}\binom{\frac12}{k}\binom{\frac12}{k'}
\Big(\sum_{(p,q)\ne (0,0)}s^pt^q\bar\zeta\hat{\psi}_{p,q}\Big)^k
\Big(\sum_{(p,q)\ne (0,0)}s^pt^q\zeta\overline{\hat{\psi}}_{p,q}\Big)^{k'}
\\
+2\Re\sum_{k\ge 1}\binom{\frac12}{k}\Big(\sum_{(p,q)\ne (0,0)}s^pt^q
\bar\zeta\hat{\psi}_{p,q}\Big)^k,\qquad \zeta\in\T.
\end{multline}
Apart from the contribution from the conformal mapping $\psi_{0,t}$, the 
series involve a truncation given by the index set $\indsett_{2\kappa+1}$, 
and hence we have no convergence issues.
The maximal index $(p,q)$ in the lexicographical ordering for which 
$\hat{\psi}_{p,q}$ appears 
in the Taylor coefficient for $s^{j+1}t^{l-1}$ of the above expression 
\eqref{eq:mod-psi-exp},
is easily seen to be $(j+1,l-1)$. The contribution corresponding to 
the maximal index comes from the last
term on the right-hand side of \eqref{eq:mod-psi-exp}, and equals
\[
2\Re\binom{\frac12}{1}\bar\zeta\hat{\psi}_{j+1,l-1}=
\Re(\bar\zeta\hat{\psi}_{j+1,l-1}).
\]
As for all the other indices, the contribution in the above sum to the 
Taylor coefficient
lies in the complexity class 
\[
\POL\big(\big\{(p,q)\in \Sigma_{j+1,l-1}:\,(p,q)
\prec_{\mathrm{L}}(j+1,l-1)\big\}\big),
\]
and the claim follows. 

Finally, we turn to the coefficient $\frakJ_{j,l}$. We know that 
$\frakJ_{0,0}=\log\,\Re(-\bar\zeta\hat{\psi}_{0,1})$,
while for indices $(j,l)\in\indsett_{2\kappa}\setminus\{(0,0)\}$ we apply the 
Fa{\`a} di Bruno's formula to the logarithm of the Jacobian expression to obtain
\begin{equation}\label{eq:Jjl}
\frakJ_{j,l}=\sum_{1\le \mu\le j+l}(-1)^{\mu}(\mu-1)!(4\hDelta R)^{\frac{\mu}{2}}
\sum_{k=1}^{j+l}\sum_{(\bm{\alpha},\beta)\in\indsetS_k^{\mathrm{L}}(\mu;(j,l))}
\prod_{r=1}^k\frac{[\partial_{s,t}^{\bm{\alpha}_r}
\Re(-\bar\zeta\partial_t\psi_{s,t}\overline{\psi'_{s,t}})]^{\beta_r}}
{(\bm{\alpha}_r!)^{\beta_r}\beta_r!}.
\end{equation}
As we may eliminate the half-powers of $\hDelta R$ by writing
\[
(4\hDelta R)^{\frac{\mu}{2}}=(4\hDelta R)^{\mu}(4\hDelta R)^{-\frac{\mu}{2}}
=(4\hDelta R)^{\mu}\Re(-\bar{\zeta}\hat{\psi}_{0,1}(\zeta))^\mu,\qquad \zeta\in\T,
\]
it follows that $\frakJ_{j,l}\in\POL(\Sigma_{j,l+1})$. 
\end{proof}

\subsection{Taylor expansion of the density in the 
master equation}
\label{ss:flow2-3}
We recall from Subsection~\ref{ss:alg-flow-over} the function
\[
\logdens_{s,t}(\zeta)=2\re h_s\circ\psi_{s,t}(\zeta)
-\tfrac{2}{s}\big((R\circ\psi_{s,t})(\zeta)-\tfrac12 t^2\big)
+\log\Big(\Re\big\{-\bar\zeta\partial_t\psi_{s,t}(\zeta)
\overline{\psi_{s,t}'(\zeta)}\big\}\Big).
\]
We compute the Taylor coefficients $\hat{\logdens}_{j,l}$
for $(j,l)\in\indsett_{2\kappa}$ given implicitly by
\[
\logdens_{s,t}(\zeta)=\sum_{(j,l)\in\indsett_{2\kappa}}
s^j t^l\hat{\logdens}_{j,l}(\zeta)
+\Ordo\big(\lvert s\rvert^{\kappa+\frac12}+\lvert t\rvert^{2\kappa+1}\big).
\]
This will determine what the master equation for the Taylor coefficients 
\eqref{eq:omega-id} entails
for the coefficient functions $b_k$ and $\hat{\psi}_{p,q}$ for 
$k\le \kappa$ and $(j,l)\in\indsett_{2\kappa+1}$.

We recall that the Taylor coefficients 
$\frakH_{j,l}$, $\frakh_{j,0}$, $\frakJ_{j,0}$, $\frakR_{j,0}$, and
$\frakr_{j,l}$
have appeared above in Propositions~\ref{prop:exp-R} and \ref{prop:exp-2}.
See in addition the explicit formul\ae{} \eqref{eq:Hjl}, \eqref{eq:hj0},
\eqref{eq:rjl},
and \eqref{eq:Jjl}.

\begin{prop}\label{prop:exp-3}
The coefficients $\hat{\logdens}_{j,l}(\zeta)$ in the above expansion are given
explicitly as follows. For $j=l=0$, we have
\[
\hat{\logdens}_{0,0}(\zeta)=2\Re b_0(\zeta) + 
\log\Re(-\bar\zeta\hat{\psi}_{0,1}(\zeta)),
\]
while for $l=0$ and $j= 1,2,3,\ldots$ the coefficient function is given by
\[
\hat{\logdens}_{j,0}(\zeta)=2\Re b_j(\zeta)+\frakT_{j,0},
\]
where $\frakT_{j,0}:=\frakh_{j,0}-\frakR_{j,0}
+\frakJ_{j,0}\in\POL(j-1,\Sigma_{j,1})$.
Moreover, for $j=0,1,2,\ldots$ and $l= 1,2,3,\ldots$ the coefficient function 
$\hat{\logdens}_{j,l}$ is given by
\[
\hat{\logdens}_{j,l}=-2(4\hDelta R(\zeta))^{\frac12}
\Re(\bar{\zeta}\hat{\psi}_{j+1,l-1}(\zeta))+\frakT_{j,l},
\]
where $\frakT_{j,l}:=\frakH_{j,l}-\frakr_{j,l}+\frakJ_{j,l}
\in\POL(j,\Sigma)$, with $\Sigma$ as the index set
\[
\Sigma=\{(p,q)\in \Sigma_{j+1,l}\cup\Sigma_{j,l+1}:\,
(p,q)\prec_{\mathrm{L}}(j+1,l-1)\}.
\]
\end{prop}

\begin{proof}
The formula for $\hat{\logdens}_{0,0}=\logdens_{s,t}\vert_{s=t=0}$ is 
immediate from the definition
\eqref{eq:omega-def}. Indeed, $\psi_{0,0}(\zeta)=\zeta$, and the 
formula \eqref{eq:omega-def} reads, where $h_s=2\Re \log f_s$,
\[
\logdens_{0,0}(\zeta)=h_0(\zeta)
+\log\big(\Re\big\{-\bar\zeta\hat{\psi}_{0,1}(\zeta)\big\}\big),
\]
where we use the fact that $\frakR_{0,0}=0$ according to 
Proposition~\ref{prop:exp-R}. The conclusion
now follows by observing that $h_0(\zeta)=2\Re b_0(\zeta)$.

The coefficients $\hat{\logdens}_{j,0}$ for $j\ge 1$ are given by
\[
\hat{\logdens}_{j,0}=\frakH_{j,0}-\frakR_{j,0}+\frakJ_{j,0}.
\]
The main contribution will come from the term $\frakH_{j,0}$, and 
we need to prove that the 
remainder of this term, as well as both terms $\frakR_{j,0}$ and 
$\frakH_{j,0}$ belong to the polynomial complexity class
$\POL(j-1,\Sigma_{j,1})$.
By Proposition~\ref{prop:exp-2}, it follows that
\[
\frakH_{j,l}=2\Re b_j + \frakh_{j,l}
\]
with $\frakh_{j,l}\in\POL(j-1,\Sigma_{j,0})$. Moreover, by 
Proposition~\ref{eq:exp-R} 
and Proposition~\ref{prop:exp-2}, respectively,  it follows that
\[
\frakR_{j,0}\in\POL(\Sigma_{j,0})\subset\POL(j-1,\Sigma_{j,1}),
\quad\text{and}\quad\frakJ_{j,0}
\in\POL(\Sigma_{j,1})\subset\POL(j-1,\Sigma_{j,1})
\]
and hence the claim follows. 

We next turn to the coefficients $\hat{\logdens}_{j,l}$ with 
$(j,l)\in\indsett_{2\kappa}$ for which $l\ge 1$.
The main term of
\[
\hat{\logdens}_{j,l}=\frakH_{j,l}-\frakR_{j,l}+\frakJ_{j,l}
\]
will come from the term $\frakR_{j,l}$, while the total remainder, 
consisting of the remainder from the term $\frakR_{j,l}$
together with the full terms $\frakH_{j,l}$ and $\frakJ_{j,l}$, 
is supposed to lie in the correct
polynomial complexity class $\POL(j,\Sigma)$, where
\[
\Sigma=\{(p,q)\in\Sigma_{j+1,l}\cup\Sigma_{j,l+1}:\,
(p,q)\prec_{\mathrm{L}}(j+1,l-1)\}.
\] 
By Proposition~\ref{prop:exp-2}, we have for such indices $(j,l)$ that
\[
\frakR_{j,l}=-2(4\Re R)^{\frac12}\Re(\bar{\zeta}\hat{\psi}_{j+1,l-1}) 
+ \frakr_{j,l},
\]
where $\frakr_{j,l}\in\POL(j,\Sigma)$. By Proposition~\ref{prop:exp-2}
 it also follows that
\[
\frakH_{j,l}\in\POL(j,\Sigma_{j,l})\subset \POL(j,\Sigma),
\quad\text{and}\quad\frakJ_{j,l}\in\POL(\Sigma_{j,l+1})\subset\POL(j,\Sigma).
\]
which proves the claim.
\end{proof}

\subsection{Algorithmic resolution of the master equation}
\label{ss:main-alg}
We are now ready to make the algorithm outlined in
Subsection~\ref{ss:alg-flow-over} rigorous.
We recall the master equation for the 
Taylor coefficients \eqref{eq:omega-id}
\[
\hat{\logdens}_{j,l}=\begin{cases}
-\frac12\log (4\pi)&\quad\text{for}\quad \zeta\in\T\;\;\text{and}\;\; 
(j,l)=(0,0),
\\
0&\quad\text{for}\quad \zeta\in\T\;\;\text{and}\;\; 
(j,l)\in\indsett_{2\kappa}\setminus
\{(0,0)\}.
\end{cases}
\]
In order to solve this system, we solve for the coefficient functions of $h_s$ 
and $\psi_{s,t}$ iteratively, according to the algorithm outlined in 
Subsection~\ref{ss:alg-flow-over}.

\begin{proof}[Proof of Proposition~\ref{prop:flow-general}]
In view of Propositions~\ref{prop:conformal} and \ref{prop:conf-expl}, 
the conformal 
mapping $\psi_{0,t}$ and its Taylor coefficients $\hat{\psi}_{0,l}$ for 
$l=0,1,\ldots,2\kappa+1$ with respect to the time parameter $t$ of the 
flow are 
well-defined, and they satisfy the required smoothness properties: 
for $t$ near $0$, the conformal mapping $\psi_{0,t}$ extends holomorphically 
across the boundary $\T$ to an exterior disk $\D_\e(0,\sqrt{\hrho_1})$
according to Proposition~\ref{prop:psi0t-ext}.
In addition, the derivative $\psi'_{0,t}$ remains uniformly bounded 
as long as the weight $R$ is confined to a uniform family in 
$\mathcal{W}(\hrho_0,\sigma_0)$ for fixed $\hrho_0$ and $\sigma_0$. 
Moreover, the coefficient functions $\hat{\psi}_{0,l}$ extend holomorphically to 
$\D_\e(0,\hrho_1)$, by Remark~\ref{rem:Hop-hol}.
This completes {\sc Step 1} of the algorithmic procedure.

Turning our attention to {\sc Step 2}, we recall from 
Proposition~\ref{prop:conf-expl} that on the circle $\T$, we have
$\Re(-\bar\zeta\hat{\psi}_{0,1})=(4\hDelta R(\zeta))^{-\frac12}$. Hence,
by Proposition~\ref{prop:exp-3} 
the equation $\hat{\logdens}_{0,0}=-\frac12\log(4\pi)$ is equivalent to
$$
2\Re b_0-\tfrac12\log(4\hDelta R)=-\frac12\log(4\pi) \quad\text{on }\,\T.
$$
Since we want the function $f_s$ to be zero-free in the exterior disk 
and real at infinity, this tell us that
\[
b_0=-\tfrac14\log(4\pi)+\tfrac14\Hop_{\D_\e}\big[\log (4\hDelta R)\big].
\]
We note that this automatically gives the normalization $\Im b_0(\infty)=0$.
By Proposition~\ref{prop-rho1sigma1} and Remark~\ref{rem:Hop-hol},
it follows that $b_0$ extends holomorphically 
to the exterior disk $\D_\e(0,\hrho_1)$.
Moreover, $b_0$ clearly remains uniformly bounded 
provided that $R$ is confined to a uniform family in 
$\mathcal{W}(\hrho_0,\sigma_0)$.
This completes {\sc Step 2}.

We proceed to {\sc Step 3} of the algorithmic procedure. We are now in 
the following situation. For some $j_0\ge 1$, 
our known data set is $\POL(j_0-1,\Sigma)$, where
\[
\Sigma=\{(j,l)\in\indsett_{2\kappa}: (j,l)\prec_{\mathrm{L}}(j_0,0)\}
\]
and all elements of $\POL(j_0-1,\Sigma)$ meet 
the required extension conditions. 
In particular, all the functions $b_0,\ldots,b_{j_0-1}$ and 
$\hat{\psi}_{j,l}$ for all $(j,l)\in\indsett_{2\kappa+1}$ with 
$(j,l)\prec_{\mathrm{L}} (j_0,0)$
are already known. 
In addition, the relations \eqref{eq:omega-id} are met 
for all $(j,l)\in\indsett_{2\kappa}$ with $(j,l)\prec_{\mathrm{L}} (j_0-1,1)$.
We will now show how this allows us to obtain the relations \eqref{eq:omega-id} 
for all subsequent indices $(j,l)\in\indsett_{2\kappa}$ with $(j,l)
\prec_{\mathrm{L}}(j_0,0)$, 
by making appropriate 
choices of the functions $\hat{\psi}_{j_0,l}$ for $l\ge 0$ with 
$(j_0,l)\in\indsett_{2\kappa+1}$. The additional indices for which we need 
to solve
\eqref{eq:omega-id} are those $(j,l)\in \indsett_{2\kappa}$ of the form 
$(j,l)=(j_0-1,l+1)$, where $l\ge 0$.

To achieve this, we assume that for all $l$ with $0\le l\le l_0-1$, 
we have obtained the coefficient functions
$\hat{\psi}_{j_0,l}$ by solving the equation \eqref{eq:omega-id} for the 
index pair $(j_0-1,l+1)$, 
and turn to the next equation. This reads 
$\hat{\logdens}_{j_0-1,l_0+1}=0$, as long as 
$(j_0-1,l_0+1)\in\indsett_{2\kappa}$. At this point, the known data set 
is $\POL(j_0-1,\Sigma')$, where
\[
\Sigma'=\{(j,l)\in\indsett_{2\kappa+1}: (j,l)\prec_{\mathrm{L}}(j_0,l_0)\}
\]
If $l_0$ is too large for $(j_0,l_0)\in\indsett_{2\kappa+1}$ to 
hold, we are in fact done, we don't need to obtain $\hat{\psi}_{j_0,l_0}$ 
for such indices. 
On the other hand, if $(j_0,l_0)\in\indsett_{2\kappa+1}$ we proceed as follows. 
By Proposition~\ref{prop:exp-3}, the equation $\hat{\logdens}_{j_0-1,l_0+1}=0$ 
may be written in the form
$$
\hat{\logdens}_{j_0-1,l_0+1}=-2(4\hDelta R)^{\frac12}
\Re(\bar\zeta\hat{\psi}_{j_0,l_0})+
\frakT_{j_0-1,l_0+1}=0 \quad\text{on }\,\T,
$$
where $\frakT_{j_0-1,l_0+1}\in\POL(j_0-1,\Sigma')$.
We provide a solution to this equation by the formula
\begin{equation}\label{eq:def-psi-j-l}
\hat{\psi}_{j_0,l_0}=\tfrac12\zeta\Hop_{\D_\e}
\bigg[\frac{\frakT_{j_0-1,l_0+1}}
{(4\hDelta R)^{\frac12}}\bigg].
\end{equation}
The function $\frakT_{j_0-1,l_0+1}$ has a polarization which is 
holomorphic in 
$(z,\bar w)$ for $(z,w)\in\hat{\mathbb{A}}(\hrho_1,\sigma_1)$, and the same 
holds for the weight $R$. As a consequence, it follows that 
$\hat{\psi}_{j_0,l_0}$ extends holomorphically to the exterior disk
$\D_\e(0,\hrho_1)$, and that $\hat{\psi}_{j_0,l_0}(\zeta)
=\Ordo(\lvert \zeta\rvert)$ with 
an implicit constant which is uniformly bounded, 
provided that $R$ is confined to a uniform family in 
$\mathcal{W}(\hrho_0,\sigma_0)$.

The base step $l_0=0$ of the induction procedure of {\sc Step 3} is
entirely analogous. Indeed, the known data set is $\POL(j_0-1,\Sigma)$
with $\Sigma$ as above, and by Proposition~\ref{prop:exp-3} the 
relevant equation 
$\hat{\logdens}_{j_0-1,1}=0$ takes the form
\[
-2(4\hDelta R)^{\frac12}\Re(\bar\zeta\hat{\psi}_{j_0,0})+\frakT_{j_0-1,1}=0 
\quad \text{on }\,\T,
\] 
where $\frakT_{j_0-1,1}$ in particular lies in $\POL(j_0-1,\Sigma)$. 
Hence we obtain $\hat{\psi}_{j_0,0}$ by the formula \eqref{eq:def-psi-j-l}
with $l=0$ replaced by $0$. 

We now turn to {\sc Step 4}. After the completion of {\sc Step 3}, 
the situation is as follows. the known data set is $\POL(j_0-1,\Sigma)$
with
\[
\Sigma=\{(j,l):(j,l)\prec_{\mathrm{L}}(j_0+1,0)\},
\]
where every element of $\POL(j_0-1,\Sigma)$ polarizes to 
$\hat{\mathbb{A}}(\hrho_1,\sigma_1)$. 
In addition, the relations \eqref{eq:omega-id} are met for all 
$(j,l)\in\indsett_{2\kappa}$ with $(j,l)\prec_{\mathrm{L}}(j_0,0)$.
We recall that this in particular this means that the known data set 
consists of the 
coefficient functions $b_0,\ldots,b_{j_0-1}$ and
$\hat{\psi}_{j,l}$ for $(j,l)\in\indsett_{2\kappa+1}$ with 
$(j,l)\prec_{\mathrm{L}} (j_0+1,0)$ are known.
In this step, we need to find the function $b_{j_0}$, and verify that the 
relation \eqref{eq:omega-id} is then met with $(j,l)=(j_0,0)$.
To this end, we apply Proposition~\ref{prop:exp-3}, and observe that 
the equation \eqref{eq:omega-id} with $(j,l)=(j_0,0)$ is equivalent to having
$$
\hat{\logdens}_{j_0,0}=2\Re b_{j_0}
+\frakT_{j_0,0}=0\quad\text{on }\T,
$$ 
where $\frakT_{j_0,0}\in \POL(j_0-1,\Sigma)$, with the same $\Sigma$ as above. 
Since $\POL(j_0-1,\Sigma)$ is a collection of known functions, we hence 
obtain an equation for the unknown function $b_{j_0}$, with solution
$$
b_{j_0}=-\tfrac12 \Hop_{\D_\e}\big[\frakT_{j_0,0}\big].
$$
In view of Proposition~\ref{prop-rho1sigma1} and Remark~\ref{rem:Hop-hol},
the function $b_{j_0}$ extends holomorphically to the exterior disk 
$\D_\e(0,\hrho_1)$, and remains uniformly bounded if the weight $R$ is 
confined to a uniform family in $\mathcal{W}(\hrho_0,\sigma_0)$. Moreover, 
we observe that $b_{j_0}$ has the required normalization at infinity: 
$\Im b_{j_0}(\infty)=0$.

We finally turn to {\sc Step 5}. The key observation is that we are now in a 
position to return to {\sc Step 3} followed by {\sc Step 4}, 
with $j_0$ replaced by $j_0+1$. 
Since {\sc Step 1} and {\sc Step 2} combine to form
the initial data for {\sc Steps 3 and 4} with $j_0=1$, the algorithm 
produces iteratively the entire set of coefficient functions, and solves 
in the process all the equations 
\eqref{eq:omega-id} for $(j,l)\in\indsett_{2\kappa}$.

Equipped with the functions $b_j$ for $j=0,\ldots,\kappa$, the 
conformal mappings
$\psi_{0,t}$ and the coefficients $\hat{\psi}_{j,l}$ for 
$(j,l)\in\indsett_{2\kappa+1}\cap\{(j,l):j\ge 1\}$, we observe 
that the functions 
$h_s$ and $\psi_{s,t}$ given by
$$
h_s(\zeta)=\sum_{j=0}^{\kappa}s^jb_j(\zeta),\qquad \psi_{s,t}(\zeta)=
\psi_{0,t}+\sum_{\substack{(j,l)\in\indsett_{2\kappa+1}\\ j\ge 1}}s^jt^l
\hat{\psi}_{j,l}(\zeta)
$$
are well-defined, and have the desired smoothness and mapping properties.
By the Becker-Pommerenke criterion of Lemma~\ref{lem:becker}, it is 
immediate that
$\psi_{s,t}$, as defined, is univalent in a neighborhood of the closed 
exterior disk 
$\overline{\D}_\e$ for $s$ and $t$ close to $0$.
As $\psi_{s,t}$ extends holomorphically
to the exterior disk $\D_\e(0,\sqrt{\hrho}_1)$, and since $\psi_{s,t}$ is a 
smooth perturbation
of the identity it follows that $\psi_{s,t}$ is univalent on $\D_\e(0,\hrho_2)$ 
and that
\[
\psi_{s,t}(\D_\e(0,\hrho_2))\subset\D_\e(0,\hrho_1)
\]
for $s,t$ close to $0$, provided that $\hrho_2$ is chosen 
appropriately with $\sqrt{\hrho_1}\le \hrho_2<1$.

The conclusion of Proposition~\ref{prop:flow-general} is now an immediate 
consequence of the relations \eqref{eq:omega-id} for the Taylor coefficients 
of the logarithm of the function
$$
\exp(\logdens_{s,t}(\zeta))=\lvert f_s\circ\psi_{s,t}(\zeta)\rvert^2\e^{-2s^{-1}
\{(R\circ\psi_{s,t})(\zeta)-\tfrac{1}{2}t^2\}}
\Re\big(-\bar\zeta\partial_t\psi_{s,t}(\zeta)
\overline{\psi_{s,t}'(\zeta)}\big),\qquad \zeta\in\T
$$
in the variables $(s,t)$ near $(0,0)$, as verified in the above algorithm. 
\end{proof}

\subsection{Implementation of the orthogonal foliation flow for 
\texorpdfstring{$R=R_\tau$}{R=R-tau}}
\label{ss:consequence-flow}
It remains only to prove the key lemma (Lemma~\ref{lem:main-flow}). 
The hard work was completed in the previous subsection. The existence of the 
orthogonal foliation flow now follows if we use $s=1/m$ as our quantization
parameter.

\begin{proof}[Proof of Lemma~\ref{lem:main-flow}]
We first claim that $Q\circ\phi_\tau^{-1}$ is uniformly real-analytic in
the exterior disk $\D_{\e}(0,\rho_{0,0})$ for $\tau\in I_{\epsilon_0}$.
By this we mean that there exists a number $\sigma_{0,0}>0$ such that
$Q\circ\phi_\tau$ has a polarization which is uniformly bounded 
on the 2$\sigma_{0,0}$-fattened diagonal annulus 
$\hat{\mathbb{A}}(\rho_{0,0},\sigma_{0,0})$ (see Definition~\ref{def:classR}).
Let $\rho_{1,0}$ be the number given by 
$\rho_{1,0}:=\max\big\{\rho_{0,0},((1+\sigma_{0,0}^2)^{\frac12}
+\sigma_{0,0})^{-1}\big\}$.
Moreover, the function $Q^\circledast_\tau\circ\phi_\tau^{-1}$, 
which is the harmonic extension of $Q\circ\phi_\tau^{-1}\vert_{\T}$ 
to the exterior disk $\D_\e$, is uniformly bounded on 
$\hat{\mathbb{A}}(\rho_{1,0},\sigma_{0,0})$
in view of Proposition~\ref{prop:f-T-hol} and an elementary decomposition
of harmonic functions holomorphic and conjugate holomorphic functions.
In view of \eqref{eq:Q-breve-circledast}, the same holds for 
$\breve{Q}_\tau\circ\phi_\tau^{-1}$ and consequently also for 
$R_\tau=(Q-\breve{Q}_\tau)\circ\phi_\tau^{-1}$.
In view of the uniform flatness of $R_\tau$ 
near the unit circle, the function
$R_{\tau,0}$ defined by the relation 
$R_\tau(\zeta)=(1-|\zeta|^2)^2 R_{\tau,0}(\zeta)$
enjoys the same property as well, namely that its polarization is uniformly
bounded on $\hat{\mathbb{A}}(\rho_{1,0},\sigma_{0,0})$.
By possibly replacing $\sigma_{0,0}$ by a smaller positive number $\sigma_{1,0}$
we may guarantee that the polarization of 
$R_{\tau,0}$ remains bounded away from zero
in the slightly smaller fattened diagonal annulus 
$\hat{\mathbb{A}}(\rho_{1,0},\sigma_{1,0})$.
If necessary, we replace $\rho_{1,0}$ by the larger number 
$\rho_{2,0}=\max\{\rho_{1,0},((1+\sigma_{1,0}^2)^{\frac12}
+\sigma_{1,0})^{-1}\}$,
which is still smaller than $1$.

In view of the above considerations and the uniform bounds from 
Proposition~\ref{prop:R}, the family $R_\tau$ for $\tau\in I_{\epsilon_0}$
constitute a uniform family $\mathcal{W}(\hrho_0,\sigma_0)$ 
where $\hrho_0=\rho_{2,0}$ and $\sigma_0=\sigma_{1,0}$. 
By Proposition~\ref{prop-rho1sigma1} we obtain 
numbers $\hrho_1$ and $\sigma_1$ with $0<\hrho_1<1$ and $\sigma_1>0$. We set
$\rho_0=\hrho_1$ and apply Proposition~\ref{prop:flow-general} to obtain the
desired conformal mappings $\psi_{m,n,t}=\psi_{s,t}$ and 
$f_{m,n}^{\langle \kappa\rangle}=f_s$
with associated asymptotic expansion to precision $\kappa$, 
where $s=m^{-1}$. Here, the function $f_{m,n}^{\langle \kappa\rangle}$
is holomorphic and bounded on $\D_\e(0,\rho_0)$, positive at infinity
and bounded away from zero in the entire exterior disk $\D_\e(0,\rho_0)$.
Moreover, the flow equation \eqref{eq:main-flow} 
of Lemma~\ref{lem:main-flow} holds to the desired accuracy, in view of
Proposition~\ref{prop:flow-general} with $s=m^{-1}$.
\end{proof}

\section{Connection with the matrix 
\texorpdfstring{$\bar\partial$}{d-bar}-problem
of Its and Takhtajan}
\label{s:RHP}
\subsection{Matrix
\texorpdfstring{$\bar\partial$}{d-bar}-problems
and orthogonal polynomials}
Given the successful application of Riemann-Hilbert problem methods to 
the study of orthogonal polynomials in the context of the real line and 
the unit circle, it has been proposed that the planar orthogonal 
polynomials should be approached in a similar fashion. Following 
Its and Takhtajan \cite{Its}, we consider a matrix $\bar\partial$-problem 
(or a {\em soft Riemann-Hilbert problem})
and see how it fits in with our orthogonal foliation flow.

A polynomial is said to be \emph{monic} if it has leading coefficient
equal to $1$. So, let $\pi_{m,n}$ denote the monic orthogonal 
polynomial of degree $n$ with respect to the measure 
$\e^{-2mQ}\diffA$ where $Q$ is assumed $1$-admissible. In other words,
$\pi_{m,n}$ is given by
\[
\pi_{m,n}(z)=\kappa^{-1}_{m,n}P_{m,n}(z),
\]
where $\kappa_{m,n}$ is the leading coefficient of the normalized orthogonal
polynomial $P_{m,n}$.

If $f \in L^p(\C)$ for some $1<p<2$, we let $\mathbf{C}f$ be its 
Cauchy transform, given by
\[
\mathbf{C}f(z)=\int_{\C}\frac{f(w)}{z-w}\diffA(w)
\]
which is well-defined almost everywhere and represents a function which is 
locally in the Sobolev space $W^{1,p}$.
The importance of the Cauchy transform comes from the fact that 
in the sense of distribution theory, $\bar\partial\mathbf{C}f=f$.

In \cite{Its}, Its and Takhtajan propose to study the asymptotics of 
$\pi_{m,n}$ starting
from the observation that the matrix-valued function
\begin{equation}\label{eq:Y-matrix}
Y_{m,n}(z)=
\begin{pmatrix}
\pi_{m,n}(z)& -\mathbf{C}\big[\bar{\pi}_{m,n}\e^{-2mQ}\big](z)\\
-\kappa_{m,n-1}^{2}\pi_{m,n-1}(z)&
\kappa_{m,n-1}^{2}\mathbf{C}\big[\bar{\pi}_{m,n-1}
\e^{-2mQ}\big](z),
\end{pmatrix}
\end{equation}
solves the $\bar\partial$-problem
\begin{equation}\label{eq:d-bar}
\begin{cases}
\bar\partial Y(z)=-\bar{Y}(z)W(z),&\text{ for }z\in\C,\\
Y(z)=\big(I+\Ordo\big(z^{-1}\big)\big)
\begin{pmatrix}
z^n&0\\0&z^{-n}
\end{pmatrix},&\text{ as }\lvert z\rvert\to+\infty,
\end{cases}
\end{equation}
where $W(z)=W_{m}(z)$ is the matrix-valued function
$$
W(z)=
\begin{pmatrix}
0&\e^{-2mQ(z)}\\
0&0
\end{pmatrix}.
$$
Moreover, the solution is unique, as shown in \cite{Its}. 
We remark that classical Riemann-Hilbert problems where a jump occurs 
on a curve $\Gamma$ may be phrased as $\bar\partial$-problems
where $\bar\partial Y(z)$ is understood as a matrix-valued measure supported 
$\Gamma$, and the above problem is a natural generalization to a more 
genuinely two-dimensional situation.

The idea that underlies the Its-Takhtajan approach, as well as the 
classical Riemann-Hilbert approach to orthogonal polynomials, is the 
expectation that one may constructively obtain an approximate
solution $\tilde Y=\tilde{Y}_{m,n}(z)$ to the problem \eqref{eq:Y-matrix} (or 
the corresponding RHP), which should then  
produce an entry $(\tilde{Y}_{m,n})_{1,1}$ which is approximately equal 
to $\pi_{m,n}(z)$.

\subsection{Integration of Riemann-Hilbert problems along
curve families}
Unfortunately, it has proven difficult to solve the problem 
\eqref{eq:Y-matrix} constructively.
The following simple observation shows how our orthogonal foliation flow
reduces the $\bar\partial$-problem to a family of more classical
Riemann-Hilbert problems along closed curves.

In order to describe this problem, we denote by $J(z)$ a $2\times 2$ 
jump matrix 
and let $\Gamma$ be an oriented smooth simple closed curve in $\C$.
We denote by $\Omega^+$ and $\Omega^-$ the interior and exterior components of 
the complement $\C\setminus\Gamma$, respectively.
If $f$ is a function defined on $\C\setminus\Gamma$, which is 
continuous up to the boundary $\Gamma$ as seen from each component,
we define the two boundary value functions $f^+$ and $f^{-}$ on $\Gamma$ by
$$
f^\pm(\zeta)=\lim_{\substack{z\to\zeta\\ z\in \Omega^\pm}}f(z),
\qquad \zeta\in\Gamma.
$$
We consider the Riemann-Hilbert problem of finding a $2\times2$ matrix-valued 
function $Y(z)$ which meets
\begin{equation}
\label{eq:RHP}
\begin{cases}
Y \text{ is holomorphic on } \C\setminus\Gamma,& \\
Y^+(z)=Y^-(z)+\bar{Y}^-(z)J(z),&\text{ for } z\in \Gamma,\\
Y(z)=(I+\Ordo(z^{-1}))\begin{pmatrix}
z^n&0\\0&z^{-n}\end{pmatrix},&\text{ as }|z|\to+\infty.
\end{cases}
\end{equation}
In order to analyze this problem, we need a variant of the Cauchy transform,
which applies to functions defined on $\Gamma$. 
For smooth $\Gamma$ and reasonable $f$, we write
\[
{\bf C}_\Gamma f(z)=\frac{1}{2\pi \imag}\int_{\Gamma}\frac{f(w)}{w-z}\diff w,
\qquad z\in\C\setminus\Gamma.
\]
As is well-known, the classical Plemelj formula is a useful tool in the 
study of Riemann-Hilbert problems:
\begin{equation}
\label{eq:Plemelj}
({\bf C}_\Gamma f)^+(z)=({\bf C}_\Gamma f)^{-}(z)+f(z).
\end{equation}
We now connect the more classical Riemann-Hilbert problem \eqref{eq:RHP} 
with the matrix $\bar\partial$-problem \eqref{eq:d-bar}. 

\begin{prop}
\label{prop:RHP-correspondence}
Let $\{\Gamma_t\}_{t\in I}$ be a smooth strictly expanding flow of 
positively oriented simple closed curves, and denote by $\calD$ the union
$\calD=\bigcup_{t\in I}\Gamma_t$. Let $\omega(z)$ denote a smooth positive 
function on $\calD$, and denote by $\xi:\calD\to\C$ the vector field 
$\nu\bar{\eta}$, where $\eta(z)$ denotes the outward unit normal field to 
the curve family and $\nu$ denotes the scalar normal velocity of the flow. 
Then, for each $t\in I$, there is a unique solution $Y_t(z)$ to the 
Riemann-Hilbert problem \eqref{eq:RHP} with jump matrix
$$
J=\begin{pmatrix} 0 & 2\omega \xi
\\ 0 & 0\end{pmatrix}.
$$
Moreover, if there exists a continuous positive function $\lambda(t)$ 
such that $(Y_t)_{1,1}$ and $\lambda(t)(Y_t)_{2,1}$ are independent of $t$, 
then the matrix-valued function
\[
Y(z)=\Lambda_1^{-1}
\int_{I}\Lambda(t) Y_t(z)\diff t\,
\Lambda_2^{-1}
\]
is the unique solution to \eqref{eq:d-bar}, with 
$W=\begin{pmatrix}0& 
1_{\calD}\,\omega\\ 0&0\end{pmatrix}$, provided that 
\[
\Lambda(t)=
\begin{pmatrix}1& 
0\\ 0&\lambda(t)\end{pmatrix},\quad
\Lambda_1=\begin{pmatrix}1& 
0\\ 0&\int_I\lambda(t)\diff t\end{pmatrix},\quad
\Lambda_2=\begin{pmatrix}|I|& 
0\\ 0&1\end{pmatrix}.
\]
\end{prop}

\begin{proof}
We first establish the existence of solutions to the problem 
\eqref{eq:RHP} of $\Gamma_t$, which may be expressed in terms of
a family of $t$-dependent orthogonal polynomials.
We recall that $\xi$ factors as $\nu\bar{\eta}$, where $\nu$ denotes the 
speed of the boundary in the normal direction while $\eta$ denotes the 
outward pointing unit normal field. Since arc-length measure
$\lvert \diff z\rvert$ on $\Gamma_t$ relates to the complex line element 
$\diff z$ by $\diff z=\tau\lvert\diff z\rvert$ where $\tau$ denotes the unit 
tangent vector field along $\Gamma_t$, it follows that
\begin{equation}\label{eq:differentials}
\frac{1}{2\pi\imag}\diff z=
\frac{1}{2\pi}(-\imag\tau)\rvert\diff z\rvert=\eta\,\diffs 
\end{equation}
where we recall the convention $\diffs=\frac{\lvert \diff z\rvert}{2\pi}$.
From this it follows that $(2\pi\imag)^{-1}\xi\,\diff z=\nu\,\diffs$, 
and we may consequently define an inner product by
$$
\langle f,g\rangle_t:=\int_{\Gamma_t}f(z)\bar{g}(z)\nu(z)\diffs(z)=
\frac{1}{2\pi\imag}\int_{\Gamma_t}f(z)\bar{g}(z)\xi(z)\diff z.
$$
Let $\{{\pi}^\star_{n,t}\}_n$ denote the sequence of monic 
orthogonal polynomials
with respect to this inner product, such that $\pi^\star_{n,t}$ has degree $n$,
and denote by ${\kappa}^\star_{n,t}$ the sequence of leading coefficient of the 
corresponding normalized orthogonal polynomials 
${P}^\star_{n,t}={\kappa}^\star_{n,t}{\pi}^\star_{n,t}$.
It is straightforward to check that the function
\[
\begin{pmatrix}
{\pi}^\star_{n,t}&2{\bf C}_{\Gamma_t}[\bar{{\pi}}^\star_{n,t}\omega\xi]\\
-\frac12({\kappa}^\star_{n-1,t})^{2}{\pi}^\star_{n-1,t}& 
-({\kappa}_{n-1,t}^\star)^{2}
{\bf C}_{\Gamma_t}[\bar{{\pi}}^\star_{n-1,t}\omega\xi](z)
\end{pmatrix}
\]
supplies a solution to the Riemann-Hilbert problem \eqref{eq:RHP}.

As for the uniqueness, it is clear from Plemelj's formula \eqref{eq:Plemelj}
and the jump condition that any solution $Y_t(z)$ must take the form
\[
Y_t(z)=
\begin{pmatrix}
a_t(z)&u_t(z) + 2{\bf C}_{\Gamma_t}[\bar{a}_t\omega\xi](z)\\
b_t(z) & v_t(z) + 2{\bf C}_{\Gamma_t}[\bar{b}_t\omega\xi](z)\\
\end{pmatrix},
\]
where $a_t,b_t,u_t,v_t$ are entire functions. From the growth constraint 
at infinity,  we see that these four functions are all polynomials. 
Moreover, $u_t=v_t=0$ for the same reason. A standard expansion of the
Cauchy kernel at infinity shows that 
$a_t$ is a monic polynomial of degree $n$ which is orthogonal to the lower
degree polynomials $\mathrm{Pol}_n$ with respect to $\omega\xi\diff z$ on 
$\Gamma_t$. It follows that $a_t={\pi}^\star_{n,t}$.
Analogously, $b_t$ is given by 
$b_t=-\frac12({\kappa}_{n-1,t}^\star)^2{\pi}^\star_{n-1,t}$. 
We have established the unique solvability of the Riemann-Hilbert problem
\eqref{eq:RHP} with the given jump matrix.

Next, we turn to the connection with the $\bar\partial$-problem 
\eqref{eq:d-bar}.
Under the assumption that $(Y_t)_{1,1}=a_t=A$ is independent of $t$,
and that for some $t$-dependent parameter $\lambda(t)$, the expression
$\lambda(t)(Y_t)_{2,1}=\lambda(t)b_t=B$ is also 
independent of $t$, we may consequently write
$$
\Lambda(t)\,Y_t(z)=
\begin{pmatrix}
A(z)&2{\bf C}_{\Gamma_t}[\bar{A}\omega\xi](z)\\
B(z)&2{\bf C}_{\Gamma_t}[\bar{B}\omega\xi](z)
\end{pmatrix},
$$
where we recall that $\Lambda(t)$ is the matrix given in the
proposition.
Recall that we may integrate over the flow $\{\Gamma_t\}_t$ using the 
disintegration
$$
\int_{t\in I}\bigg\{2\int_{\Gamma_t}u(z)\nu(z)\diffs\bigg\}\diff t
=\int_{\calD}u(z)\diffA(z),
$$
for functions $u$ such that the indicated integrals have a well-defined
meaning. It now follows that if 
$\langle\lambda\rangle_I=\int_I\lambda(t)\diff t$, the matrix-valued function 
\begin{multline*}
\hat Y(z):=\Lambda_1^{-1}
\int_{I}\Lambda(t) Y_t(z)\diff t\,
\Lambda_2^{-1}=\Lambda_1^{-1}
\begin{pmatrix}
|I|\,A(z)&-{\bf C}[\bar{A}\omega\, 1_{\calD}](z)\\
|I|\,B(z)&-{\bf C}[\bar{B}\omega\, 1_\calD](z)
\end{pmatrix}
\Lambda_2^{-1}
\\
=
\begin{pmatrix}
A(z)&-{\bf C}[\bar{A}\omega\, 1_{\calD}](z)\\
(\langle\lambda\rangle_I)^{-1}\,B(z)
&-(\langle\lambda\rangle_I)^{-1}\,{\bf C}[\bar{B}\omega\, 1_\calD](z)
\end{pmatrix}
\end{multline*}
solves
\[
\bar\partial \hat Y(z)=
\begin{pmatrix}
0&-\bar{A}\omega\, 1_{\calD}\\
0
&-(\langle\lambda\rangle_I)^{-1}\bar{B}\omega\, 1_\calD
\end{pmatrix}
=-\overline{\hat Y(z)}\begin{pmatrix}
0&\omega\, 1_{\calD}\\
0&0
\end{pmatrix}
\]
with asymptotics
\[
\hat Y(z)=\big(I+\Ordo\big(z^{-1}\big)\big)
\begin{pmatrix}
z^n&0\\0&z^{-n}
\end{pmatrix},
\text{ as }\lvert z\rvert\to+\infty,
\]
as a consequence of the corresponding asymptotics of $Y_t$ for each $t\in I$.
\end{proof}

\begin{rem}
{\rm (a)}\, 
For the orthogonal foliation flow, in the context of a neighborhood of the 
boundary curve of the droplet $\calS_\tau$ with $\tau=\frac{n}{m}$, the
(approximate) orthogonal polynomial of degree $n$ is also
approximately orthogonal to the lower degree polynomials along the
individual flow loops corresponding to $\omega=\e^{-2mQ}$. So, in view
of Proposition \ref{prop:RHP-correspondence}, the conditions 
\begin{equation}\label{eq:compat-flow}
\partial_t(Y_t)_{1,1}=0,\quad \partial_t\big(\lambda(t)(Y_{t})_{2,1}\big)=0
\end{equation}
should be met at least approximately for some appropriate scalar-valued
function $\lambda(t)$ (cf. the presentation in
Subsection~\ref{ss:idea-sketch}). 
Alternatively, we could use \eqref{eq:compat-flow} 
as a criterion to define a flow of curves. In the given setting, this should 
give us back our orthogonal foliation flow. 
In other words, \eqref{eq:compat-flow} should be analogous to 
the condition \eqref{eq:omega-id}, once the Riemann-Hilbert problems of 
Proposition~\ref{prop:RHP-correspondence} are approximately
solved in a constructive fashion, and we would expect that in an approximate 
sense,
\[
\Gamma_t\sim\phi_\tau^{-1}(\psi_{m,n,-t}(\T)).
\]
It is entirely possible that the conditions \eqref{eq:compat-flow} would 
be more stable close to the zeros of the orthogonal polynomial $\pi_{m,n}$.
For instance, this might be the case with a highly eccentric ellipse. 

\noindent {\rm (b)}\, In their work, Its and Takhtajan use a bounded 
domain $\Omega$ to address possible
convergence issues. Here, the potential $Q$ grows sufficiently rapidly, 
so there is no need for us to consider such a truncation.
\end{rem}

\end{document}